\documentclass{amsart}
\usepackage{graphicx,verbatim, amsmath, amssymb, amsthm, amsfonts, epsfig, ifthen,mathtools,epstopdf,caption,enumerate,hyperref,commath}	
\DeclareFontFamily{U}{mathx}{\hyphenchar\font45}
\DeclareFontShape{U}{mathx}{m}{n}{<-> mathx10}{}
\DeclareSymbolFont{mathx}{U}{mathx}{m}{n}
\DeclareMathAccent{\widebar}{0}{mathx}{"73}

\newtheorem{proposition}{Proposition}[section]
\newtheorem{theorem}[proposition]{Theorem}

\newtheorem{lemma}[proposition]{Lemma}
\newtheorem{prop}[proposition]{Proposition}
\newtheorem{cor}[proposition]{Corollary}

\newtheorem{phen}{Phenomenon}

\theoremstyle{definition}
\newtheorem{example}[proposition]{Example}
\newtheorem{definition}[proposition]{Definition}

\theoremstyle{remark}
\newtheorem{remark}[proposition]{Remark}

\numberwithin{equation}{section}
\usepackage[usenames]{color}

\usepackage{color}

\newcommand{\margincolor}{red}      
\definecolor{darkgreen}{rgb}{0,0.7,0}

\newcounter{margincounter}
\newcommand{\marginnum}{
\ifnum\value{margincounter}<10
\textcolor{\margincolor}{\begin{picture}(0,0)\put(2.2,2.4){\circle{9}}\end{picture}\footnotesize\arabic{margincounter}}
\else\ifnum\value{margincounter}<100
\textcolor{\margincolor}{\begin{picture}(0,0)\put(4.256,2.5){\circle{11}}\end{picture}\footnotesize\arabic{margincounter}}
\else
\textcolor{\margincolor}{\begin{picture}(0,0)\put(6.8,2.5){\circle{14}}\end{picture}\footnotesize\arabic{margincounter}}
\fi\fi
}

\newcommand{\newword}[1]{\textbf{\emph{#1}}}

\newcommand{\integers}{\mathbb Z}
\newcommand{\rationals}{\mathbb Q}

\newcommand{\reals}{\mathbb R}

\newcommand{\thet}{\vartheta}

\newcommand{\id}{\operatorname{id}}

\newcommand{\sgn}{\operatorname{sgn}}
\newcommand{\vsgn}{\mathbf{sgn}}

\newcommand{\Geom}{{\operatorname{\textbf{Geom}}}}

\newcommand{\A}{{\mathcal A}}

\newcommand{\F}{{\mathcal F}}
\newcommand{\D}{{\mathcal D}}

\newcommand{\W}{{\mathcal W}}

\newcommand{\ck}{^\vee}

\newcommand{\toname}[1]{\stackrel{#1}{\longrightarrow}}

\renewcommand{\th}{^\text{th}}

\newcommand{\0}{{\mathbf{0}}}

\newcommand{\Proj}{\mathrm{Proj}}

\newcommand{\diag}{\mathrm{diag}}

\DeclareMathOperator{\Span}{Span}
\DeclareMathOperator{\supp}{supp}

\newcommand{\g}{\mathbf{g}}

\renewcommand{\c}{\mathbf{c}}
\renewcommand{\b}{\mathbf{b}}
\renewcommand{\k}{\mathbf{k}}
\renewcommand{\a}{\mathbf{a}}
\newcommand{\e}{\mathbf{e}}
\newcommand{\x}{\mathbf{x}}
\newcommand{\y}{\mathbf{y}}
\newcommand{\z}{\mathbf{z}}

\renewcommand{\v}{\mathbf{v}}
\newcommand{\n}{\mathbf{n}}

\newcommand{\w}{\mathbf{w}}
\newcommand{\tB}{\tilde{B}}

\newcommand{\M}{\mathcal{M}}

\newcommand{\V}{\mathcal{V}}

\newcommand{\Rel}{\operatorname{Rel}}

\renewcommand{\S}{\mathbf{S}}
\renewcommand{\M}{\mathbf{M}}
\newcommand{\Var}{\operatorname{Var}}

\newcommand{\Fan}{\operatorname{Fan}}
\newcommand{\ScatFan}{\operatorname{ScatFan}}

\DeclareMathOperator{\Cart}{Cart}
\renewcommand{\d}{{\mathfrak d}}
\renewcommand{\D}{{\mathfrak D}}
\newcommand{\Ram}{{\operatorname{Ram}}}

\newcommand{\fakesubsec}[1]{\medskip\noindent\textbf{#1.}}  

\newcommand{\B}[2]{\begin{bsmallmatrix*}[r]0&#1\\#2&0\end{bsmallmatrix*}}

\newcommand{\hy}{\hat{y}}

\title[Dominance phenomena]{Dominance phenomena: \\ mutation, scattering and cluster algebras}
\author{Nathan Reading}
\thanks{This paper is based upon work partially supported by the National Science Foundation under Grant Numbers DMS-1101568 and DMS-1500949.}

\begin{document}

\begin{abstract}
An exchange matrix $B$ dominates an exchange matrix $B'$ if the signs of corresponding entries weakly agree, with the entry of $B$ always having weakly greater absolute value.
When $B$ dominates $B'$, interesting things happen in many cases (but not always):
the identity map between the associated mutation-linear structures is often mutation-linear;
the mutation fan for $B$ often refines the mutation fan for $B'$;
the scattering (diagram) fan for $B$ often refines the scattering fan for $B'$;
and there is often an injective homomorphism from the principal-coefficients cluster algebra for $B'$ to the principal-coefficients cluster algebra for $B$, preserving $\g$-vectors and sending the set of cluster variables for $B'$ (or an analogous larger set) into the set of cluster variables for $B$ (or an analogous larger set).
The scope of the description ``often'' is not the same in all four contexts and is not settled in any of them.
In this paper, we prove theorems that provide examples of these dominance phenomena.
\end{abstract}

\maketitle

\setcounter{tocdepth}{2}
\tableofcontents

\section{Introduction}\label{intro sec}  An \newword{exchange matrix} is a skew-symmetrizable integer matrix.
Given $n\times n$ exchange matrices $B=[b_{ij}]$ and $B'=[b'_{ij}]$, we say $B$ \newword{dominates} $B'$ if for each $i$ and~$j$, we have $b_{ij}b'_{ij}\ge0$ and $|b_{ij}|\ge|b'_{ij}|$.
This paper explores the consequences of the dominance relationship between $B$ and $B'$ for the mutation-linear algebra, mutation fans, scattering diagrams, and principal-coefficients cluster algebras associated to $B$ and $B'$.
We present our results by describing several phenomena that often occur when $B$ dominates $B'$.
We call these ``phenomena'' because the hypotheses are not yet nailed down and because there are negative examples (where a phenomenon does not occur). 
The goal of this paper is to prove theorems that give compelling and surprising examples of the phenomena.

By providing examples of the phenomena, we wish to establish that something real and nontrivial is happening, with an eye towards two potential benefits:
In one direction, we anticipate that researchers from the various areas will apply their tools to find additional examples of the phenomena, necessary and/or sufficient conditions for the phenomena to occur, and/or additional dominance phenomena.
In the other direction, since the phenomena discussed here concern fundamental aspects of matrix mutation, the geometry of scattering diagrams, and the commutative algebra of cluster algebras, we anticipate that the phenomena will lead to insights in the various areas where matrix mutation, scattering diagrams, and cluster algebras are fundamental.

We now describe the phenomena and state the main results.

\subsection{Mutation-linear maps}\label{mulin intro}
The study of mutation-linear algebra was initiated in \cite{universal} and continued in \cite{unisurface,unitorus,unisphere}.
The ``mutation'' in ``mutation-linear'' is matrix mutation in the sense of cluster algebras.  
(See \cite[Definition~4.2]{ca1} or Section~\ref{mla sec}.)
To put mutation-linear algebra into context, consider the formulation of (ordinary) linear algebra as the study of linear relations on a module.
For example, the notion of basis is defined in terms of the existence and non-existence of certain linear relations.
As another example, a map $\lambda:V\to V'$ is linear if for every linear relation $\sum_{i\in S}c_i\v_i$ in $V$, the sum $\sum_{i\in S}c_i\lambda(\v_i)$ is a linear relation in $V'$.

In the same sense, mutation-linear algebra is the study of $B$-coherent linear relations.
A \newword{$B$-coherent linear relation}\footnote{Here in the introduction, we sweep a technicality under the rug.  See Definition~\ref{B coher def}.} among vectors in $\reals^n$ is a linear relation $\sum_{i\in S}c_i\v_i$ with the following property:
If the $\v_i$ are placed as coefficient rows under $B$ and the resulting extended exchange matrix is subjected to an arbitrary sequence of mutations to produce new coefficient rows $\v'_i$, then the relation $\sum_{i\in S}c_i\v'_i=\0$ holds.
The notation $\reals^B$ is shorthand for the \newword{mutation-linear structure} on $\reals^n$, meaning the set $\reals^n$ together with the collection of $B$-coherent linear relations.
Mutation-linear algebra is closely tied to cluster algebras.
For example, finding a \newword{basis} for $\reals^B$ is the same thing as finding a cluster algebra for $B$ with universal (geometric) coefficients \cite[Theorem~4.4]{universal}.
The key mutation-linear notion in this paper is the notion of a mutation-linear map.
A map $\lambda$ from $\reals^B$ to $\reals^{B'}$ is \newword{mutation-linear} if for every $B$-coherent linear relation $\sum_{i\in S}c_i\v_i$, the sum $\sum_{i\in S}c_i\lambda(\v_i)$ is a $B'$-coherent linear relation.
A mutation-linear map defines a functor from the category of geometric cluster algebras with exchange matrix $B$, under coefficient specialization, to the same category for $B'$.
(See Section~\ref{functor sec}.)

From the definition of a mutation-linear map, it is not clear whether interesting mutation-linear maps exist.
The first phenomenon points out that they often do. 

\begin{phen}[Identity is mutation-linear]\label{id phen}
Suppose $B$ and $B'$ are exchange matrices such that $B$ dominates~$B'$.  
In many cases, the identity map from $\reals^B$ to $\reals^{B'}$ is mutation-linear.
\end{phen}

%
There are non-examples of Phenomenon~\ref{id phen}; the surprise is that the phenomenon occurs as much as it does.
One simple non-example is when $B$ comes from a quiver with three vertices and three arrows, forming a directed cycle, and $B'$ is defined by deleting one arrow from the cycle.
On the positive side, we verify the phenomenon for $B$ acyclic of finite type.
We prove the following theorem using results of \cite{diagram}.

\begin{theorem}\label{id phen finite}  
Suppose $B$ is an acyclic exchange matrix of finite type, and suppose $B'$ is another exchange matrix dominated by $B$.
Then the identity map from $\reals^B$ to $\reals^{B'}$ is mutation-linear.
\end{theorem}

Replacing $\reals$ with $\rationals$ in Phenomenon~\ref{id phen}, we obtain a weaker phenomenon, which we will call the ``rational version'' of Phenomenon~\ref{id phen}, exemplified by Theorem~\ref{id phen resect}, below.
We define a simple operation, called resection, on triangulated surfaces that produces a dominance relation among signed adjacency matrices.
Resection and the Null Tangle Property are defined in Section~\ref{bij surf sec}.
A \newword{null surface} is a once-punctured monogon, once-punctured digon, or unpunctured quadrilateral.

\begin{theorem}\label{id phen resect}
Given a marked surface $(\S,\M)$ with a triangulation $T$, perform a resection of $(\S,\M)$ compatible with $T$ and let $T'$ be the triangulation induced by $T$ on the resected surface $(\S',\M')$.
If every component of $(\S,\M)$ either has the Null Tangle Property or is a null surface and if $(\S',\M')$ has the Curve Separation Property, then the identity map from $\rationals^{B(T)}$ to $\rationals^{B(T')}$ is mutation-linear.
\end{theorem}

The only surfaces currently known to have the Null Tangle Property (or defined to be null surfaces) are those whose connected components are of the following types:  a disk with $0$, $1$, or $2$ punctures, an annulus with $0$ or $1$ punctures, a sphere with three boundary components and no punctures, a torus with one puncture and no boundary components, and a sphere with four punctures and no boundary components.
(See \cite[Theorem~3.2]{unitorus} for the once-punctured torus, \cite[Theorem~1.1]{unisphere} for the four-punctured sphere, and \cite[Theorem~7.4]{unisurface} for the other surfaces.)
Interestingly, this list is closed under resection.
This list includes the surfaces of finite type---those with finitely many triangulations.
No surfaces are known (or conjectured) not to have the Null Tangle Property except the null surfaces.

In her Ph.D. thesis~\cite{Viel}, Shira Viel extends Theorem~\ref{id phen resect} to the case of orbifolds in the sense of \cite{FeShT2}.  
Specifically, Viel considers resection, in the sense of this paper, on orbifolds and also defines a notion of resection that is special to orbifolds.
For both types of resection, she proves the analog \cite[Theorem~4.1.1]{Viel} of Theorem~\ref{id phen resect}.

Say that $B'=[b'_{ij}]$ is obtained from $B=[b_{ij}]$ by \newword{erasing edges} if $b'_{ij}\in\set{0,b_{ij}}$ for every pair $i,j$.
In this case, $B$ dominates $B'$.
The reference to edges refers to a graph on the indices of $B$ with an edge connecting two indices $i$ and $j$ if and only if $b_{ij}\neq0$:
The edges ``erased'' are $i$---$j$ such that $b'_{ij}=0\neq b_{ij}$.

\begin{theorem}\label{id phen erase}
Suppose the indices of $B$ are written as a disjoint union $I\cup J$ and suppose $B'$ is obtained by erasing all edges of $B$ that connect indices in $I$ to indices in $J$.
Then the identity map from $\reals^B$ to $\reals^{B'}$ is mutation-linear.
\end{theorem}

A $2\times 2$ exchange matrix is $B=\begin{bsmallmatrix*}[r]0&a\\b&0\end{bsmallmatrix*}$ with $ab\le0$.
It is of finite type if and only if $ab\ge-3$ and of affine type if and only if $ab=-4$.
Otherwise, it is of wild type.
The next theorem completely describes Phenomenon~\ref{id phen} for $2\times 2$ exchange matrices.

\begin{theorem}\label{id phen 2x2}
Suppose $B$ and $B'$ are $2\times2$ exchange matrices.
\begin{enumerate}
\item 
If $B$ does not dominate $B'$, then $\id:\reals^B\to\reals^{B'}$ is not mutation-linear.
\item 
If neither $B$ nor $B'$ is wild then $\id:\reals^B\to\reals^{B'}$ is mutation-linear if and only if $B$ dominates $B'$.
\item
If exactly one of $B$ and $B'$ is wild, then $\id:\reals^B\to\reals^{B'}$ is mutation-linear if and only if either
\begin{enumerate}
\item $B'=\begin{bsmallmatrix*}[r]0&0\\0&0\end{bsmallmatrix*}$, or
\item
$B'=\begin{bsmallmatrix*}[r]0&\pm1\\\mp1&0\end{bsmallmatrix*}$ and $B$ dominates $B'$ but agrees with $B'$ except in one position.
\end{enumerate}
\item
If both $B$ and $B'$ are wild, then $\id:\reals^B\to\reals^{B'}$ is not mutation-linear.
\end{enumerate}
\end{theorem}

\subsection{Mutation fans}
The \newword{mutation fan} $\F_B$ of an exchange matrix $B$ is a fan that encodes the piecewise-linear geometry of mutations of $B$.
(See Definition~\ref{mut fan def}.)
Given fans $\F$ and $\F'$ with $\cup\F=\cup\F'$, we say that $\F$ \newword{refines} $\F'$ (and $\F'$ \newword{coarsens} $\F$) if every cone of $\F$ is contained in a cone of $\F'$, or equivalently if every cone of $\F'$ is a union of cones of $\F$.
If the mutation-linear structure $\reals^B$ admits a cone basis in the sense of Definition~\ref{cone R def}, then the identity is mutation-linear from $\reals^B$ to $\reals^{B'}$ if and only if the mutation fan $\F_B$ refines the mutation fan $\F_{B'}$ (Proposition~\ref{refines FB}).
We are not aware of any $B$ such that $\reals^B$ does not admit a cone basis, and we expect that the following phenomenon is equivalent to Phenomenon~\ref{id phen}.

\begin{phen}[Refinement of mutation fans]\label{ref phen}
Suppose $B$ and $B'$ are exchange matrices such that $B$ dominates~$B'$.  
In many cases, the mutation fan $\F_B$ refines the mutation fan $\F_{B'}$.
\end{phen}

In all of the examples and non-examples described above for Phenomenon~\ref{id phen}, a cone basis is known to exist for $\reals^B$.
In particular, our proof of Theorem~\ref{id phen finite} proceeds by proving the following equivalent theorem.

\begin{theorem}\label{ref phen finite}  
Suppose $B$ is an acyclic exchange matrix of finite type, and suppose $B'$ is another exchange matrix dominated by $B$.
Then the mutation fan $\F_B$ refines the mutation fan $\F_{B'}$.
\end{theorem}

We prove Theorem~\ref{ref phen finite} by relating the mutation fan to the Cambrian fan and appealing to the analogous refinement theorem for Cambrian fans, proved in \cite{diagram}.
\begin{remark}
This refinement relation among Cambrian fans was the origin of the present study of dominance.
It was discovered though lattice-theoretic investigations of the weak order \cite{diagram}, following a clue in work of Simion \cite{Simion}.
This is another example where the lattice theory of the weak order on finite Coxeter groups provides a guide to discovering much more general phenomena.
(Compare \cite{cambrian,sortable,typefree,framework}.)
\end{remark}

Given a triangulation $T$ of a marked surface $(\S,\M)$ having the Null Tangle Property (or a weaker property called the Curve Separation Property), the \newword{rational part} of the mutation fan $\F_{B(T)}$ (Definition~\ref{rat part def}) is a fan $\F_\rationals(T)$ called the \newword{rational quasi-lamination fan}.
The cones of $\F_\rationals(T)$ are the real spans of shear coordinates of pairwise compatible collections of certain curves called \newword{allowable curves}.
Since each of these shear coordinates is an integer vector, each cone of $\F_\rationals(T)$ is a rational cone.
Furthermore, we will see that every point in $\rationals^B$ is contained in some cone of $\F_\rationals(T)$.
Thus if $|T|=n$, then $\F_\rationals(T)\cap\rationals^n=\set{C\cap\rationals^n:C\in\F_\rationals(T)}$ is a complete fan in $\rationals^n$.
We prove Theorem~\ref{id phen resect} as a consequence of the following theorem. 
The theorem provides an example of a ``rational version'' of Phenomenon~\ref{ref phen}.

\begin{theorem}\label{ref phen resect}    
Given a marked surface $(\S,\M)$ and a triangulation $T$ with $|T|=n$, perform a resection of $(\S,\M)$ compatible with $T$ to obtain $(\S',\M')$ and let $T'$ be the triangulation of $(\S',\M')$ induced by $T$.
Then $\F_\rationals(T)\cap\rationals^n$ refines $\F_\rationals(T')\cap\rationals^n$.
If also every component of $(\S,\M)$ and $(\S',\M')$ has the Curve Separation Property, then the rational part of $\F_{B(T)}$ refines the rational part of $\F_{B(T')}$.
\end{theorem}

As in the case of mutation-linear maps, the thesis \cite{Viel} extends this example of Phenomenon~\ref{ref phen} to orbifolds \cite[Theorem~4.1.2]{Viel}.

%
Proposition~\ref{refines FB}, mentioned above, also says that, independent of any hypotheses beyond dominance, if the identity map $\reals^B\to\reals^{B'}$ is mutation-linear, then $\F_B$ refines $\F_{B'}$.
Thus we obtain the following theorem from Theorem~\ref{id phen erase}.

\begin{theorem}\label{edge erase ref}
Suppose the indices of $B$ are written as a disjoint union $I\cup J$ and suppose $B'$ is obtained by erasing all edges of $B$ that connect indices in $I$ to indices in $J$.
Then $\F_B$ refines $\F_{B'}$.
\end{theorem}

For every $2\times 2$ exchange matrix $B$, $\reals^B$ admits a cone basis \cite[Section~9]{universal}, so Theorem~\ref{id phen 2x2} implies the following theorem (in light of Proposition~\ref{refines FB} as above).

\begin{theorem}\label{ref phen 2x2}
Suppose $B$ and $B'$ are $2\times2$ exchange matrices.
\begin{enumerate}[\rm(1)]
\item \label{no dom no ref}
If $B$ does not dominate $B'$, then $\F_B$ does not refine $\F_{B'}$.
\item \label{neither wild}
If neither $B$ nor $B'$ is wild then $\F_B$ refines $\F_{B'}$ if and only if $B$ dominates~$B'$.
\item \label{one wild}
If exactly one of $B$ and $B'$ is wild, then $\F_B$ refines $\F_{B'}$ if and only if either
\begin{enumerate}
\item $B'=\begin{bsmallmatrix*}[r]0&0\\0&0\end{bsmallmatrix*}$, or
\item
$B'=\begin{bsmallmatrix*}[r]0&\pm1\\\mp1&0\end{bsmallmatrix*}$ and $B$ dominates $B'$ but agrees with $B'$ except in one position.
\end{enumerate}
\item \label{both wild}
If both $B$ and $B'$ are wild, then $\F_B$ does not refine $\F_{B'}$ unless $B=B'$.
\end{enumerate}\end{theorem}

\subsection{Scattering diagrams}
The mutation fan is related to the \newword{cluster scattering diagram} of \cite{GHKK}.
We will not give a complete definition of cluster scattering diagrams here, but will give enough of the definition and quote results that let us connect them to mutation fans.
The cluster scattering diagram associated to an exchange matrix $B$ is a certain collection $\W$ of walls---codimension-$1$ cones---together with some additional algebraic data that we will mostly ignore here. 
In a way that is made more precise in Section~\ref{scat sec}, the walls cut the ambient space into convex cones, and these cones, together with their faces, comprise a complete fan \cite[Theorem~3.1]{scatfan} called the \newword{(cluster) scattering fan} and denoted by $\ScatFan(B)$.

\begin{phen}[Refinement of scattering fans]\label{ref scat phen}
Suppose $B$ and $B'$ are exchange matrices such that $B$ dominates~$B'$.  
In many cases, $\ScatFan(B)$ refines $\ScatFan(B')$.
\end{phen}

Most, but not all, of our examples and nonexamples of Phenomenon~\ref{ref scat phen} come from the relationship between $\ScatFan(B)$ and $\F_B$.
The most general statement about that relationship is the following theorem, which is \cite[Theorem~4.10]{scatfan}.  

\begin{theorem}\label{scat ref mut}
The scattering fan $\ScatFan(B)$ refines the mutation fan $\F_B$ for any exchange matrix~$B$.
\end{theorem}

It is conjectured \cite[Conjecture~4.11]{scatfan} that the two fans coincide if and only if either $B$ is $2\times2$ and of finite or affine type or $B$ is $n\times n$ for $n>2$ and of finite \emph{mutation} type.
Since $\ScatFan(B)$ and $\F_B$ are known to coincide in finite type, Theorem~\ref{ref phen finite} implies the following example of Phenomenon~\ref{ref scat phen}.

\begin{theorem}\label{ref scat phen finite acyclic}  
If $B$ is acyclic and of finite type and $B$ dominates $B'$, then $\ScatFan(B)$ refines the $\ScatFan(B')$.
\end{theorem}

Also because $\ScatFan(B)$ and $\F_B$ coincide in finite type, the non-example for Phenomenon~\ref{ref phen} (where $B$ comes from a cyclically-directed triangle and $B'$ is obtained by deleting one arrow) is a non-example of Phenomenon~\ref{ref scat phen} as well.

More generally, as an immediate consequence of Theorem~\ref{scat ref mut}, we have the following result which lets us obtain examples of Phenomenon~\ref{ref scat phen} from examples of Phenomenon~\ref{ref phen}.

\begin{theorem}\label{ref phen ref scat phen}
Suppose that $B$ dominates $B'$ and that $\ScatFan(B')$ coincides with $\F_{B'}$.
If Phenomenon~\ref{ref phen} occurs for $B$ and $B'$, then Phenomenon~\ref{ref scat phen} also occurs.
\end{theorem}

%
Theorem~\ref{ref phen ref scat phen} has the following consequence.

\begin{cor}\label{ref scat phen finite}  
Suppose that $B$ is of finite type, that $B$ dominates $B'$, and that Phenomenon~\ref{ref phen} occurs for $B$ and $B'$.
Then $\ScatFan(B)$ refines $\ScatFan(B')$.
\end{cor}

Since finite mutation fans coincide with their rational parts, and since non-null surfaces of finite type have the Null Tangle Property \cite[Theorem~7.4]{unisurface}, we also have the following corollary.

\begin{cor}\label{ref scat finite surface}
Given a marked surface $(\S,\M)$ of finite type with triangulation~$T$, perform a resection of $(\S,\M)$ compatible with $T$ to obtain a triangulation $T'$ on the resected surface.  
Then $\ScatFan(B(T))$ refines $\ScatFan(B(T'))$.
\end{cor}

Our final example of Phenomenon~\ref{ref scat phen} is a departure from the correspondence with Phenomenon~\ref{ref phen}.
For any $2\times 2$ matrix, the full-dimensional rational cones of the scattering fan coincide with the full-dimensional rational cones of the mutation fan (because both are known to coincide with the $\g$-vector cones of clusters).
In wild type, the closure of the complement of the full-dimensional rational cones is a single irrational cone of the mutation fan.
It is expected (see \cite[Example~1.15]{GHKK}) that every rational ray in this irrational cone is a wall of the scattering diagram.
If this expectation is correct, then in particular every ray in the irrational cone is a distinct cone of the scattering fan.
We call the expectation that every ray in the irrational cone is a distinct cone of the scattering fan the Discreteness Hypothesis for the purpose of stating the following theorem, which is proved in Section~\ref{scat sec}.

\begin{theorem}\label{ref scat phen 2x2}
Assume the Discreteness Hypothesis. 
Then for $2\times2$ exchange matrices $B$ and $B'$, $\ScatFan(B)$ refines $\ScatFan(B')$ if and only if $B$ dominates~$B'$.
\end{theorem}

\subsection{Ring homomorphisms between cluster algebras}\label{hom intro} 
Another phenomenon surrounding dominance of exchange matrices is the existence of ring homomorphisms between cluster algebras with \emph{different} exchange matrices of the \emph{same} size, preserving $\g$-vectors.
This thought-provoking surprise suggests, in particular, that the category of principal-coefficients cluster algebras and $\g$-vector-preserving ring homomorphisms may reward study.
(Or perhaps one should consider a category with more morphisms, namely ring homomorphisms that act linearly on $\g$-vectors.)

We start with the acyclic, finite-type version of the phenomenon.
Write $\A_\bullet(B)$ for the principal-coefficients cluster algebra associated to $B$ and recall from above that when $B$ is of finite type, the mutation fan $\F_B$ coincides with the $\g$-vector fan for $B^T$.  
When $B$ is acyclic and of finite type and $B$ dominates $B'$, Theorem~\ref{ref phen finite}---applied to $B^T$ and $(B')^T$---implies that the set of $\g$-vectors of cluster variables in $\A_\bullet(B')$ is a subset of the set of $\g$-vectors of cluster variables in $\A_\bullet(B)$.  
Thus there is a natural inclusion of the sets of cluster variables, sending each cluster variable for $B'$ to the cluster variable for $B$ having the same $\g$-vector.   
We will see below, with additional details, that this inclusion extends to an injective ring homomorphism.

The principal-coefficients cluster algebra for an $n\times n$ exchange matrix is a subring of the ring of Laurent polynomials in indeterminates $x_1,\ldots,x_n$, with coefficients (ordinary) integer polynomials in indeterminates $y_1,\ldots,y_n$.
The \newword{$\g$-vector} is a $\integers^n$-grading of the cluster algebra given by setting the $\g$-vector of $x_k$ equal to the standard unit basis vector $\e_k$ and the $\g$-vector of $y_k$ equal to the negative of the $k\th$ column of the exchange matrix.
A \newword{cluster monomial} is a monomial in the cluster variables in some cluster.
In finite type, for each integer vector $\lambda\in\integers^n$, there exists a unique cluster monomial whose $\g$-vector is~$\lambda$.

We define $\A_\bullet(B')$ in terms of primed indeterminates $x'_i$ and $y'_j$.
A ring homomorphism from $\A_\bullet(B')$ is determined by its values on $x'_1,\ldots,x'_n$ and $y'_1,\ldots,y'_n$.
For each $k$, let $z_k$ be the cluster monomial whose $\g$-vector is the $k\th$ column of $B$ minus the $k\th$ column of $B'$.
We define a set map $\nu_\z$ on $\set{x'_1,\ldots,x'_n,y'_1,\ldots,y'_n}$ by
\begin{equation}\label{natural candidate}
\begin{aligned}
\nu_\z(x'_k)&=x_k\\
\nu_\z(y'_k)&=y_kz_k
\end{aligned}
\end{equation}
for all $k=1,\ldots,n$, and extend $\nu_\z$ to a ring homomorphism from $\A_\bullet(B')$ to the ring of Laurent polynomials in $x_1,\ldots,x_n$, with coefficients integer polynomials in $y_1,\ldots,y_n$.
The homomorphism $\nu_\z$ has a nice reformulation in terms of $F$-polynomials.
Define $\hy_i$ to be $y_ix_1^{b_{1i}}\cdots x_n^{b_{ni}}$.
As a consequence of \cite[Corollary~6.3]{ca4}, every cluster monomial $z$ is $x_1^{g_1}\cdots x_n^{g_n}$ times a polynomial $F_z$ in the $\hy_i$, where $(g_1,\ldots,g_n)$ is the $\g$-vector of $z$.
The map $\nu_\z$ sends each $\hy_k'$ to $\hy_k\cdot F_{z_k}$.
\textit{A priori}, it is not clear that the image of $\A_\bullet(B')$ under $\nu_\z$ is even contained in $\A_\bullet(B)$.
The following theorem is a joint result of this paper and Viel's Ph.D. thesis~\cite{Viel}.

\begin{theorem}\label{hom phen fin conj} 
If $B$ is acyclic and of finite type and $B$ dominates $B'$, then $\nu_\z$ is an injective, $\g$-vector-preserving ring homomorphism from $\A_\bullet(B')$ to $\A_\bullet(B)$ and sends each cluster variable in $\A_\bullet(B')$ to a cluster variable in $\A_\bullet(B)$.
\end{theorem}

We have verified Theorem~\ref{hom phen fin conj} computationally for $n\times n$ exchange matrices with $n\le8$, thus in particular handling the exceptional types.
(See Section~\ref{acyc fin sec}.)
Furthermore, Theorem~\ref{hom phen surf thm}, below, in particular verifies the theorem when $B$ is acyclic of type A or D.
The proof is completed in \cite[Theorem~4.15]{Viel}, which uses the orbifolds model to verify the theorem when $B$ is acyclic of type B or C.

To generalize Theorem~\ref{hom phen fin conj} beyond finite type, we need a generalization of the set of cluster monomials.
A natural choice is the theta functions arising from cluster scattering diagrams \cite{GHKK}.
Due to differences in conventions already mentioned (and discussed in \cite[Section~1]{scatcomb}), there is a global transpose.
Because we want to work with $\g$-vectors using the conventions of~\cite{ca4}, the theta functions in this paper refer to the \newword{transposed cluster scattering fan} $\ScatFan^T(B)=\ScatFan(B^T)$.
(See \cite[Section~2.3]{scatcomb}.)
For each integer vector, there is a Laurent polynomial called a \newword{theta function} obtained as a sum of monomials arising from \newword{broken lines} in the transposed cluster scattering diagram.
Let $z_i$ be the theta function of the vector obtained as the $k\th$ column of $B$ minus the $k\th$ column of $B'$ and define $\nu_\z$ as in \eqref{natural candidate}.
The cluster variables are generalized by what we call \newword{ray theta functions}.
There is one ray theta function for each rational ray of the transposed cluster scattering fan.
It is the theta function of the shortest integer vector in the ray.

\begin{phen}[Injective homomorphisms]\label{hom phen}
In many cases, when $B$ dominates $B'$, the map $\nu_\z$ is an injective, $\g$-vector-preserving ring homomorphism from $\A_\bullet(B')$ to $\A_\bullet(B)$.
In a smaller set of cases, $\nu_\z$ sends each ray theta function for $B'$ to a ray theta function for $B$.
\end{phen}


We describe two instances of Phenomenon~\ref{hom phen} in the $2\times 2$ case.

\begin{theorem}\label{hom phen 2x2 part 2}
If $B$ and $B'$ are $2\times2$ exchange matrices such that $B$ dominates~$B'$, with $B$ of finite or affine type, then both parts of Phenomenon~\ref{hom phen} occur.
\end{theorem}

\begin{theorem}\label{hom phen 2x2 part 1}
If $B$ and $B'$ are $2\times2$ exchange matrices such that $B$ dominates $B'$, with $B'$ of finite type, then $\nu_\z$ is an injective, $\g$-vector-preserving ring homomorphism from $\A_\bullet(B')$ to $\A_\bullet(B)$.
Assuming the Discreteness Hypothesis, $\nu_\z$ sends ray theta functions to ray theta functions unless $B=\B ba$ and $B'=\B dc$ with $cd=-3$ and $1\not\in\set{|a|,|b|}$.
\end{theorem}

Theorem~\ref{hom phen 2x2 part 1} illustrates why Phenomenon~\ref{hom phen} references a ``smaller set of cases.'' 
Further understanding of Phenomenon~\ref{hom phen} in the $2\times2$ case is limited, for now, by a lack of detailed constructions of scattering diagrams in the ``wild'' $2\times 2$ cases.

The following example of Phenomenon~\ref{hom phen} is precisely the cases of Theorem~\ref{hom phen fin conj} where $B$ is acyclic of type A or D.

\begin{theorem}\label{hom phen surf thm}
Suppose $(\S,\M)$ is a once-punctured or unpunctured disk and suppose $T$ is a triangulation of $(\S,\M)$ such that $B(T)$ is acyclic.
If $B(T)$ dominates $B'$, then both parts of Phenomenon~\ref{hom phen} occur.
\end{theorem}

In Section~\ref{resect hom sec}, we prove Theorem~\ref{hom phen surf thm} by observing that every matrix $B'$ dominated by $B(T)$ can be obtained by a resection and then establishing Phenomenon~\ref{hom phen} for each relevant resection.
Although in proving some of the cases, we generalize slightly to handle some of the non-acyclic cases, our treatment of the surfaces case suggests that acyclicity, at least in some local sense near the changed entries in the exchange matrix, is essential to Phenomenon~\ref{hom phen}.
The arguments given in this paper can be extended to deal with most resections of twice-punctured disks (surfaces of affine type D), providing a step towards a version of Theorem~\ref{hom phen fin conj} where $B$ is of affine type.
The cost is an increase in complexity, and we have not pursued the extension here. 

\subsection{Related work}
Several other lines of research touch tangentially on the ideas presented in this paper.
The notion of dominance appears in \cite{HLY,HL} as an integral part of the definition of \emph{seed homomorphisms}.
In \cite{HLY,HL}, the condition is broader in the sense that it allows passing to submatrices and/or making a global sign change.
Also, the condition in \cite{HLY,HL} is applied to extended exchange matrices, not just to exchange matrices.
Some surface-cutting constructions similar to resection have appeared in \cite{ADS,HL}, but resection is different because it aims to do something different, namely to construct dominance relations without restricting to submatrices and without ``freezing'' variables.

All of the references cited just above are concerned primarily with notions of morphisms of cluster algebras, but, to the author's knowledge, the homomorphisms proposed here do not fit into any category of cluster algebras already proposed.
One category, proposed in \cite{ADS}, features \emph{rooted cluster morphisms}.
These require that initial cluster variables (including ``frozen'' variables) map to cluster variables or to integers.
In our examples, the map restricts to a bijection on (non-frozen) initial cluster variables, but the frozen variables (which we call \emph{coefficients} and denote by $y_i$) may map to other elements of the cluster algebra.
In \cite{HLY}, it is shown that each rooted cluster morphism defines a seed homomorphism, or, loosely speaking, a dominance relationship in the broader sense allowing submatrices and/or a sign change as above.
Other notions of morphisms between cluster algebras include the \emph{coefficient specializations} of \cite[Section~12]{ca4} and \cite[Section~3]{universal} and the \emph{quasi-homomorphisms} of \cite{Fraser}, both of which fix the exchange matrix~$B$.

\fakesubsec{Plan for the rest of the paper}
The rest of this paper is devoted to proving the theorems stated above that provide examples of dominance phenomena.
We begin by defining mutation-linear algebra and proving or quoting key results in Sections~\ref{partial sec}--\ref{mla sec}.
We then consider the phenomena in the order in which they were introduced above, except that we consider Phenomena~\ref{id phen} and~\ref{ref phen} together.

\section{Mutation-linear algebra}\label{mulin sec}
In this section, we quote definitions and results about mutation-linear algebra, prove new preliminary results, and discuss some examples of mutation-linear maps.

\subsection{Partial linear structures}\label{partial sec}
To put the notion of mutation-linear algebra into context, we briefly explore a more general notion that we will call a ``partial linear structure.''
We do not propose at this time a systematic study of partial linear structures, because we are aware of only one interesting class of examples (the mutation-linear structures).
However, we hope that the idea behind mutation-linear structures will be clarified by this brief foray into greater generality.

To make the notion of partial linear structures completely clear, we first formulate the usual notion of linear algebra entirely in terms of linear relations, saying exactly nothing surprising in the process.

Let $M$ be a module over a ring $R$ (with $R$ having a multiplicative identity $1$).
Consider formal expressions of the form $\sum_{i\in S}c_ix_i$, where $S$ is some finite indexing set, each $c_i$ is an element of $R$, and each $x_i$ is an element of $M$.
The expression $\sum_{i\in S}c_ix_i$ is a \newword{linear relation} on $M$ if it evaluates (using the action of $R$ on $M$ and the addition operation in $M$) to the zero element of $M$.
A linear relation is \newword{trivial} if it is empty or can be reduced to the empty relation by repeated applications of combining like terms (i.e.\ replacing $ax+bx$ by $cx$ where $a+b=c$) and deleting terms of the form $0x$.
Since addition in $M$ is commutative, we consider linear relations up to commutativity, so that for example $c_1x_1+c_2x_2$ and $c_2x_2+c_1x_1$ are considered to be the \emph{same} linear relation.
Let $\Rel(M,R)$ be the set of linear relations on $M$ with coefficients in $R$.
A set $A\subseteq M$ is independent if every linear relation among elements of $A$ is trivial.
A set $A\subseteq M$ spans $M$ if, for every element $x\in M$, there exists a linear relation writing $x$ as a linear combination of elements of $A$.
Now suppose $M$ is a module over $R$ and $M'$ is a module over $R'$.
A map $\lambda:M\to M'$ induces a map on formal sums. 
Specifically, reusing the name $\lambda$ for the induced map, $\lambda\bigl(\sum_{i\in S}c_ix_i\bigr)$ is defined to be $\sum_{x\in S}r_x\lambda(x)$.
A map is linear, in the usual sense, if and only if every linear relation in $\lambda(\Rel(M,R))\subseteq\Rel(M',R')$.


In a partial linear structure, we modify all linear-algebraic constructions by ignoring all linear relations not in some fixed subset of $\Rel(M,R)$.

\begin{definition}[\emph{Partial linear structure}]\label{def partial}
A \newword{partial linear structure} is a triple $(M,R,\V)$ such that $M$ is a module over a ring $R$ (with identity) and $\V$ is a subset of $\Rel(M,R)$, satisfying the following conditions:
\begin{enumerate}
\item[(i)] (\textbf{Empty relation}.) The set $\V$ contains the empty relation (the relation with no terms).
\item[(ii)] (\textbf{Irrelevance of zeros}.) For any $x\in M$, a relation $0x+\sum_{i\in S}c_ix_i$ is in $\V$ if and only if $\sum_{i\in S}c_ix_i$ is in $\V$,
\item[(iii)] (\textbf{Combining like terms}.)  If $c=a+b$ in $R$, then $ax+bx+\sum_{i\in S}c_ix_i$ is in $\V$ if and only if $cx+\sum_{i\in S}c_ix_i$ is in $\V$.
\item[(iv)] (\textbf{Scaling}.)  If $c\in R$ and if $\sum_{i\in S}c_ix_i$ is in $\V$, then $\sum_{i\in S}d_ix_i$ is in $\V$, where each $d_i$ is $cc_i$.
\item[(v)] (\textbf{Formal addition}.)  If $\sum_{i\in S}c_ix_i$ and $\sum_{j\in T}d_jy_j$ are in $\V$, then the formal sum $\sum_{i\in S}c_ix_i+\sum_{j\in T}d_jy_j$ is in $\V$.
\end{enumerate}
\end{definition}

\begin{example}[\textbf{The trivial partial linear structure}]\label{trivial example}
If $\V_0$ is the set of all trivial relations, then $(M,R,\V_0)$ is a partial linear structure.
In light of conditions (i) and (ii) of Definition~\ref{def partial}, every partial linear structure $(M,R,\V)$ has $\V_0\subseteq\V$.
\end{example}

The definition of a partial linear structure easily implies the following properties.
\begin{prop}\label{properties}
Suppose $(M,R,\V)$ is a partial linear structure.
\begin{enumerate}
{\normalfont \item[(vi)] (\textbf{Tautology.})} $\V$ contains all relations of the form $1x+(-1)x$ for $x\in M$.
{\normalfont \item[(vii)] (\textbf{Formal substitution}.)}  If $a+\sum_{i\in S}(-c_i)x_i$ and $da+\sum_{j\in T}d_jy_j$ are in $\V$, then $\sum_{i\in S}(dc_i)x_i+\sum_{j\in T}d_jy_j$ is in $\V$.
\end{enumerate}
\end{prop}

We will call the relations in $\V$ the \newword{valid} linear relations.
Given a partial linear structure, new versions of the usual linear-algebraic notions can be obtained by substituting the set $\V$ for the set $\Rel(M,R)$ in the definitions.
For example, we define a set $A\subseteq M$ to be \newword{independent} in $(M,R,\V)$ if every \emph{valid} linear relation among elements of $A$ is trivial.  
A set $A\subseteq M$ \newword{spans} $(M,R,\V)$ if, for every element $x\in M$, there exists a \emph{valid} linear relation $x+\sum_{i\in S}(-c_i)x_i$ with $\set{x_i:i\in S}\subseteq A$.
A \newword{basis} for $(M,R,\V)$ is an independent spanning set.

\begin{theorem}\label{basis exists}
If $(M,R,\V)$ is a partial linear structure and $R$ is a field, then $(M,R,\V)$ admits a basis.
\end{theorem}
The proof of Theorem~\ref{basis exists} follows the usual non-constructive proof of the existence of a basis for an arbitrary vector space, using Zorn's lemma.
This proof, for the special case of mutation-linear algebra defined in Section~\ref{mla sec}, is given in \cite[Proposition~4.6]{universal}.
We omit it here.  

Given partial linear structures $(M,R,\V)$ and $(M',R',\V')$, a map $\lambda:M\to M'$ is \newword{linear}, with respect to the partial linear structures, if the induced map on linear relations restricts to a map from $\V$ to $\V'$.
That is, the map is linear if every \emph{valid} linear relation is mapped to a \emph{valid} linear relation.

\subsection{Mutation-linear structures}\label{mla sec}
We now define a partial linear structure that we call a mutation-linear structure.
An \newword{exchange matrix} is an $n\times n$ skew-symmetrizable integer matrix $B=[b_{ij}]$.
Skew-symmetrizability means that there exist positive integers $d_1,\ldots,d_n$ with $d_ib_{ij}=-d_jb_{ji}$ for all $i,j\in[n]$ and implies that $b_{ij}$ and $b_{ji}$ are either both zero or have strictly opposite signs.

Let $\tB$ be an $(n+\ell)\times n$ matrix whose top $n$ rows agree with $B$ and whose last $\ell$ rows are vectors in $\reals^n$, with $\ell\ge0$.
For each $k=1,\ldots,n$, the \newword{mutation} of $\tB$ at index $k$ is $\tB'=\mu_k(\tB)$ with entries given by
\begin{equation}\label{b mut}
b_{ij}'=\left\lbrace\!\!\begin{array}{ll}
-b_{ij}&\mbox{if }i=k\mbox{ or }j=k;\\
b_{ij}+\sgn(b_{kj})\,[b_{ik}b_{kj}]_+&\mbox{otherwise.}
\end{array}\right.
\end{equation}
Here $[x]_+$ means $\max(0,x)$.
The operation $\mu_k$ is an involution (i.e.\ $\mu_k(\mu_k(\tB))=\tB$).
We also use the symbol $\mu_k$ to denote the map $B\mapsto\mu_k(B)$ given by the same formula.
Given a sequence $\k=k_q,\ldots,k_1$, the notation $\mu_\k$ means $\mu_{k_q}\circ\mu_{k_{q-1}}\circ\cdots\circ\mu_{k_1}$.
An exchange matrix $B$ is called \newword{mutation finite} if the set $\set{\mu_\k(B)}$, where $\k$ ranges over all sequences of indices of $B$, is finite.

Given $B$ and a sequence $\k$ of indices, the \newword{mutation map} $\eta_\k^B:\reals^n\to\reals^n$ is defined as follows:
Given $\a\in\reals^n$, let $\tB$ be the $(n+1)\times n$ matrix with $B$ in the top $n$ rows and $\a$ in the bottom row.
Then $\eta_\k^B(\a)$ is defined to be the bottom row of $\mu_\k(\tB)$.
This is a piecewise linear homeomorphism from $\reals^n$ to itself.
When $\k$ consists of a single index $k$, if $\a=(a_1,\ldots,a_n)$, then $\eta_k^B(\a)=(a_1',\ldots,a_n')$ with each $a_j'$ given by
\begin{equation}\label{mutation map def}
a'_j=\left\lbrace\begin{array}{ll}
-a_k&\mbox{if }j=k;\\
a_j+a_kb_{kj}&\mbox{if $j\neq k$, $a_k\ge 0$ and $b_{kj}\ge 0$};\\
a_j-a_kb_{kj}&\mbox{if $j\neq k$, $a_k\le 0$ and $b_{kj}\le 0$};\\
a_j&\mbox{otherwise.}
\end{array}\right.
\end{equation}
When $\k$ is the empty sequence, $\eta^B_\k$ is the identity map.
For a nonempty sequence $\k=k_q,k_{q-1},\ldots,k_1$, if we define $B_1=B$ and $B_{i+1}=\mu_{k_i}(B_i)$ for $i=1,\ldots,q$, then
\begin{equation}\label{eta def}
\eta_\k^B=\eta^B_{k_q,k_{q-1}\ldots,k_1}=\eta_{k_q}^{B_{q}}\circ\eta_{k_{q-1}}^{B_{q-1}}\circ\cdots\circ\eta_{k_1}^{B_1}
\end{equation}

Let $R$ be $\integers$ or any subfield of $\reals$.
We call $R$ the \newword{underlying ring}.
Since the entries of $B$ are integers, the mutation maps on $\reals^n$ restrict to maps $R^n\to R^n$ for any underlying ring $R$.
These maps also commute with scaling by a nonnegative element $c\in R$, and have the property that 
\begin{equation}\label{eta antipodal}
\eta_\k^B(\a)=-\eta_\k^{-B}(-\a).
\end{equation}

\begin{definition}[\emph{$B$-coherent linear relation}]\label{B coher def}
Given a finite set $S$, vectors $(\v_i:i\in S)$ in $\reals^n$, and elements $(c_i:i\in S)$ of $R$, the formal expression $\sum_{i\in S}c_i\v_i$ is a \newword{$B$\nobreakdash-coherent linear relation with coefficients in $R$} if the equalities
\begin{eqnarray}
\label{linear eta}
&&\sum_{i\in S}c_i\eta^B_\k(\v_i)=\mathbf{0},\mbox{ and}\\
\label{piecewise eta}
&&\sum_{i\in S}c_i\mathbf{min}(\eta^B_\k(\v_i),\mathbf{0})=\mathbf{0}
\end{eqnarray}
hold for every sequence $\k=k_q,\ldots,k_1$ of indices.
Here $\mathbf{min}$ takes the minimum in each component separately.
Since, in particular, \eqref{linear eta} holds when $\k$ is the empty sequence, a $B$-coherent linear relation is in particular a linear relation in the usual sense. 
Recall that a linear relation $\sum_{i\in S}c_i\v_i$ is \newword{trivial} if it is empty or can be reduced to the empty relation by repeated applications of combining like terms (i.e.\ replacing $a\v+b\v$ by $c\v$ where $a+b=c$) and deleting terms of the form $0\v$.
\end{definition}

\begin{example}\label{not B coher}
It is easy to produce linear relations that are not $B$-coherent.
For example, if $B=\begin{bsmallmatrix*}[r]0&1\\-1&0\end{bsmallmatrix*}$, the relation $1\cdot\begin{bsmallmatrix*}[r]1&0\end{bsmallmatrix*}+1\cdot\begin{bsmallmatrix*}[r]-1&0\end{bsmallmatrix*}$ is not $B$-coherent because $\eta^B_1(\begin{bsmallmatrix*}[r]1&0\end{bsmallmatrix*})=\begin{bsmallmatrix*}[r]-1&1\end{bsmallmatrix*}$ but $\eta^B_1(\begin{bsmallmatrix*}[r]-1&0\end{bsmallmatrix*})=\begin{bsmallmatrix*}[r]1&0\end{bsmallmatrix*}$.
On the other hand, the relation $1\cdot\begin{bsmallmatrix*}[r]1&0\end{bsmallmatrix*}+1\cdot\begin{bsmallmatrix*}[r]0&1\end{bsmallmatrix*}+(-1)\cdot\begin{bsmallmatrix*}[r]1&1\end{bsmallmatrix*}$ \emph{is} $B$-coherent.
To see that this latter relation is $B$-coherent, one can apply mutation maps of the form $\eta_{1212\cdots}^B$ and $\eta_{2121\cdots}^B$ to the relation and observe that the results exhibit periodicity.
Thus \eqref{linear eta} and \eqref{piecewise eta} need only be checked a finite number of times.
Alternatively, and looking ahead, one can prove $B$-coherence of the relation using Proposition~\ref{local coherent}.
\end{example}

\begin{definition}[\emph{Mutation-linear structure}]\label{mu-lin struct def}
Given any underlying ring $R$ and any exchange matrix $B$, the \newword{mutation-linear structure} is the partial linear structure $R^B=(R^n,R,\V)$ where $\V$ is the set of $B$-coherent relations over $R$.
We will also write $R^B$ for the set $R^n$, understood to have this mutation-linear structure.
\end{definition}

Since $R^B$ is a partial linear structure, we can do ``mutation-linear algebra'' in~$R^B$.
That is, we can study the notions of basis and of linear maps with respect to the partial linear structure $(R^n,R,\V)$.
A set $A\subseteq R^B$ is \newword{independent} if every $B$-coherent linear relation among elements of $A$ is trivial.
The set $A$ is \newword{spanning} if for every $\v\in R^n$, there exists a $B$-coherent linear relation $\v-\sum_{i\in S}c_i\v_i$ with $\set{\v_i:i\in S}\subseteq A$. 
A \newword{basis} for $R^B$ is a set that is both independent and spanning.
Theorem~\ref{basis exists} implies that, when $R$ is a field, a basis exists for $R^B$.
See \cite[Proposition~4.6]{universal} for a proof in the mutation-linear case.
We have no proof that a basis exists for $\integers^B$ for every exchange matrix $B$, but also we have no example of an exchange matrix $B$ such that $\integers^B$ has no basis.


\begin{remark}\label{R- notation}
A basis for $R^B$ was called an ``$R$-basis for $B$'' in \cite{universal}, and similarly, there were ``$R$-independent sets for $B$'' and ``$R$-spanning sets for $B$,'' because the notion of a mutation-linear structure $R^B$ was not explicitly named in \cite{universal}.
\end{remark}

\subsection{Mutation-linear maps and the mutation fan}\label{mulin map sec}
A map $\lambda:R^B\to R^{B'}$ induces a map on formal linear combinations that sends $\sum_{i\in S}c_i\v_i$ to $\sum_{i\in S}c_i\lambda(\v_i)$.
\begin{definition}
The map $\lambda:R^B\to R^{B'}$ is \newword{mutation-linear} if every $B$-coherent linear relation with coefficients in $R$ is sent by $\lambda$ to a $B'$-coherent linear relation with coefficients in $R$.
\end{definition}

\begin{prop}\label{linear scaling}
Suppose $\lambda:R^B\to R^{B'}$ is a mutation-linear map.
If $\v\in R^B$ and $\w=c\v$ for $c$ a nonnegative element of $R$, then $\lambda(\w)=c\lambda(\v)$.
\end{prop}
\begin{proof}
Since each mutation map $\eta_\k^B$ commutes with nonnegative scaling, the relation $c\v+(-1)\w$ is $B$-coherent.
Thus if $\lambda:R^B\to R^{B'}$ is a mutation-linear map, then $c\lambda(\v)+(-1)\lambda(\w)$ is a $B'$-coherent linear relation.
In particular $c\lambda(\v)+(-1)\lambda(\w)$ is a linear relation, so $\lambda(\w)=c\lambda(\v)$.
\end{proof}

We omit the proof of the following easy proposition.


\begin{prop}\label{restrict linear}
Suppose $R_0$ and $R_1$ are underlying rings with $R_0\subseteq R_1$.
If $\lambda:R_1^B\to R_1^{B'}$ is mutation-linear and restricts to a map from $R_0^B$ to $R_0^{B'}$, then the restriction is mutation-linear.
\end{prop}

On the other hand, given a mutation-linear map $\lambda:\integers^B\to\integers^{B'}$, one can extend $\lambda$ to a map $\overline\lambda:\rationals^B\to\rationals^{B'}$ by clearing denominators in the usual way, and $\overline{\lambda}$ can be shown to be mutation-linear as well.
In other cases, it is not clear whether mutation-linear maps can be extended to larger underlying rings. 

\begin{definition}[\emph{The mutation fan}]\label{mut fan def}
We define an equivalence relation $\equiv^B$ on $\reals^B$ by setting $\a_1\equiv^B\a_2$ if and only if $\vsgn(\eta^B_\k(\a_1))=\vsgn(\eta^B_\k(\a_2))$ for every sequence $\k$ of indices.
Here $\vsgn(\a)$ denotes the vector of signs ($-1$, $0$,  or $1$) of the entries of $\a$.
The equivalence classes of $\equiv^B$ are called \newword{$B$-classes}.
The closures of $B$-classes are called \newword{$B$-cones}.
These are closed convex cones \cite[Proposition~5.4]{universal}, meaning that they are closed under nonnegative scaling and addition.
The set of $B$-cones and their faces constitute a complete fan \cite[Theorem~5.13]{universal} called the \newword{mutation fan}~$\F_B$.
(A fan is a collection of convex cones, closed under taking faces, such that the intersection of any two cones is a face of each.
Although we have no examples of $B$-cones that are not \emph{polyhedral} cones, the possibility has not been ruled out.
Thus to refer to a ``face'' in this definition, we must use the general definition of a face of a convex body, rather than the more special notion of a face of a polyhedron.)
\end{definition}

The following results are \cite[Proposition~7.1]{universal} and part of \cite[Proposition~5.5]{universal}.
\begin{prop}\label{antipodal FB}
$\F_{-B}=-\F_B$.
\end{prop}

\begin{prop}\label{cones preserved}
For any sequence $\k$ of indices, a set $C$ is a $B$-cone if and only if $\eta_\k^B(C)$ is a $\mu_\k(B)$-cone.
\end{prop}

The formal expression $\sum_{i\in S}c_i\v_i$ is a \newword{$B$-local linear relation} if it is a linear relation in the usual sense and if $\set{\v_i:i\in S}$ is contained in some $B$-cone.
The following is \cite[Proposition~5.9]{universal}.
\begin{prop}\label{local coherent} 
Every $B$-local linear relation is a $B$-coherent linear relation.
\end{prop}

A collection of vectors in $R^n$ is \newword{sign-coherent} if for any $k\in[n]$, the $k\th$ coordinates of the vectors in the collection are either all nonnegative or all nonpositive.
The following is \cite[Proposition~5.30]{universal}.
\begin{prop}\label{contained Bcone}
A set $C\subseteq\reals^n$ is contained in some $B$-cone if and only if the set $\eta_\k^B(C)$ is sign-coherent for every sequence $\k$ of indices in $[n]$.
\end{prop}

\begin{prop}\label{B-cone necessary}
If $\lambda:R^B\to R^{B'}$ is a mutation-linear map and $C$ is any cone of $\F_B$, then the restriction of $\lambda$ to $C\cap R^B$ is a linear map (in the usual sense) into some cone of $\F_{B'}$.
\end{prop}
\begin{proof}
We argue the case where $C$ is a $B$-cone.
Since every cone of $\F_B$ is a face of some $B$-cone, the result then follows for arbitrary cones of $\F_B$.

Let $\v$ and $\w$ be elements of $C\cap R^B$ and let $\x=\v+\w$.
Then $\x$ is in $C\cap R^B$ as well.
The linear relation $\v+\w+(-1)\x$ is a $B$-local linear relation, and thus $B$-coherent by Proposition~\ref{local coherent}.
Mutation-linearity of $\lambda$ implies that $\lambda(\v)+\lambda(\w)+(-1)\lambda(\x)$ is a $B'$-coherent linear relation, and thus in particular a linear relation in the usual sense.
Thus $\lambda(\x)=\lambda(\v)+\lambda(\w)$, and this, together with Proposition~\ref{linear scaling}, shows that the restriction of $\lambda$ to $C\cap R^B$ is linear.

Suppose for the sake of contradiction that $\lambda$ does not map $C\cap R^B$ into a $B'$-cone.
Then in particular, there are two elements $\v$ and $\w$ of $C\cap R^B$ such that $\lambda(\v)$ and $\lambda(\w)$ are not contained in a common $B'$-cone.
Proposition~\ref{contained Bcone} says that there exists a sequence $\k$ of indices and an index $j$ such that $\eta^{B'}_\k\lambda(\v)$ and $\eta^{B'}_\k\lambda(\w)$ strictly disagree in the sign of their $j\th$ entry.
Without loss of generality, $\eta^{B'}_\k\lambda(\v)$ has strictly positive $j\th$ entry and $\eta^{B'}_\k\lambda(\w)$ has strictly negative $j\th$ entry.
Again, let $\x=\v+\w$.
Write $h_\v$, $h_\w$ and $h_\x$ for the $j\th$ entries of $\eta^{B'}_\k\lambda(\v)$, $\eta^{B'}_\k\lambda(\w)$, and $\eta^{B'}_\k\lambda(\x)$, respectively.
As above, $\lambda(\v)+\lambda(\w)+(-1)\lambda(\x)$ is a $B'$-coherent linear relation.
Thus if the $j\th$ entry of $\eta^{B'}_\k\lambda(\x)$ is nonpositive, then \eqref{linear eta} and \eqref{piecewise eta} imply that $h_\v+h_\w-h_\x=0$ and $h_\w-h_\x=0$, so that $h_\v=0$.
If the $j\th$ entry of $\eta^{B'}_\k\lambda(\x)$ is positive, then \eqref{piecewise eta} implies that $h_\w=0$.
In either case, we have reached a contradiction, because $h_\v$ and $h_\w$ strictly disagree in sign.
\end{proof}

We wish to prove the converse of Proposition~\ref{B-cone necessary}, but we will need an additional hypothesis.
(The additional hypothesis is necessary for our proof, but we know of no counterexample to the converse without the additional hypothesis.)

\begin{definition}\label{cone R def}
A basis $U$ for $R^B$ is \newword{positive} if, for every $\a\in R^n$, there exists a $B$-coherent linear relation $\a+\sum_{i\in S}(-c_i)\b_i$ such that all of the $c_i$ are nonnegative elements of $R$.
A \newword{cone basis} for $R^B$ is an independent set $U$ in $R^B$ such that, for every $B$-cone $C$, the $R$-linear span of $U\cap C$ contains $R^n\cap C$.
\end{definition}

By \cite[Proposition~6.4]{universal}, a cone basis for $R^B$ also spans $R^B$, and thus is a basis for $R^B$.
In fact, a stronger property than spanning is easily proved. 
\begin{prop}\label{B-cone comb}
Suppose a cone basis $U$ exists for $R^B$.
Then for every $\a\in R^B$, there exists a $B$-local linear relation $\a+\sum_{i\in S}(-c_i)\v_i$ such that each $\v_i$ is in $U$.
\end{prop}
\begin{proof}
If $C$ is any $B$-cone containing $\a$, then the definition of a cone basis says that $\a$ is a linear combination (in the usual sense) $\sum_{i\in S}c_i\v_i$ of elements of $U\cap C$.
\end{proof}

Our only use for positive bases in this paper is the following fact, which allows us to quote results from other papers where positive bases are constructed.
This fact is a part of \cite[Proposition~6.7]{universal}.
\begin{prop}\label{positive cone basis}
If $U$ is a positive basis for $R^B$, then it is a cone basis for $R^B$.
\end{prop}

The existence of a cone basis for $\reals^B$ greatly simplifies mutation-linear algebra.

\begin{proposition}\label{cone B-local}
If a cone basis exists for $R^B$, then a linear relation is $B$-coherent if and only if it can be reduced to a (possibly empty) formal sum of $B$-local relations by a sequence of changes, each of which adds a term with coefficient zero or un-combines like terms.
\end{proposition}
\begin{proof}
Proposition~\ref{local coherent} implies that any formal sum of $B$-local relations is $B$-coherent.
Deleting a term with coefficient zero or combining like terms preserves $B$-coherence, so any linear relation that can be reduced to a formal sum of $B$-local relations by adding terms and/or un-combining like terms is $B$-coherent.

On the other hand, let $\sum_{i\in S}c_i\v_i$ be a $B$-coherent linear relation.
For each $i\in S$, write a $B$-local linear relation $\v_i+\sum_{j\in T_i}(-d_{ij})\w_{ij}$ such that each $\w_{ij}$ is an element of the cone basis.
(We can do this by Proposition~\ref{B-cone comb}.)
Starting with the $B$-coherent linear relation $\sum_{i\in S}c_i\v_i$, we apply Proposition~\ref{properties}(vii) repeatedly to replace each $c_i\v_i$ with $\sum_{j\in T_i}(c_id_{ij})\w_{ij}$.
We conclude that $\sum_{i\in S}\sum_{j\in T_i}(c_id_{ij})\w_{ij}$ is a $B$-coherent linear relation.
Since the $\w_{ij}$ are chosen from a basis, this relation is trivial, meaning that either it is empty or it can be reduced to the empty relation by repeated applications of combining like terms (replacing $a\x+b\x$ by $c\x$ where $a+b=c$) and deleting terms of the form $0\x$.
Therefore also $\sum_{i\in S}\sum_{j\in T_i}(-c_id_{ij})\w_{ij}$ is trivial.

Now, starting again with the $B$-coherent linear combination $\sum_{i\in S}c_i\v_i$, we can add in the trivial relation $\sum_{i\in S}\sum_{j\in T_i}(-c_id_{ij})\w_{ij}$, by a sequence of additions of terms with coefficient zero and/or un-combinings of like terms.
The result can be rewritten $\sum_{i\in S}\bigl(c_i\v_i+\sum_{j\in T_i}(-c_id_{ij})\w_{ij}\bigr)$ using commutativity.
(Recall that we consider linear relations up to commutativity.)
Each $c_i\v_i+\sum_{j\in T_i}(-c_id_{ij})\w_{ij}$ is $B$-local because $\v_i+\sum_{j\in T_i}(-d_{ij})\w_{ij}$ is.
\end{proof}

We now prove a converse to Proposition~\ref{B-cone necessary}, when a cone basis exists for $R^B$.

\begin{prop}\label{B-cone sufficient R}
Suppose $R^B$ admits a cone basis.
Then $\lambda:R^B\to R^{B'}$ is mutation-linear if and only if, for any $B$-cone $C$, the restriction of $\lambda$ to $C\cap R^B$ is a linear map (in the usual sense) into some $B'$-cone.
\end{prop}

\begin{proof}
One direction is Proposition~\ref{B-cone necessary}.
For the other direction, given a $B$-coherent linear relation $R=\sum_{i\in S}c_i\v_i$, let $\lambda(R)$ be the formal sum $\sum_{i\in S}c_i\lambda(\v_i)$.
Proposition~\ref{cone B-local} says that $R$ can be reduced to a formal sum of $B$-local relations by adding terms with coefficient zero or un-combining like terms.
If the corresponding changes are made to $\lambda(R)$, by the hypothesis on $\lambda$, the result is a formal sum of $B'$-local linear relations, which is therefore $B'$-coherent by Proposition~\ref{local coherent}.
Then $\lambda(R)$ is $B$-coherent because undoing the changes (i.e.\ recombining like terms and deleting zero terms) preserves $B$-coherence.
\end{proof}

\begin{remark}\label{cone basis exist remark}
We have no proof that every mutation-linear structure $R^B$ has a cone basis.
(Indeed, as mentioned, we have no proof that $\integers^B$ admits even a basis for every $B$.)
However, we know of no example of a mutation-linear structure $R^B$ that does not admit a cone basis.
\end{remark}

\subsection{Mutation-linear algebra and universal coefficients}\label{functor sec}
It was mentioned in Section~\ref{mulin intro} that by \cite[Theorem~4.4]{universal}, finding a basis for $R^B$ is the same thing as finding a cluster algebra for $B$ with universal (geometric) coefficients.
We now elaborate on that remark, and describe how mutation-linear maps figure into the picture.
For more details and background, see \cite[Sections~2--4]{universal}.

For a fixed exchange matrix $B$, we consider a category $\Geom_R(B)$ whose objects are \newword{cluster algebras of geometric type} over $R$ in the sense of \cite[Section~2]{universal}, all having the same initial exchange matrix $B$ and all having the same initial cluster.
This notion of geometric type is broader than the usual notion in \cite{ca4}, because it allows extended exchange matrices to have infinitely many coefficient rows, and these coefficient rows are allowed to have entries in $R$, which may be larger than~$\integers$.
The arrows in $\Geom_R(B)$ are \newword{coefficient specializations}, similar to those defined in \cite[Section~12]{ca4}, but with an extra topological requirement necessitated by the possible infinity of rows.
(When there are countably many rows, the requirement is that the coefficient specialization be continuous in a ``formal power series'' topology on the coefficient semifield.)
Let $B$ be an exchange matrix and let $I$ and $J$ be arbitrary index sets.
Let $\tB_1$ be an extended exchange matrix with exchange matrix $B$ and coefficient rows $(a_i:i\in I)$.
Let $\tB_2$ be an extended exchange matrix with exchange matrix~$B$ and coefficient rows $(b_j:j\in J)$.
A coefficient specialization between cluster algebras $\A(\tB_1)$ and $\A(\tB_2)$ amounts to a collection of $B$-coherent linear relations $b_j+\sum_{i\in S_j}(-c_{ij})a_i$, one for each $j\in J$, expressing the coefficient row $b_j$ of $\tB_2$ as a $B$-coherent linear combination of coefficient rows of $\tB_1$.

Given two exchange matrices $B$ and $B'$ and a mutation-linear map ${\lambda:R^B\to R^{B'}}$, we define a functor from $\Geom_R(B)$ to $\Geom_R(B')$:
If $\tB$ is an extension of $B$ with coefficient rows ${(c_k:k\in K)}$, then the functor sends $\A(\tB)$ to $\A(\tB')$, where $\tB'$ is an extension of $B'$ with coefficient rows $(\lambda(a_k):k\in K)$.
For $\tB_1$ and $\tB_2$ as in the previous paragraph, a coefficient specialization from $\A(\tB_1)$ to $\A(\tB_2)$ defined by $B$-coherent linear relations $b_k+\sum_{i\in S_k}(-c_{ik})a_i$ is sent to the coefficient specialization from $\A(\tB'_1)$ to $\A(\tB'_2)$ defined by the linear relations $\lambda(b_k)+\sum_{i\in S_k}(-c_{ik})\lambda(a_i)$.
The latter relations are $B'$-coherent because $\lambda$ is mutation-linear.
Indeed, given an arbitrary map $\lambda:R^B\to R^{B'}$, the construction described here yields a functor if and only if $\lambda$ is mutation-linear.

\subsection{Examples of mutation-linear maps}\label{ex sec}
Phenomenon \ref{id phen} considers bijective mutation-linear maps that are not isomorphisms, because their inverses are not mutation-linear (unless $B'=B$).
By way of context, here we mention some examples of mutation-linear maps, specifically mutation-linear isomorphisms, surjections and injections.
We omit the straightforward proofs.
In Section~\ref{phen 1 2 sec}, we return to the topic of dominance and discuss Phenomenon~\ref{id phen}.

\subsubsection{Mutation-linear isomorphisms}\label{iso sec}
\begin{prop}\label{mut map iso}
For any exchange matrix $B$, any underlying ring $R$ and any sequence $\k$ of indices, the mutation map $\eta_\k^B:R^B\to R^{\mu_\k(B)}$ is a mutation-linear isomorphism.
\end{prop}


\begin{prop}\label{antip iso}
For any exchange matrix $B$ and any underlying ring $R$, the antipodal map $\x\mapsto-\x$ is a mutation-linear isomorphism from $R^B$ to $R^{-B}$.
\end{prop}


Given a permutation $\pi$ of the indexing set of $B=[b_{ij}]$, let $\pi(B)$ be $[b_{\pi(i)\pi(j)}]$ and let $\pi$ also denote the linear map sending a vector $(a_1,\ldots,a_n)$ to $(a_{\pi(1)},\ldots,a_{\pi(n)})$.

\begin{prop}\label{pi iso}
For any exchange matrix $B$ and any underlying ring $R$, the linear reindexing map $\pi:R^B\to R^{\pi(B)}$ is a mutation-linear isomorphism.
\end{prop}
\begin{proof}
One can check that for any sequence $\k$ of indices and any vector $\v$, we have $\eta^{\pi(B)}_\k(\pi(\v))=\pi(\eta^B_{\pi(\k)}(\v))$.
Thus, given a $B$-coherent linear map $\sum_{i\in S}c_i\v_i$, for any sequence $\k$, the sum $\sum_{i\in S}c_i\eta_\k^{\pi(B)}(\pi(\v_i))$ is equal to
$\pi\bigl(\sum_{i\in S}c_i\eta_{\pi(\k)}^B(\v_i)\bigr)$, which equals zero because $\sum_{i\in S}c_i\v_i$ is a $B$-coherent linear relation and $\pi$ is a linear map.
Thus $\sum_{i\in S}c_i\pi(\v_i)$ is $\pi(B)$-coherent.
We see that $\pi$ is mutation-linear.  
The inverse map $\pi^{-1}$ is mutation-linear by the same argument.
\end{proof}

Let $B=[b_{ij}]$ and $B'=[b'_{ij}]$ be exchange matrices.
Then $B'$ is a \newword{rescaling} of $B$ if there exists a diagonal matrix $\Sigma=\diag(\sigma_1,\ldots,\sigma_n)$ with positive entries such that $B'=\Sigma^{-1} B\Sigma$.
Equivalently, $b'_{ij}=\frac{\sigma_j}{\sigma_i}b_{ij}$ for all $i$ and $j$.
More information on rescaling exchange matrices is found in \cite[Section~7]{universal}.
%
%

\begin{prop}\label{rescale iso}
Let $B$ be an exchange matrix and take the underlying ring $R$ to be a field.
Suppose $B'$ is a rescaling of $B$.
Write $B'=\Sigma^{-1} B\Sigma$ for $\Sigma$ having entries in $R$.
Then the map $\v\mapsto(\v\Sigma):R^B\to R^{B'}$ is a mutation-linear isomorphism.
\end{prop}

When $R=\integers$, the map $\v\mapsto(\v\Sigma)$ is still mutation-linear from $\integers^B$ to $\integers^{B'}$.
It is injective, but need not be surjective.


\subsubsection{Surjective mutation-linear maps}\label{surj sec}
For any subset $I$ of the indices of $B$, define $B_I$ to be the matrix obtained from $B$ by deleting row $j$ and column $j$ for all $j\not\in I$.
Retain the original indexing of $B_I$, so that it is indexed by~$I$.
Let $\Proj_I:\reals^B\to\reals^{B_I}$ be the usual projection (ignoring the $j\th$ coordinate for each $j\not\in I$).

\begin{prop}\label{Proj mulin}
Let $B$ be an exchange matrix.
For any subset $I$ of the indices of $B$, the map $\Proj_I:\reals^B\to\reals^{B_I}$ is mutation-linear.
\end{prop}

%

We also mention an \textit{ad hoc} construction of a surjective mutation-linear map by ``wrapping'' a ``larger'' mutation fan around a ``smaller'' one.
The exchange matrix $B=\begin{bsmallmatrix*}[r]0&1\\-3&0\end{bsmallmatrix*}$ has eight maximal cones, while the mutation fan for  $B'=\begin{bsmallmatrix*}[r]0&0\\0&0\end{bsmallmatrix*}$ has four.
(See Figure~\ref{finite mutation fans}.)
In this case, it is known that $\reals^B$ admits a cone basis.
One can construct a piecewise-linear branched double-cover $\lambda:\reals^B\to \reals^{B'}$ that maps each cone of $\F_B$ linearly to a cone of $\F_{B'}$ and thus is mutation-linear by Proposition~\ref{B-cone sufficient R}.
This construction admits many variations.


\subsubsection{Injective mutation-linear maps}\label{inj sec}
As mentioned above, nontrivial rescalings, when $R=\integers$ are mutation-linear and injective but possibly not bijective.
Much stranger injective mutation-linear maps can be found from $\reals^B$ to $\reals^{B'}$ when $B$ is $n\times n$ and $B'$ is $m\times m$ with $n<m$.
In this case, by Proposition~\ref{B-cone sufficient R}, one might, for example, map $\F_B$ injectively into the $n$-skeleton of $\F_{B'}$, sending each cone linearly to a cone.
A less strange special case is when $B=(B')_I$ in the sense of Section~\ref{surj sec}, taking the natural injection into the subspace of $\reals^B$ indexed by $I$.
Results of \cite[Section~8]{universal} can be used to show that such an injection is mutation-linear, at least in the case where $(B')_I$ is of finite type.

\section{Mutation-linear maps and refinement of the mutation fan}\label{phen 1 2 sec}
In this section, we discuss Phenomena~\ref{id phen} and~\ref{ref phen}.
%
We begin by observing (in Proposition~\ref{refines FB}), the close connection between Phenomenon~\ref{id phen} and Phenomenon~\ref{ref phen}.
Similarly, in Proposition~\ref{rat part refine mulin}, we connect rational versions of the two phenomena.  
Then we gather examples of the two phenomena. 

\subsection{Connecting Phenomenon~\ref{id phen} and Phenomenon~\ref{ref phen}}\label{mulin refine sec}
The following proposition is an immediate consequence of Propositions~\ref{B-cone necessary} and~\ref{B-cone sufficient R}.

\begin{prop}\label{refines FB}
If the identity map $\reals^B\to\reals^{B'}$ is mutation-linear, then $\F_B$ refines $\F_{B'}$.
If $\reals^B$ admits a cone basis, then the identity map $\reals^B\to\reals^{B'}$ is mutation-linear if and only if $\F_B$ refines $\F_{B'}$.
\end{prop}

In the case of resection of surfaces, where we deal with rational versions of the phenomena, we need a rational version of Proposition~\ref{refines FB}.
In light of Proposition~\ref{restrict linear}, for fixed $B$ and $B'$, the statement that the identity map from $\reals^B$ to $\reals^{B'}$ is mutation-linear is equivalent to the statement that the identity map from $R^B$ to $R^{B'}$ is mutation-linear for any underlying ring $R$.  
Thus in particular the conclusion of Theorem~\ref{id phen resect} is formally weaker than Phenomenon~\ref{id phen}.
We now develop a geometric formulation of the weaker assertion.

%

\begin{definition}[\emph{Rational part of a fan}]\label{rat part def}
Suppose $\F$ and $\tilde\F$ are fans in $\reals^n$ satisfying the following conditions:
\begin{enumerate}[(i)]
\item \label{R cones R part}
Each cone in $\tilde\F$ is the nonnegative $\reals$-linear span of finitely many rational vectors.
\item \label{cone in cone R part}
Each cone in $\tilde\F$ is contained in a cone of $\F$.
\item \label{max cone R part}
For each cone $C$ of $\F$, there is a unique largest cone (under containment) among cones of $\tilde\F$ contained in $C$.
This largest cone contains $\rationals^n\cap C$.
\end{enumerate}
Then $\tilde\F$ is called the \newword{rational part} of $\F$.
The rational part of a given $\F$ might not exist, but if it exists, it is unique.
(See \cite[Definition~6.9]{universal} for examples of non-existence and a proof of uniqueness, in the more general context where $\rationals$ is replaced by any underlying ring $R$.)
\end{definition}

We have defined the \emph{rational} part of a real fan to be another \emph{real} fan.
Instead, one might want to consider the ``rational part'' of a fan as a fan in $\rationals^n$, by taking $\set{C\cap\rationals^n:C\in\tilde\F}$.
We write $\tilde\F\cap\rationals^n$ for this fan in $\rationals^n$.
%
%


\begin{prop}\label{rat part refine mulin}
Suppose that $B$ and $B'$ are $n\times n$ exchange matrices, that $\tilde\F_B$ is the rational part of $\F_B$, and that $\tilde\F_{B'}$ is the rational part of $\F_{B'}$.
If the identity map $\rationals^B\to\rationals^{B'}$ is mutation-linear, then $\tilde\F_B\cap\rationals^n$ refines $\tilde\F_{B'}\cap\rationals^n$.
Assuming also that $\rationals^B$ admits a cone basis, the identity map is mutation-linear if and only if $\tilde\F_B\cap\rationals^n$ refines $\tilde\F_{B'}\cap\rationals^n$.
\end{prop}

\begin{proof}
Suppose the identity map $\rationals^B\to\rationals^{B'}$ is mutation-linear.
Suppose $\tilde C$ is any cone of $\tilde\F_B$.
By definition of the rational part of a fan, there exists a cone $C$ of $\F_B$ containing $\tilde C$.
Proposition~\ref{B-cone necessary} says that $C\cap\rationals^n$ is contained in some cone $C'$ of $\F_{B'}$.
There exists a cone $\tilde C'$ of $\tilde\F_{B'}$ contained in $C'$ and containing $C'\cap\rationals^n$.
Thus $\tilde C\cap\rationals^n\subseteq\tilde C'$.
We have shown that $\tilde\F_B\cap\rationals^n$ refines $\tilde\F_{B'}\cap\rationals^n$.

Assuming now that $\rationals^B$ admits a cone basis, suppose conversely that $\tilde\F_B\cap\rationals^n$ refines $\tilde\F_{B'}\cap\rationals^n$.
Given any $B$-cone $C$, there is a cone $\tilde C$ of $\tilde\F_B$ with $C\cap\rationals^n\subseteq\tilde C\subseteq C$.
Since $\tilde\F_B$ refines $\tilde\F_{B'}$ as fans in $\rationals^n$, there exists a cone $\tilde C'$ of $\tilde\F_{B'}$ containing~$\tilde C$.
Since some cone $C'$ of $\F_{B'}$ contains $\tilde C'$, we have satisfied the hypothesis of Proposition~\ref{B-cone sufficient R}, for $\lambda$ equal to the identity map, so the identity map is mutation-linear.
\end{proof}

\subsection{Erasing edges}\label{erase sec}
We now prove Theorem~\ref{id phen erase}.
Theorem~\ref{edge erase ref} then follows immediately by Proposition~\ref{refines FB}.
We use the following observation, which follows from \eqref{b mut} and \eqref{mutation map def} and a simple induction.
We continue the notation of Section~\ref{surj sec}.
\begin{prop}\label{I mu eta}
Suppose $I$ is a subset of the indices of $B$ and suppose $\k$ is a sequence of indices in $I$.
Then $(\mu_\k(B))_I=\mu_\k(B_I)$.
Furthermore, if $\v\in\reals^B$, then $\eta_\k^{B_I}(\Proj_I(\v))=\Proj_I(\eta_\k^B(\v))$.
\end{prop}

\begin{proof}[Proof of Theorem~\ref{id phen erase}]
By hypothesis, the indices of $B$ are written as a disjoint union $I\cup J$.
Also, $B'=[b'_{ij}]$ has $b'_{ij}=b'_{ji}=0$ for all $i\in I$ and $j\in J$ and every other entry of $B'$ agrees with the corresponding entry of $B$.

Suppose $\sum_{i\in S}c_i\v_i$ is a $B$-coherent linear relation.
We will show that $\sum_{i\in S}c_i\v_i$ is a $B'$-coherent relation.
Considering Proposition~\ref{I mu eta} for all sequences $\k$ of indices in $I$, we see that $\sum_{i\in S}c_i\Proj_I(\v_i)$ is a $B_I$-coherent linear relation.
Similarly, $\sum_{i\in S}c_i\Proj_J(\v_i)$ is $B_J$-coherent.
Each vector $\x\in\reals^B$ can be written $\x=\Proj_I(\x)+\Proj_J(\x)$.
Since $b'_{ij}=0$ whenever $i\in I$ and $j\in J$ and because $B_I=B'_I$ and $B_J=B'_J$, we see from \eqref{mutation map def} that $\eta_k^{B'}(\x)=\eta_k^{B_I}\Proj_I(\x)+\Proj_J(\x)$ if $k\in I$ and $\eta_k^{B'}(\x)=\Proj_I(\x)+\eta_k^{B_J}\Proj_J(\x)$ if $k\in J$.
Let $\k$ be any sequence of indices of $B$, let $\k_I$ be the subsequence of $\k$ consisting of indices in $I$, and let $\k_J$ be the subsequence of $\k$ consisting of indices in $J$.
Then $\sum_{i\in S}c_i\eta_\k^{B'}(\v_i)=\sum_{i\in S}c_i\eta_{\k_I}^{B_I}(\Proj_I(\v_i))+\sum_{i\in S}c_i\eta_\k^{B_J}(\Proj_J(\v_i))$.
Since $\sum_{i\in S}c_i\Proj_I(\v_i)$ is a $B_I$-coherent linear relation and $\sum_{i\in S}c_i\Proj_J(\v_i)$ is a $B_J$-coherent linear relation, we conclude that $\sum_{i\in S}c_i\eta_\k^{B'}(\v_i)=\0$.
We have shown that the identity map $\reals^B\to\reals^{B'}$ is mutation-linear.
\end{proof}

\subsection{\texorpdfstring{The $2\times2$ case}{The 2 by 2 case}}\label{2x2 ref sec}
In this section, we prove Theorem~\ref{ref phen 2x2}, which characterizes when Phenomenon~\ref{ref phen} occurs for $2\times2$ exchange matrices.
As explained in the introduction, results of \cite[Section~9]{universal} combined with Proposition~\ref{refines FB} then immediately imply Theorem~\ref{id phen 2x2}.

Suppose $B=\begin{bsmallmatrix*}[r]0&a\\b&0\end{bsmallmatrix*}$ is a $2\times2$ exchange matrix.
We now briefly review the description of the mutation fan of $B$ given in \cite[Section~9]{universal}.
Since $\F_B$ is complete and contained in $\reals^2$, we can describe it completely by listing its rays.

If $a=b=0$, then $\F_B$ has rays $\pm\begin{bsmallmatrix*}[r]1\\0\end{bsmallmatrix*}$ and $\pm\begin{bsmallmatrix*}[r]0\\1\end{bsmallmatrix*}$.
Otherwise, there are two possible sign patterns for $B$, either $a>0$ and $b<0$ or $a<0$ and $b>0$.
Figure~\ref{finite mutation fans} shows the mutation fans in the case of finite type (i.e.\ when $ab>-4$) with the additional condition that $0\le a\le-b$.
The rays occurring in these pictures are spanned by the vectors $\pm\begin{bsmallmatrix*}[r]1\\0\end{bsmallmatrix*}$, $\pm\begin{bsmallmatrix*}[r]0\\1\end{bsmallmatrix*}$, $\begin{bsmallmatrix*}[r]1\\-1\end{bsmallmatrix*}$, $\begin{bsmallmatrix*}[r]2\\-1\end{bsmallmatrix*}$, $\begin{bsmallmatrix*}[r]3\\-2\end{bsmallmatrix*}$, and $\begin{bsmallmatrix*}[r]3\\-1\end{bsmallmatrix*}$.
\begin{figure}[ht]
\begin{tabular}{ccccccc}
\scalebox{0.7}{\includegraphics{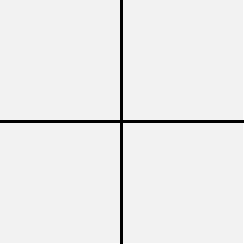}}&&
\scalebox{0.7}{\includegraphics{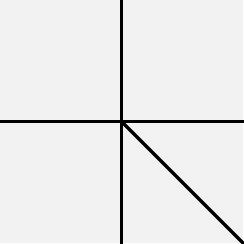}}&&
\scalebox{0.7}{\includegraphics{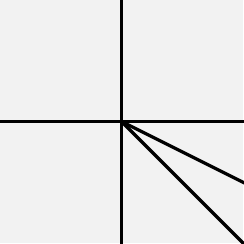}}&&
\scalebox{0.7}{\includegraphics{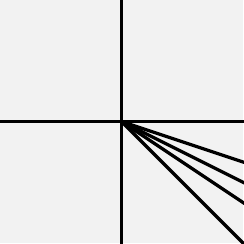}}
\\[2 pt]
$B=\begin{bsmallmatrix*}[r]0&0\\0&0\end{bsmallmatrix*}$&&
$B=\begin{bsmallmatrix*}[r]0&1\\-1&0\end{bsmallmatrix*}$&&
$B=\begin{bsmallmatrix*}[r]0&1\\-2&0\end{bsmallmatrix*}$&&
$B=\begin{bsmallmatrix*}[r]0&1\\-3&0\end{bsmallmatrix*}$
\end{tabular}
\caption{Mutation fans $\F_B$ for $2\times2$ exchange matrices of finite type}
\label{finite mutation fans}
\end{figure}

The remaining finite type cases can be recovered via two observations:
First, Proposition~\ref{antipodal FB} says that $\F_{-B}$ is the antipodal opposite of $\F_B$, and second, passing from $\begin{bsmallmatrix*}[r]0&a\\b&0\end{bsmallmatrix*}$ to $\begin{bsmallmatrix*}[r]0&b\\a&0\end{bsmallmatrix*}$ reflects the mutation fan through the line $x_1=x_2$.

In the cases where $ab\le-4$, we define $P_0=P_1=0$ and, for $m\ge2$
\begin{equation}\label{Pm recur}
P_m=\begin{cases}
-abP_{m-1}-P_{m-2}&\text{if }m\text{ is even, or}\\
P_{m-1}-P_{m-2}&\text{if }m\text{ is odd.}\\
\end{cases}
\end{equation}
This is a recursion, given in \cite[(3.1)]{scatcomb}, for the polynomials defined in \cite[(9.5)]{universal}.
The first several values of $P_m(a,b)$ are shown here.
\[\small\begin{tabular}{|r||c|c|c|c|c|c|c|c|}
\hline
$m$&$0$&$1$&$2$&$3$&$4$&$5$\\\hline
$P_m$	&$1$	&$1$	&$-ab-1$&$-ab-2$&$a^2b^2+3ab+1$	&$a^2b^2+4ab+3$	\\\hline
\end{tabular}\]
Define, for $k\ge0$, vectors 
{\allowdisplaybreaks\begin{align*}
&\v_k=\v_k(a,b)=\begin{cases}
\begin{bsmallmatrix}\sgn(a)P_k(a,b)\\-aP_{k+1}(a,b)\end{bsmallmatrix}&\text{if $k$ is even, or}\\[5pt]
\begin{bsmallmatrix}-bP_k(a,b)\\\sgn(b)P_{k+1}(a,b)\end{bsmallmatrix}&\text{if $k$ is odd.}
\end{cases}\\[5pt]
&\v_\infty=\v_\infty(a,b)=\begin{bsmallmatrix}2\sgn(a)\sqrt{-ab}\\-a(\sqrt{-ab}+\sqrt{-ab-4})\end{bsmallmatrix}\\[5pt]
&\w_k=\w_k(a,b)=\begin{cases}
\begin{bsmallmatrix}-bP_{k+1}(a,b)\\\sgn(b)P_k(a,b)\end{bsmallmatrix}&\text{if $k$ is even, or}\\[5pt]
\begin{bsmallmatrix}\sgn(a)P_{k+1}(a,b)\\-aP_k(a,b)\end{bsmallmatrix}&\text{if $k$ is odd.}
\end{cases}\\[5pt]
&\w_\infty=\w_\infty(a,b)=\begin{bsmallmatrix}-b(\sqrt{-ab}+\sqrt{-ab-4})\\2\sgn(b)\sqrt{-ab}\end{bsmallmatrix}.
\end{align*}}
Still assuming ${ab\le-4}$, the rays of $\F_B$ are spanned by the vectors:
\[\set{\pm\begin{bsmallmatrix*}[r]1\\0\end{bsmallmatrix*}, \pm\begin{bsmallmatrix*}[r]0\\1\end{bsmallmatrix*}}\cup\set{\v_k:k=0,\ldots}\cup\set{\v_\infty}\cup\set{\w_k:k=0,\ldots}\cup\set{\w_\infty}.\]

It is apparent that a mutation fan $\F_B$ in $\reals^2$ refines a mutation fan $\F_{B'}$ in $\reals^2$ if and only if the set of rays of $\F_{B'}$ is contained in the set of rays of $\F_B$.
Since the rays spanned by $\pm\begin{bsmallmatrix*}[r]1\\0\end{bsmallmatrix*}$ and $\pm\begin{bsmallmatrix*}[r]0\\1\end{bsmallmatrix*}$ are in every mutation fan, we can ignore these rays.
If $B'=\begin{bsmallmatrix*}[r]0&0\\0&0\end{bsmallmatrix*}$, then $\F_B$ refines $\F_{B'}$ for every $B$.
Otherwise, it is apparent that $\F_B$ cannot refine $\F_{B'}$ unless $B$ and $B'$ have weakly the same sign pattern. 
By Proposition~\ref{antipodal FB}, we need only consider the case where $B=\begin{bsmallmatrix*}[r]0&a\\b&0\end{bsmallmatrix*}$ with $a>0$ and $b<0$ and $B'$ has weakly the same sign pattern.
Having made this reduction, we also may as well consider, rather than containment of rays, containment of the set of slopes of rays.
We ignore slopes $0$ and $\infty$ and call the slopes of the remaining rays the \newword{relevant slopes}.
In finite type, the relevant slopes are shown in Table~\ref{fin slopes}.
\begin{table}[ht]\vspace{-5pt}
{\small$\begin{array}{ccc|ccc}
&B&&&\text{Slopes}&\\\hline&&&&&\\[-9pt]
&\begin{bsmallmatrix*}[r]0&1\\-1&0\end{bsmallmatrix*}&&&-1&\\[2pt]\hline&&&&&\\[-9pt]
&\begin{bsmallmatrix*}[r]0&1\\-2&0\end{bsmallmatrix*}&&&-1,-\frac12&\\[2pt]\hline&&&&&\\[-9pt]
&\begin{bsmallmatrix*}[r]0&2\\-1&0\end{bsmallmatrix*}&&&-2,-1&\\[2pt]\hline&&&&&\\[-9pt]
&\begin{bsmallmatrix*}[r]0&1\\-3&0\end{bsmallmatrix*}&&&-1,-\frac12,-\frac23,-\frac13&\\[2pt]\hline&&&&&\\[-9pt]
&\begin{bsmallmatrix*}[r]0&3\\-1&0\end{bsmallmatrix*}&&&-3,-2,-\frac32,-1&\\[2pt]\hline\\
\end{array}$}
\caption{Relevant slopes of $\F_B$ for $B$ of finite type, $a>0$ and $b<0$}
\label{fin slopes}
\end{table}

For $ab\le-4$ (and continuing with $a>0$ and $b<0$), the relevant slopes are
{\allowdisplaybreaks\begin{align*}
s_k(a,b)&=\begin{cases}
-\frac{aP_{k+1}(a,b)}{P_k(a,b)}&\text{if $k$ is even, or}\\[4pt]
\frac{P_{k+1}(a,b)}{bP_k(a,b)}&\text{if $k$ is odd.}
\end{cases}\\[2pt]
s_\infty(a,b)&=-\frac{a(\sqrt{-ab}+\sqrt{-ab-4})}{2\sqrt{-ab}}\\[2pt]
t_k(a,b)&=\begin{cases}
\frac{P_k(a,b)}{bP_{k+1}(a,b)}&\text{if $k$ is even, or}\\[4pt]
-\frac{aP_k(a,b)}{P_{k+1}(a,b)}&\text{if $k$ is odd.}
\end{cases}\\[5pt]
t_\infty(a,b)&=\frac{2\sqrt{-ab}}{b(\sqrt{-ab}+\sqrt{-ab-4})}.
\end{align*}}
All of these slopes are negative. 
The slopes $s_k$ increase and limit to $s_\infty$.
The slopes $t_k$ decrease and limit to $t_\infty$, and we have $s_\infty\le t_\infty$.
When $ab=-4$, $t_\infty$ and $s_\infty$ are equal and rational, but for $ab<-4$, $t_\infty$ and $s_\infty$ are distinct and irrational.
Some of these slopes are shown in Table~\ref{aff slopes} for the affine types.
\begin{table}[ht]
{\small$\begin{array}{ccc|ccc|ccc|c}
&B&&&s_0,s_1,s_2,s_3\ldots&&&t_0,t_1,t_2,t_3\ldots&&s_\infty=t_\infty\\\hline&&&&&&&&&\\[-9pt]
&\begin{bsmallmatrix*}[r]0&1\\-4&0\end{bsmallmatrix*}&&&-1,-\frac34,-\frac23,-\frac58,\ldots&&&-\frac14,-\frac13,-\frac38,-\frac25,\ldots&&-\frac12\\[2pt]\hline&&&&&&&&&\\[-9pt]
&\begin{bsmallmatrix*}[r]0&2\\-2&0\end{bsmallmatrix*}&&&-2,-\frac32,-\frac43,-\frac54,\ldots&&&-\frac12,-\frac23,-\frac34,-\frac45,\ldots&&-1\\[2pt]\hline&&&&&&&&&\\[-9pt]
&\begin{bsmallmatrix*}[r]0&4\\-1&0\end{bsmallmatrix*}&&&-1,-\frac43,-\frac32,-\frac85,\ldots&&&-4,-3,-\frac83,-\frac52,\ldots&&-2\\[2pt]\hline\\
\end{array}$}
\caption{Some relevant slopes of $\F_B$ for $B$ of affine type, $a>0$ and $b<0$}
\label{aff slopes}
\end{table}

We now prove the characterization of Phenomenon~\ref{ref phen} for $2\times2$ exchange matrices.

\begin{proof}[Proof of Theorem~\ref{ref phen 2x2}]
A brief inspection of Tables~\ref{fin slopes} and~\ref{aff slopes} is enough to verify Theorem~\ref{ref phen 2x2}\eqref{neither wild} and the non-wild case of Theorem~\ref{ref phen 2x2}\eqref{no dom no ref}.

We next check Theorem~\ref{ref phen 2x2}\eqref{one wild} and the case of Theorem~\ref{ref phen 2x2}\eqref{no dom no ref} where exactly one of $B$ and $B'$ is wild.
If $B'$ is wild and $B$ is not, then $\F_{B'}$ has two irrational relevant slopes, while $\F_B$ has only rational relevant slopes, so $\F_B$ does not refine $\F_{B'}$.
Now consider the case where $B=\begin{bsmallmatrix*}[r]0&a\\b&0\end{bsmallmatrix*}$ is wild and $B'$ is not.
If $B'=\begin{bsmallmatrix*}[r]0&0\\0&0\end{bsmallmatrix*}$, then $\F_B$ refines $\F_{B'}$.
Now consider the case where $B'=\begin{bsmallmatrix*}[r]0&1\\-1&0\end{bsmallmatrix*}$.
If $a>2$ and $b<-2$, then
\[s_\infty=-\frac{a(\sqrt{-ab}+\sqrt{-ab-4})}{2\sqrt{-ab}}<-1<\frac{2\sqrt{-ab}}{b(\sqrt{-ab}+\sqrt{-ab-4})}=t_\infty,\]
and thus $-1$ is not a relevant slope of $\F_B$, so $\F_B$ does not refine $\F_{B'}$.
If $a=1$, then $s_0(a,b)=-a=-1$, so $\F_B$ refines $\F_{B'}$, and similarly if $b=-1$, then $t_0(a,b)=\frac1b=-1$, so $\F_B$ refines $\F_{B'}$.
Next, consider the case where $B'=\begin{bsmallmatrix*}[r]0&1\\-2&0\end{bsmallmatrix*}$.
Since $B'$ dominates $\begin{bsmallmatrix*}[r]0&1\\-1&0\end{bsmallmatrix*}$, by Theorem~\ref{ref phen 2x2}\eqref{neither wild} we conclude that $\F_B$ does not refine $\F_{B'}$ except possibly when $a=1$ or $b=-1$.
But if $a=1$, then $b\le-5$, so
\[s_\infty=-\frac{a(\sqrt{-ab}+\sqrt{-ab-4})}{2\sqrt{-ab}}<-\frac12<\frac{2\sqrt{-ab}}{b(\sqrt{-ab}+\sqrt{-ab-4})}=t_\infty,\]
and thus $-\frac12$ is not a relevant slope of $\F_B$, so $\F_B$ does not refine $\F_{B'}$.
If $b=-1$, then $a\ge5$, and we check similarly that $s_\infty<-2<t_\infty$, so that $\F_B$ does not refine $\F_{B'}$.
The case where $B'=\begin{bsmallmatrix*}[r]0&2\\-1&0\end{bsmallmatrix*}$ is dealt with by the symmetric argument.
Since every other non-wild exchange matrix $B'$ dominates either $\begin{bsmallmatrix*}[r]0&1\\-2&0\end{bsmallmatrix*}$ or $\begin{bsmallmatrix*}[r]0&2\\-1&0\end{bsmallmatrix*}$, we see by Theorem~\ref{ref phen 2x2}\eqref{neither wild} that no non-wild $B'$ can have $\F_B$ refine $\F_{B'}$.
We have verified Theorem~\ref{ref phen 2x2}\eqref{one wild}.

Finally, suppose $B=\begin{bsmallmatrix*}[r]0&a\\b&0\end{bsmallmatrix*}$ and $B'=\begin{bsmallmatrix*}[r]0&c\\d&0\end{bsmallmatrix*}$ are distinct wild $2\times2$ exchange matrices.
Since $\F_B$ and $\F_{B'}$ each have exactly two irrational slopes, for a refinement relation to exist in either direction, in particular, these two slopes must be the same in $\F_B$ as in $\F_{B'}$.
The conditions that the two slopes are the same in each fan can be rewritten, after some manipulation, as 
\[abd-\sqrt{a^2b^2d^2+4abd^2}=bcd-\sqrt{b^2c^2d^2+4b^2cd}\]
and
\[abc+\sqrt{a^2b^2c^2+4abc^2}=acd+\sqrt{a^2c^2d^2+4a^2cd}\]
We know that these slopes are irrational (in the case $ab<-4$), so the square roots in the above equations are all irrational.
No two distinct irrational square roots of integers can differ by a rational number,
so the square roots on each side must be equal, and thus the integers on each side must be equal.
That is, $abd=bcd$ and $abc=acd$, and therefore $a=c$ and $b=d$.
We have verified Theorem~\ref{ref phen 2x2}\eqref{both wild} and the final case of Theorem~\ref{ref phen 2x2}\eqref{no dom no ref}.
\end{proof}

\subsection{Acyclic finite type}\label{bij fin sec}  
An exchange matrix $B$ is \newword{acyclic} if, after reindexing if necessary, it has the property that $b_{ij}\le0$ whenever $i>j$.
The exchange matrix $B$ is of \newword{finite type} if every cluster algebra with $B$ as its initial exchange matrix has finitely many cluster variables.
In this section, we prove Theorems~\ref{id phen finite} and~\ref{ref phen finite}, which say that Phenomena~\ref{id phen} and~\ref{ref phen} occur whenever $B$ is acyclic and of finite type.

By \cite[Theorem~10.12]{universal}, if $B$ is of finite type (whether or not it is acyclic), then $R^B$ admits a positive basis.
Thus Propositions~\ref{positive cone basis} and~\ref{refines FB} imply the following theorem, which says that Phenomena~\ref{id phen} and~\ref{ref phen} are equivalent in finite type.  
\begin{theorem}\label{finite refines FB}  
If $B$ and $B'$ are exchange matrices of the same size and $B$ is of finite type, then $\id:\reals^B\to\reals^{B'}$ is mutation-linear if and only if $\F_B$ refines~$\F_{B'}$.
\end{theorem}

\begin{remark}\label{SH}
Several results from \cite{universal} refer to the ``Standard Hypotheses,'' which amount to an assertion called ``sign-coherence of $\c$-vectors,'' which is now a theorem of \cite{GHKK}.
Thus we ignore the Standard Hypotheses when we quote from \cite{universal}.
\end{remark}

The following theorem is established as part of the proof of \cite[Theorem~10.12]{universal}, but is unfortunately not stated separately as a numbered result in~\cite{universal}.
We need no details here about $\g$-vectors, because we only wish to concatenate Theorem~\ref{FB gvec finite} with a theorem below about Cambrian fans.

\begin{theorem}\label{FB gvec finite}
If $B$ is an exchange matrix of finite type, then $\F_B$ consists of the $\g$-vector cones associated to $B^T$ and their faces.
\end{theorem}

Given an exchange matrix $B$, replacing each $0$ on the diagonal with~$2$ and turning all positive off-diagonal entries negative results in a symmetrizable \newword{(generalized) Cartan matrix} $\Cart(B)$.
A Cartan matrix is of \newword{finite type} if and only if it is positive definite.
By results of \cite{ca2}, $\Cart(B)$ is of finite type if and only if $B$ is acyclic and of finite type.
The Coxeter element $c$ associated to $B$ and the Cambrian fan associated to $(\Cart(B),c)$ are defined in many places, including 
\cite[Section~4]{camb_fan},
\cite[Section~9]{typefree}, 
and
\cite[~Section~5]{framework}.
We need few details here, because we will only continue concatenating results.
We construct a root system $\Phi(\Cart(B))$ in $\reals^n$ in such a way that $\e_1,\ldots,\e_n$ are the simple \emph{co}-roots.
The Cambrian fan is a certain refinement of the Coxeter fan for $\Cart(B)$, the collection of hyperplanes normal to roots in $\Phi(\Cart(B))$.
If $B$ is an acyclic exchange matrix of finite type and $c$ is the associated Coxeter element, then the Cambrian fan associated to $(\Cart(B),c)$ consists of the $\g$-vector cones associated to $B$ and their faces.
(This was conjectured in \cite[Section~10]{camb_fan}, where it was proved in a special case.
It was proved in \cite[Theorem~1.10]{YZ} and later as \cite[Corollary~5.16]{framework}.)
Combining this fact with Theorem~\ref{FB gvec finite}, we obtain the following theorem.

\begin{theorem}\label{camb FB finite}
If $B$ is an acyclic exchange matrix of finite type and $c$ is the associated Coxeter element, then $\F_{B^T}$ is the Cambrian fan associated to $(\Cart(B),c)$.
\end{theorem}

A Cartan matrix $A=[a_{ij}]$ \newword{dominates} a Cartan matrix $\A'=[a'_{ij}]$ if $|a_{ij}|\ge|a'_{ij}|$ for all $i$ and $j$.
The following is \cite[Proposition~1.10]{diagram}.  

\begin{prop}\label{dom subroot}
Suppose $A$ and $A'$ are symmetrizable Cartan matrices such that $A$ dominates~$A'$.
If $\Phi(A)$ and $\Phi(A')$ are both defined with respect to the same simple roots $\alpha_i$, then $\Phi(A)\supseteq\Phi(A')$ and $\Phi_+(A)\supseteq\Phi_+(A')$.
\end{prop}

The following theorem, which is \cite[Theorem~1.11]{diagram}, relies of Proposition~\ref{dom subroot} as it applies to the dual root system $\Phi\ck(A)$ to $\Phi(A)$.

\begin{theorem}\label{camb fan coarsen}
Suppose $A$ and $A'$ are Cartan matrices such that $A$ dominates $A'$ and suppose $W$ and $W'$ are the associated groups, both generated by the same set $S$.
Suppose $c$ and $c'$ are Coxeter elements of $W$ and $W'$ respectively that can be written as a product of the elements of $S$ in the same order.
Choose a root system $\Phi(A)$ and a root system $\Phi(A')$ so that the simple \emph{co}-roots are the same for the two root systems.
Construct the Cambrian fan for $(A,c)$ by coarsening the fan determined by the Coxeter arrangement for $\Phi(A)$ and construct the Cambrian fan for $(A',c')$ by coarsening the fan determined by the Coxeter  arrangement for $\Phi(A')$.
Then the Cambrian fan for $(A,c)$ refines the Cambrian fan for $(A',c')$.
Whereas the codimension-$1$ faces of the Cambrian fan for $(A,c)$ are orthogonal to co-roots (i.e.\ elements of $\Phi\ck(A)$), the Cambrian fan for $(A',c')$ is obtained by removing all codimension-$1$ faces orthogonal to elements of $\Phi\ck(A)\setminus\Phi\ck(A')$.  
\end{theorem}

The orthogonality appearing in Theorem~\ref{camb fan coarsen} and below in Theorem~\ref{FB coarsen} is the standard inner product on $\reals^n$ defined in terms of the basis of the $\e_i$ and has nothing to do with the Euclidean inner products associated to $A$ and $A'$.

If $B$ dominates $B'$, then $\Cart(B)$ dominates $\Cart(B')$.
Thus we can translate Theorem~\ref{camb fan coarsen} into a more detailed version of Theorem~\ref{id phen finite}.
Theorem~\ref{camb FB finite} identifies the Cambrian fan as $\F_{B^T}$ but by Proposition~\ref{antipodal FB}, we can equally well consider the relationship between $\F_{-B^T}$ and $\F_{-(B')^T}$.
Passing from $B$ to $-B^T$ corresponds to passing from $\Cart(B)$ to $\Cart(B)^T$ while preserving $c$, and transposing the Cartan matrix switches the roles of $\Phi$ and $\Phi\ck$.
Thus in rewriting Theorem~\ref{camb fan coarsen} as the following theorem, we identify simple \emph{roots} with the $\e_i$ rather than simple \emph{co-roots}.

\begin{theorem}\label{FB coarsen}
Suppose $B$ and $B'$ are acyclic exchange matrices of finite type such that $B$ dominates $B'$ and let $A=\Cart(B)$ and $A'=\Cart(B')$.
Realize the root systems $\Phi(A)$ and $\Phi(A')$ in $\reals^n$ such that for each $i$, the simple root $\alpha_i$ for~$A$, the simple root $\alpha'_i$ for $A'$ and the unit basis vector $\e_i$ all coincide.
Each codimension\nobreakdash-$1$ face of $\F_B$ is orthogonal to an element of $\Phi(A)$, and $\F_{B'}$ is obtained from $\F_B$ by removing all codimension-$1$ faces orthogonal to roots in $\Phi(A)\setminus\Phi(A')$.
\end{theorem}

Theorem~\ref{FB coarsen} completes the proof of Theorem~\ref{ref phen finite}.
Theorem~\ref{id phen finite} follows immediately by Proposition~\ref{refines FB}.

\subsection{Resection of marked surfaces}\label{bij surf sec}
We next show that rational versions of Phenomena~\ref{id phen} and~\ref{ref phen} occur in some cases where $B$ is the signed adjacency matrix of a surface.
For the sake of brevity, we refrain from repeating here all of the background material found in~\cite{unisurface}.
(See also the seminal works \cite{cats1,cats2} on cluster algebras and surfaces.)
We do, however, give the most basic background.
Unjustified assertions in this basic background are established in \cite{cats1}.

Let $\S$ be a surface obtained from a compact, oriented surface without boundary by deleting a finite collection of open disks whose closures are pairwise disjoint.
Usually, $\S$ is required to be connected, but this is merely an ``irreducibility'' criterion and is not essential for what we do here.
Indeed, since we will consider an operation on $\S$ that may disconnect it, we must allow disconnected surfaces.
However, compactness implies that $\S$ has finitely many components.
The boundaries of the removed disks are called \newword{boundary components}.
Choose a finite set $\M$ of points in $\S$ called \newword{marked points}.
Each boundary component must contain at least one marked point.
The marked points in the interior of $\S$ are called \newword{punctures}.
The marked points on boundary components cut the boundary components into curves called \newword{boundary segments}.  
None of the connected components of $(\S,\M)$ may be unpunctured monogons, unpunctured digons, unpunctured triangles, 
or spheres with fewer than $4$ punctures.
Typically, one excludes once-punctured monogons as well.
Here, we allow once-punctured monogons, but in many of the definitions that follow, we have to single out once-punctured monogons as a special case.

An \newword{arc} in $(\S,\M)$ is a curve in $\S$ with endpoints in $\M$, with the following restrictions:
An arc may not intersect itself, except that its endpoints may coincide.
An arc must be disjoint from $\M$ and from the boundary of $\S$, except at its endpoints.
An arc may not bound an unpunctured monogon.
Finally, we disallow arcs that define, together with a single boundary segment connecting the endpoints, an unpunctured digon.
We consider arcs up to isotopy relative to $\M$.

Arcs $\alpha$ and $\gamma$ are \newword{incompatible} if they intersect in $\S\setminus\M$, and if the intersection cannot be removed by (independently) isotopically deforming the arcs.
If two arcs are not incompatible, they are \newword{compatible}.
A \newword{triangulation} of $(\S,\M)$ is a maximal collection of distinct pairwise compatible arcs.
Each $(\S,\M)$ admits at least two triangulations, 
except a once-punctured monogon, which admits exactly one triangulation, consisting of the unique arc.
Every triangulation of a given marked surface has the same number of arcs.
These arcs divide $\S$ into \newword{triangles}, each of which has $1$, $2$, or $3$ distinct vertices and $2$ or $3$ distinct sides.
A \newword{self-folded triangle} is a triangle with 2 distinct sides.
As a surface in its own right, a self-folded triangle is a once-punctured monogon with an arc $\alpha$ connecting the vertex of the monogon to the puncture.
From the interior, this appears as a triangle, but two of the sides of the triangle are $\alpha$.
We call $\alpha$ the \newword{fold edge} of the self-folded triangle.

The \newword{signed adjacency matrix} of a triangulation $T$ of $(\S,\M)$ is a matrix $B(T)=[b_{\alpha\beta}]_{\alpha,\beta\in T}$ indexed by the arcs of $T$.
The definition of the entries $b_{\alpha,\beta}$ is complicated by the possible presence of self-folded triangles.
If $T$ has no self-folded triangles, one can simplify what follows by taking $\pi_T$ to be the identity map.
In general however, we define a map $\pi_T$ on the arcs of $T$, fixing each arc except for the arcs that are fold edges of self-folded triangles.
If $\alpha$ is the fold edge of a self-folded triangle in $T$, then $\pi_T(\alpha)$ is the other edge of the triangle.
The entry $b_{\alpha\beta}$ is the sum $\sum_\Delta b^\Delta_{\alpha\beta}$ over all triangles $\Delta$ of $T$ which are \emph{not self-folded} of the quantities
\[b^\Delta_{\alpha\beta}=
\begin{cases} 
\,\,\,\,\,1&\text{if $\Delta$ has sides $\pi_T(\alpha)$ and $\pi_T(\beta)$ such that $\pi_T(\alpha)$} \\[-2pt]&\quad\text{is immediately followed by $\pi_T(\beta)$ in \emph{clockwise} order.}\\[3pt]
\,-1&\text{if $\Delta$ has sides $\pi_T(\alpha)$ and $\pi_T(\beta)$ such that $\pi_T(\alpha)$} \\[-2pt]&\quad\text{is immediately followed by $\pi_T(\beta)$ in \emph{counterclockwise} order.}\\[3pt]
\,\,\,\,\,0&\text{otherwise.}
\end{cases}
\]
The matrix $B(T)$ is a skew-symmetric integer matrix (an exchange matrix).
If $\S$ has multiple connected components, then $B(T)$ has a block-diagonal form with a diagonal block for each component.
The matrix $B(T)$ is the zero matrix for some $T$ if and only if $B(T)$ is the zero matrix for every $T$, if and only if every connected component of $(\S,\M)$ is a \newword{null surface} (a once-punctured monogon, once-punctured digon, or unpunctured quadrilateral).

In general, one considers \emph{tagged} arcs and \emph{tagged} triangulations.
In this section, we appeal to \cite[Proposition~12.3]{cats1}, which says that every exchange matrix arising from a tagged triangulation can also be obtained from an ordinary triangulation, and we put off defining tagged arcs  and tagged triangulations until Section~\ref{resect hom sec}.

An \newword{allowable curve} is a curve in $\S$ of one of the following forms:
\begin{itemize}
\item a closed curve;
\item a curve whose two endpoints are unmarked boundary points,
\item a curve having one endpoint an unmarked boundary point, with the other end spiraling (clockwise or counterclockwise) into a puncture, or
\item a curve spiraling into (not necessarily distinct) punctures at both ends.
\end{itemize}
However, an allowable curve may not
\begin{itemize}
\item have any self-intersections,
\item be contractible in $\S\setminus\M$,
\item be contractible to a puncture,
\item have two endpoints on the boundary and be contractible to a portion of the boundary containing zero or one marked points.
\item have both endpoints are the same boundary segment and cut out, together with the portion of the boundary between its endpoints, a once-punctured disk, or
\item have two coinciding spiral points and cut out a once-punctured disk.
\end{itemize}
We consider allowable curves up to isotopy relative to $\M$.

Essentially, two allowable curves are \newword{compatible} if they are non-intersecting, but one kind of intersection is allowed.
Suppose two allowable curves are identical except that, at one (and only one) end of the curve, they spiral opposite directions into the same point.
Then these two curves are compatible unless they are contained in a component of $(\S,\M)$ that is a once-punctured monogon.
(There are exactly two allowable curves in a once-punctured monogon, and by convention the two are incompatible.)
A \newword{(rational) quasi-lamination} is a collection of pairwise compatible allowable curves and an assignment of a positive rational weight to each curve.
We require the curves in the quasi-lamination to be distinct up to isotopy.
The set of curves that appear in a quasi-lamination $L$ is called the \newword{support} of $L$ and written $\supp(L)$.

Given a quasi-lamination $L$ and a triangulation $T$,  the \newword{shear coordinate of $L$ with respect to $T$} is a vector $\b(T,L)=(b_\gamma(T,L):\gamma\in T)$ indexed by the arcs $\gamma$ in $T$.
We define $b_\gamma(T,L)$ to be the sum $\sum_{\lambda\in L}w_\lambda b_\gamma(T,\lambda)$ over the curves $\lambda$ in $L$, where $w_\lambda$ is the weight of $\lambda$ in $L$ and $b_\gamma(T,\lambda)$ is defined as follows.

First, suppose $\gamma$ is contained in two distinct triangles of $T^\circ$ (rather than being the fold edge of a self-folded triangle).
Since $\lambda$ is defined up to isotopy, we can assume that each time $\lambda$ crosses an arc $\alpha$ of $T$, it does not cross $\alpha$ again in the opposite direction until it has crossed some other arc $\beta$ of $T$.
(In the case where $\lambda$ spirals into a puncture that is incident to only one arc $\alpha$, it crosses $\alpha$ repeatedly, but always in the same direction.)
The quantity $b_\gamma(T,\lambda)$ is the sum, over each intersection of $\lambda$ with $\gamma$, of a number in $\set{-1,0,1}$ that depends on which other arcs (or boundary segments) $\lambda$ visits immediately before and after the intersection with~$\gamma$.
(The other arcs or boundary segments are the other edges of the two triangles containing $\alpha$.)
\begin{figure}[ht]
\scalebox{0.8}{\raisebox{32 pt}{$+1$}\,\,\includegraphics{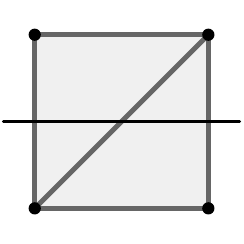}\qquad\qquad\includegraphics{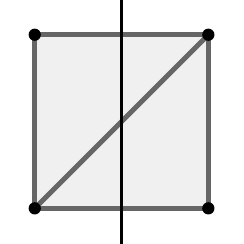}\raisebox{32 pt}{$-1$}}
\caption{Computing shear coordinates}
\label{shear fig}
\end{figure}
Figure~\ref{shear fig} shows the situations in which this number is $1$ or $-1$.
It is $0$ otherwise.
In these pictures, $\gamma$ is the diagonal of the square and $\lambda'$ is the vertical or horizonal line intersecting the square.
The quadrilaterals represented in Figure~\ref{shear fig} might have fewer than $4$ distinct vertices and $4$ distinct sides.
As a concrete example, opposite sides of the square may be identified to make a torus.
Or, one of the triangles may be self-folded (with folding edge not equal to $\gamma$). 

Next, consider the case where $\gamma$ \emph{is} the fold edge of a self-folded triangle and let $p$ be the marked point that is incident only to $\gamma$ and not to any other arc in the triangulation.
If $\lambda$ spirals into $p$, then let $\lambda'$ be the allowable curve obtained from $\lambda$ by reversing the direction of the spiral into $p$.
Otherwise let $\lambda'=\lambda$.
Let $\gamma'$ be the other edge (besides $\gamma$) of the self-folded triangle.
Define $b_\gamma(T,\lambda)=b_{\gamma'}(T,\lambda')$.
We can do this unless $\gamma$ is contained in a component of $(\S,\M)$ that is a once-punctured monogon, in which case $\gamma'$ is a boundary segment.
In this case, we define $b_\gamma(T,\lambda)$ to be $+1$ if $\lambda$ spirals into the puncture counterclockwise or $-1$ if it spirals in clockwise.

The following is \cite[Theorem~4.4]{unisurface}, a rephrasing of \cite[Theorem~13.6]{cats2}.
(In \cite[Theorem~4.4]{unisurface}, once-punctured monogons are not allowed, but the proof for once-punctured monogons is easy.)
Recall that, as sets, $\rationals^{B(T)}=\rationals^n$ when $|T|=n$.

\begin{theorem}\label{q-lam bij}
Fix a tagged triangulation $T$.
Then the map $L\mapsto\b(T,L)$ is a bijection between rational quasi-laminations and $\rationals^{B(T)}$.
\end{theorem}

A \newword{tangle} is any collection of allowable curves (with no requirement of compatibility) and an assignment of an integer weight to each curve (with no requirement of nonnegativity).
The curves must be distinct up to isotopy.
A tangle $\Xi$ can be given shear coordinates $\b(T,\Xi)$ just as a quasi-lamination can:  as a weighted sum of the shear coordinates of each curve in $\Xi$.
A \newword{null tangle} in $(\S,\M)$ is a tangle $\Xi$ with $\b(T,\Xi)=\0$ for all tagged triangulations $T$.
(Recall that we have not defined tagged triangulations.
Since we won't work closely with the Null Tangle Property in this paper, we can safely continue to put off the definition of tagged triangulations until Section~\ref{resect hom sec}.)
For technical reasons, the definition is different if $(\S,\M)$ has no boundary components and exactly one puncture.
In that case, a null tangle is a tangle $\Xi$ such that $\b(T,\Xi)$ is the zero vector for all ordinary triangulations $T$ with all tags plain.
A tangle is \newword{trivial} if all of its curves have weight zero.
We say $(\S,\M)$ has the \newword{Null Tangle Property} if every null tangle in $(\S,\M)$ is trivial.

The Null Tangle Property allows us to construct a positive basis for $\rationals^{B(T)}$ using allowable curves.
The following is \cite[Theorem~7.3]{unisurface}, restricted to ordinary triangulations (as opposed to a tagged triangulations) but allowing disconnected surfaces, with a minor correction (as explained in Remark~\ref{null tangle correction}, below).

\begin{theorem}\label{null tangle theorem}
Suppose $R$ is $\integers$ or $\rationals$ and let $T$ be a triangulation of $(\S,\M)$.
If no component of $(\S,\M)$ is a null surface, then the following are equivalent:
\begin{enumerate}[(i)]
\item \label{null tangle theorem null}
$(\S,\M)$ has the Null Tangle Property.
\item \label{null tangle theorem basis}
The shear coordinates of allowable curves form a basis for $\rationals^{B(T)}$.
\item \label{null tangle theorem pos basis}
The shear coordinates of allowable curves form a positive basis for $\rationals^{B(T)}$.
(By Proposition~\ref{positive cone basis}, this is also a cone basis for $\rationals^{B(T)}$.)
\end{enumerate}
If some component of $(\S,\M)$ is a null surface, then \eqref{null tangle theorem null} fails, but \eqref{null tangle theorem basis} and~\eqref{null tangle theorem pos basis} hold.
\end{theorem}

\begin{remark}\label{null tangle correction}
Theorem~\ref{null tangle theorem} should be taken as a correction to \cite[Theorem~7.3]{unisurface}.
Once-punctured monogons are not allowed in \cite{unisurface}, and \cite[Theorem~7.3]{unisurface} takes a \emph{tagged} triangulation $T$, but \cite[Theorem~7.3]{unisurface} should be corrected by dealing separately with the quadrilateral and digon as in Theorem~\ref{null tangle theorem}.
The assertions of Theorem~\ref{null tangle theorem} for null surfaces are easy exercises.
The correction to \cite[Theorem~7.3]{unisurface} contradicts \cite[Theorem~7.4]{unisurface} for the quadrilateral and digon, but does so in a way that preserves the truth of \cite[Corollary~7.5]{unisurface}, and therefore preserves all of the results of the paper on universal coefficients.
The error arose because the proof of \cite[Proposition~7.9]{unisurface} used \cite[Proposition~2.3]{unisurface}, whose hypotheses rule out the quadrilateral and digon (despite the assertion in the paragraph before \cite[Proposition~3.5]{unisurface}).
The main constructions and results of \cite{unisurface} are all valid, even for the quadrilateral and digon, but some some auxiliary results are wrong as stated for the quadrilateral and digon.
\end{remark}

We say $(\S,\M)$ has the \newword{Curve Separation Property} if, given incompatible allowable curves $\lambda$ and~$\lambda'$, there exists a tagged triangulation $T$ and an arc $\gamma\in T$ such that $b_\gamma(T,\lambda)$ and $b_\gamma(T,\lambda')$ have strictly opposite signs.
(Again, for technical reasons, the definition is different if $(\S,\M)$ has no boundary components and exactly one puncture.
In this case we require that $T$ be an ordinary triangulation.)
The Null Tangle Property implies the Curve Separation Property \cite[Corollary~7.14]{unisurface}.
Furthermore, every null surface satisfies the Curve Separation Property.
(We know of no surfaces that fail the Curve Separation Property, but it remains unproven for surfaces with no boundary components and exactly one puncture.)

The Curve Separation Property allows us to understand the rational part of the mutation fan $\F_B$.
(See Definition~\ref{rat part def}.)
Given a triangulation $T$ of $(\S,\M)$, the \newword{rational quasi-lamination fan} $\F_\rationals(T)$ is a fan with one cone for each set $\Lambda$ of pairwise compatible allowable curves.
Specifically, the cone associated to $\Lambda$ is $\Span_{\reals_{\ge0}}\set{\b(T,\lambda):\lambda\in\Lambda}$.
Thus the relative interior of that cone consists of the shear coordinates (with respect to $T$) of rational laminations whose support is $\Lambda$.
The following is a special case of \cite[Theorem~4.10]{unisurface}.

\begin{theorem}\label{rat FB surfaces}
If $T$ is a triangulation of $(\S,\M)$, then $\F_\rationals(T)$ is a rational simplicial fan.
It is the rational part of $\F_{B(T)}$ if and only if $(\S,\M)$ has the Curve Separation Property.
\end{theorem}

We now introduce an operation, called resection, on marked surfaces that induces a dominance relation on signed adjacency matrices.

\begin{definition}[\emph{Resecting a surface at an arc}]\label{resect surface def}
The resection operation on surfaces is illustrated in Figure~\ref{resect alpha}.
\begin{figure}[ht]
\includegraphics{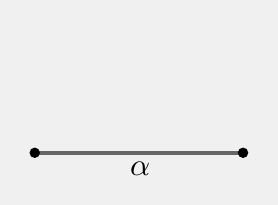}\qquad\raisebox{24pt}{\LARGE$\longrightarrow$}\qquad\includegraphics{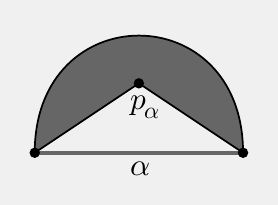}
\caption{Resection of a surface}
\label{resect alpha}
\end{figure}
Suppose $\alpha$ is an arc in $(\S,\M)$ connecting marked points $p_1$ and $p_2$.
Place a new marked point $p_\alpha$ in the interior of $\S$ close to $\alpha$.
Draw a curve in $\S$ connecting $p_1$ to $p_2$ and forming a digon with $p_\alpha$ in its interior, but containing no other marked points.
Draw two more curves inside the digon, one connecting $p_1$ to~$p_\alpha$ and the other connecting $p_2$ to $p_\alpha$.
This cuts the digon into two triangles.
Remove the interior of the one that does not have $\alpha$ as an edge.
If $p_1$ is on a boundary segment, then resulting surface must be cut at the point $p_1$ in order to satisfy the requirement that the boundary components be circles.
In this case $p_1$ becomes two marked points.
Similarly, if $p_2$ is on a boundary segment, then the surface must be cut at the point $p_2$.
The resulting marked surface is a \newword{resection of $(\S,\M)$ at $\alpha$}, and accordingly we use the verb ``resect'' and the noun ``resection'' to describe passing from $(\S,\M)$ to the resected surface.
The resected surface $\S'$ may be disconnected even if $\S$ is connected.

Typically, there are two possible resections at $\alpha$, one for each side of $\alpha$.
However, we disallow resections that create a component that is an unpunctured triangle.
\end{definition}

\begin{definition}[\emph{Resection at a collection of arcs}]\label{resect collection def}
More generally, a \newword{resection of $(\S,\M)$} is a marked surface obtained by performing any collection of resections at arcs, including possibly resecting the same arc on both sides.
This is well-defined up to isotopy as long as the arcs in question are compatible.
\end{definition}

\begin{definition}[\emph{Resection compatible with a triangulation}]\label{resect compat tri def}  
Fix a triangulation~$T$ of $(\S,\M)$.
A resection of $(\S,\M)$ is \newword{compatible with $T$} if
\begin{itemize}
\item Each arc that is resected is an arc of $T$.
\item For each arc $\alpha$ that is resected, the point $p_\alpha$ is placed in the interior of a triangle of $T$ bounded by $\alpha$ and the new curves that define the resection never leave the interior of that triangle except possibly at their endpoints.
\item If an arc $\alpha$ that is resected has an endpoint at a puncture $q$, then either both endpoints of $\alpha$ are at $q$ or at least one more arc $\beta$ with an endpoint at $q$ is resected.
In the latter case, we require that $p_\alpha$ and $p_\beta$ are in different triangles of $T$.
\end{itemize}
These requirements imply in particular that a resected arc $\alpha$ may not be the fold edge of a self-folded triangle of $T$.
Equivalently, a resected arc $\alpha$ is contained in two distinct triangles of $T$.
The resection must also satisfy the requirements of Definition~\ref{resect surface def} by not cutting off an unpunctured triangle.
\end{definition}

\begin{definition}[\emph{Triangulation induced on the resected surface}]\label{resect induce tri def}
If $(\S',\M')$ is a resection of $(\S,\M)$ that is compatible with $T$, then each arc in $T$ is also an arc in $(\S',\M')$.
These arcs cut $(\S',\M')$ into triangles, defining a triangulation $T'$ of $(\S',\M')$ called the \newword{triangulation induced on $(\S',\M')$ by $T$}.
\end{definition}

%
\begin{prop}\label{resect signed adjacency}
Given a marked surface $(\S,\M)$ with a triangulation $T$, perform a resection of $(\S,\M)$ compatible with $T$ and let $T'$ be the triangulation induced by~$T$ on the resected surface.
Then $B(T)$ dominates $B(T')$.
\end{prop}
\begin{proof}
We will prove two assertions that together amount to a version of the proposition with weaker hypotheses.

First, if $\alpha$ is an arc in $T$ not incident to the puncture in a once-punctured digon and $T'$ is the triangulation obtained by resecting at $\alpha$ according to the rule in the second bullet point in Definition~\ref{resect compat tri def}, then $B(T)$ dominates $B(T')$.
The part of $T$ right around $\alpha$ is illustrated in Figure~\ref{resect config}.
\begin{figure}[ht]
\scalebox{0.8}{\includegraphics{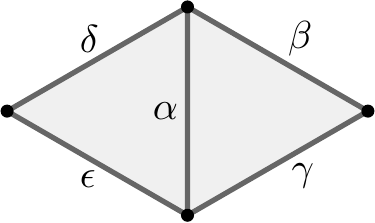}}\quad\raisebox{23pt}{\LARGE$\longrightarrow$}\quad\scalebox{0.8}{\includegraphics{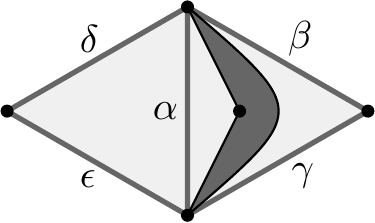}}
\caption{An illustration for the proof of Proposition~\ref{resect signed adjacency}}
\label{resect config}
\end{figure}
Since $\alpha$ is not incident to the puncture in a once-punctured digon, $\beta\neq\delta$ and $\gamma\neq\epsilon$, and $\alpha$ is distinct from the other four arcs.
The signed adjacency matrix $B(T')$ is obtained from $B(T)$ as follows.
If $\beta$ is an arc (rather than a boundary segment), then decrease the $\alpha\beta$-entry by 1 and increase the $\beta\alpha$-entry by $1$.
If $\gamma$ is an arc, then increase the $\alpha\gamma$-entry by 1 and decrease the $\gamma\alpha$-entry by $1$.
The situation may be sightly more complicated because $\beta$ may be the non-fold edge of a self-folded triangle.
If $\beta'$ is the fold edge of that triangle, then all entries of $B(T)$ and $B(T')$ indexed by $\beta'$ agree with corresponding entries indexed by $\beta$.
Similarly, $\gamma$ may be the non-fold edge of a self-folded triangle, or both $\beta$ and $\gamma$ may be.

Second, if $\alpha$ and $\beta$ are distinct arcs in $T$, both incident to the puncture in a once-punctured digon and $T'$ is obtained by resecting at $\alpha$ and at $\beta$ according to the rule in the second bullet point in Definition~\ref{resect compat tri def} and with $p_\alpha$ and $p_\beta$ in different triangles of $T$, then $B(T)$ dominates $B(T')$.
The situation is illustrated in Figure~\ref{digon resect config}.
\begin{figure}[ht]
\scalebox{0.55}{\includegraphics{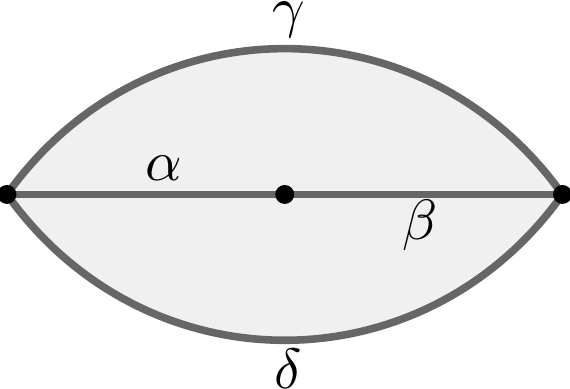}}\quad\raisebox{32pt}{\LARGE$\longrightarrow$}\quad\scalebox{0.55}{\includegraphics{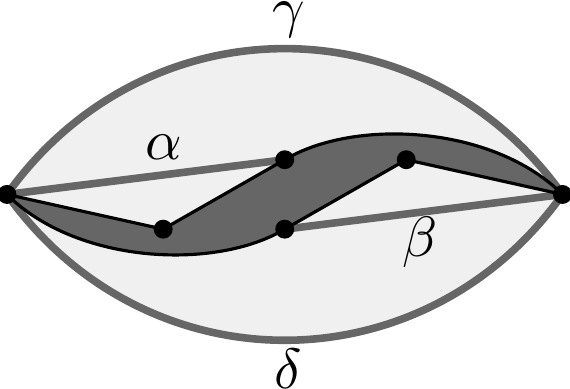}}
\caption{Another illustration for the proof of Proposition~\ref{resect signed adjacency}}
\label{digon resect config}
\end{figure}
The $\alpha\beta$-entry in the signed adjacency matrix is 0 before and after resection.
The arcs $\delta$ and $\gamma$ do not coincide, because if they do, the unresected surface is a sphere with only $3$ punctures.
Thus if $\gamma$ is an arc, resection decreases the absolute value of the $\beta\gamma$ and $\gamma\beta$ entries by $1$, and if $\delta$ is an arc, resection decreases the absolute value of the $\alpha\delta$ and $\delta\alpha$ entries by $1$.
If neither $\gamma$ nor $\delta$ is an arc, then the signed adjacency matrix is unchanged.
Either $\gamma$ or $\delta$ or both may be non-fold edges of self-folded triangles, but the dominance relation still holds, as in the previous case.
\end{proof}

\begin{example}\label{resect ex}
Figure~\ref{resect ex fig} shows two resections (each at a single arc).
The first example begins with  a torus, obtained by identifying opposite edges of a square, with one puncture at the corners of the square.
The torus is resected at the arc~$\alpha$ to obtain an annulus.
The annulus is then resected at $\gamma$ to obtain a hexagon.
\end{example}

\begin{figure}
\begin{center}
\begin{tabular}{ccccc}
\,\,\,\,\,\,\includegraphics{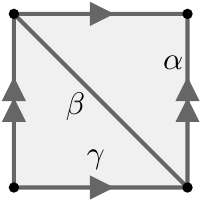}&\raisebox{27 pt}{\LARGE$\longrightarrow$}&\includegraphics{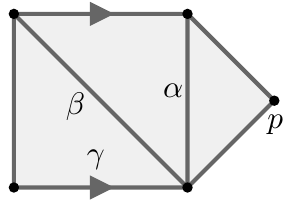}&\raisebox{27 pt}{\LARGE$=$}\quad&\raisebox{-15 pt}{\scalebox{0.75}{\includegraphics{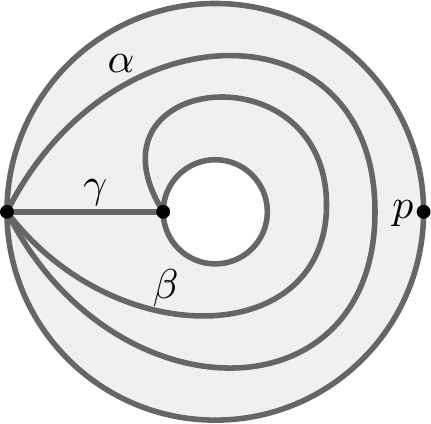}}}\\
\!$\begin{array}{c}
\mbox{\hspace{13 pt}\tiny $\alpha$ \hspace{3 pt} $\beta$\hspace{5.5 pt} $\gamma$}\\[0 pt]
\begin{smallmatrix*}[r]\\\alpha\\\beta\\\gamma\end{smallmatrix*}\begin{bsmallmatrix*}[r]
0&2&-2\\
-2&0&2\\
2&-2&0
\end{bsmallmatrix*}
\end{array}$
&&&&
\!\!\!\!$\begin{array}{c}
\mbox{\hspace{13 pt}\tiny $\alpha$ \hspace{3 pt} $\beta$\hspace{5.5 pt} $\gamma$}\\[0 pt]
\begin{smallmatrix*}[r]\\\alpha\\\beta\\\gamma\end{smallmatrix*}\begin{bsmallmatrix*}[r]
0&1&-1\\
-1&0&2\\
1&-2&0
\end{bsmallmatrix*}
\end{array}$
\end{tabular}\vspace{30pt}
\begin{tabular}{ccccc}
\raisebox{-15 pt}{\scalebox{0.75}{\includegraphics{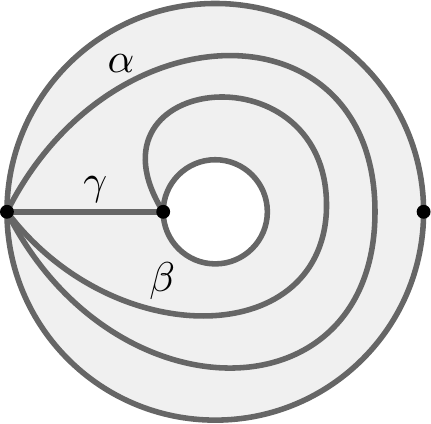}}}&\!\raisebox{27 pt}{\LARGE$\longrightarrow$}\!&\raisebox{-15 pt}{\scalebox{0.75}{\includegraphics{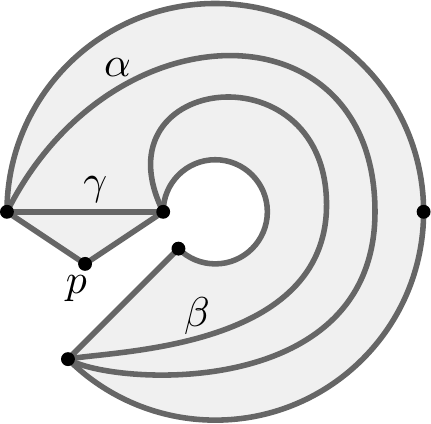}}}&\raisebox{27 pt}{\LARGE$=$}\!\!&\raisebox{-12 pt}{\scalebox{0.8}{\includegraphics{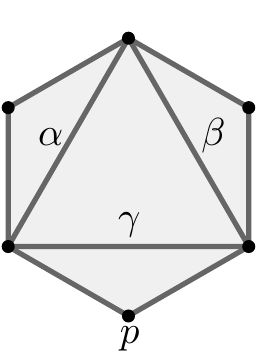}}}\\
\!\!\!\!$\begin{array}{c}
\mbox{\hspace{13 pt}\tiny $\alpha$ \hspace{3 pt} $\beta$\hspace{5.5 pt} $\gamma$}\\[0 pt]
\begin{smallmatrix*}[r]\\\alpha\\\beta\\\gamma\end{smallmatrix*}\begin{bsmallmatrix*}[r]
0&1&-1\\[1pt]
-1&0&2\\[2pt]
1&-2&0
\end{bsmallmatrix*}
\end{array}$
&&&&
\!\!\!\!$\begin{array}{c}
\mbox{\hspace{13 pt}\tiny $\alpha$ \hspace{3 pt} $\beta$\hspace{5.5 pt} $\gamma$}\\[0 pt]
\begin{smallmatrix*}[r]\\\alpha\\\beta\\\gamma\end{smallmatrix*}\begin{bsmallmatrix*}[r]
0&1&-1\\[1pt]
-1&0&1\\[2pt]
1&-1&0
\end{bsmallmatrix*}
\end{array}$
\end{tabular}
\end{center}
\caption{Examples of resection}
\label{resect ex fig}
\end{figure}

\begin{remark}\label{not every dom}
Not every exchange matrix dominated by $B(T)$ can be obtained by resecting the associated surface (even for the broader class of resections allowed in the proof of Proposition~\ref{resect signed adjacency}).
Not even every skew-symmetric exchange matrix dominated by $B(T)$ can be obtained.
This is because, in some cases, resection must change four or more entries of $B(T)$.
For example, in the labelling shown in Figure~\ref{resect config}, when neither $\beta$ nor $\gamma$ is a boundary segment, the entries $b_{\alpha\beta}$, $b_{\alpha\gamma}$, $b_{\beta\alpha}$, and $b_{\gamma\alpha}$ are all changed.
In this case, it is impossible to change only the entries $b_{\alpha\beta}$ and $b_{\beta\alpha}$, even though the result would be an exchange matrix dominated by $B(T)$.
\end{remark}

We next prove Theorems~\ref{id phen resect} and~\ref{ref phen resect}, starting with the latter.

\begin{proof}[Proof of Theorem~\ref{ref phen resect}]
Theorem~\ref{q-lam bij} says that the maps $L\mapsto\b(T,L)$ and $L'\mapsto\b(T',L')$ are bijections from rational quasi-laminations (on $(\S,\M)$ and $(\S',\M')$ respectively) to $\rationals^n$.
Thus there is a bijection from rational quasi-laminations $L$ on $(\S,\M)$ to rational quasi-laminations $L'$ on $(\S',\M')$ such that $\b(T,L)=\b(T',L')$.

For each $C\in\F_\rationals(T)$, the set $C\cap\rationals^n$ is the set of shear coordinates of all nonnegative rational weightings on some collection $\Lambda$ of pairwise-compatible allowable curves in $(\S,\M)$.
We prove the theorem by showing that, for each such $\Lambda$, there is a collection $\Lambda'$ of pairwise-compatible allowable curves in $(\S',\M')$ such that for any nonnegative rational weighting on $\Lambda$, there is a nonnegative rational weighting on $\Lambda'$ giving the same shear coordinates.
Thus $C\cap\rationals^n$ is contained in the cone of $\F_\rationals(T')$ consisting of shear coordinates of nonnegative rational weightings on $\Lambda'$.

We can almost, but not quite, construct $\Lambda'$ by resecting at one arc at a time; we will need to appeal to the third condition in Definition~\ref{resect compat tri def}.
Given a resection at~$\alpha$, let $\alpha\beta\gamma$ be the triangle where the point $p_\alpha$ is placed in the process of resecting at~$\alpha$, as illustrated in Figure~\ref{situated}.
\begin{figure}[ht]
\scalebox{0.8}{\includegraphics{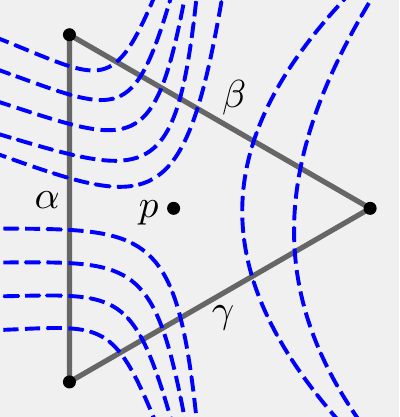}}\qquad\quad\scalebox{0.8}{\includegraphics{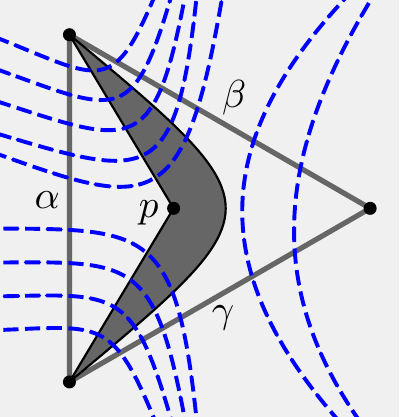}}\qquad\quad\scalebox{0.8}{\includegraphics{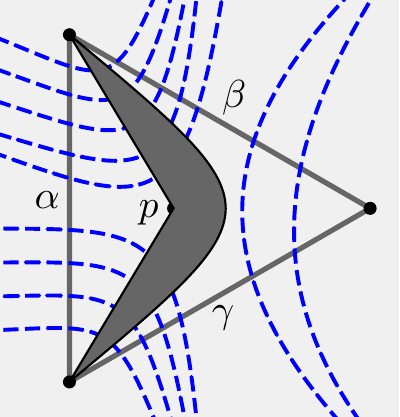}}
\caption{The identity map in terms of quasi-laminations}
\label{situated}
\end{figure}
The arc $\alpha$ cannot coincide with $\beta$ or with $\gamma$, because if so, then $\alpha$ has an endpoint incident to only one arc, violating the third condition in Definition~\ref{resect compat tri def}.
Possibly $\beta$ and $\gamma$ coincide, so that this is a self-folded triangle, but for now we assume not.

Assume as in the definition of shear coordinates that pairwise compatible isotopy representatives of the curves in $\Lambda$ have been chosen so that no curve crosses an arc and then immediately doubles back to cross in the opposite direction. 
Consider all intersections of the curves in $\Lambda$ with the triangle $\alpha\beta\gamma$.
A single curve in the quasi-lamination may intersect this triangle many times, and even infinitely many times if it spirals into a vertex of the triangle.

In the first step of constructing $\Lambda'$ from $\Lambda$, we only alter the curves inside the triangle.
Up to isotopy of curves in $(\S,\M)$, we can assume that no curve connecting $\alpha$ and $\beta$ (inside the triangle) separates the point $p$ from the arc $\gamma$.
Making that assumption for all cyclic permutations of $\alpha$, $\beta$, and $\gamma$, we can take $p$ to be situated with respect to the curves as illustrated in the left picture of Figure~\ref{situated}.
Furthermore, we can assume that the curves intersect the new boundary component as illustrated in the middle picture of Figure~\ref{situated}.
In particular, the boundary component does not intersect any of the curves connecting $\beta$ to $\gamma$.
Now remove from each curve its intersection with the new boundary component, as illustrated in the right picture of Figure~\ref{situated}.
This cuts many curves in $\Lambda$ into smaller pieces.
The resulting collection $\Lambda'$ of curves may not be a quasi-lamination because it may contain some curves that are not allowable (hereafter, ``bad curves'') and because we have not yet checked that the curves are pairwise compatible.
However, if each piece inherits a weight from the original weighted collection and we compute shear coordinates of the new weighted collection, including the bad curves, we obtain again the shear coordinates of the original collection.

We have constructed $\Lambda'$ so that each bad curve in $\Lambda'$ fits one of the following two descriptions:
\begin{itemize}
\item It has two endpoints on a boundary segment and is contractible to a portion of the boundary containing one marked point.
\item It has both endpoints on the same boundary segment and, with the portion of the boundary between its endpoints, cuts out a once-punctured disk.
\end{itemize}
The first type of bad curve can occur infinitely many times when one or more curves in $\Lambda$ spirals around an endpoint of $\alpha$.
Each such curve has all shear coordinates zero, we can delete all such curves from $\Lambda'$ without changing the shear coordinates.
The second type of bad curve makes a nonzero contribution to the shear coordinates of~$\Lambda'$.
We replace each such bad curve by two curves that start on the same boundary component as the bad curve and spiral in opposite directions around the puncture that the bad curve encloses, as illustrated in Figure~\ref{replace fig}.
\begin{figure}[ht]
\scalebox{1}{\includegraphics{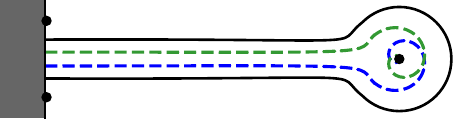}}
\caption{Replacing a bad curve with two compatible allowable curves}
\label{replace fig}
\end{figure}
(The bad curve is the solid curve, and the others are dashed.)
After these modifications, distinct curves in $\Lambda'$ may coincide up to isotopy, but if so, we delete repetitions of curves and adjust weights accordingly.

By construction, $\Lambda'$ has the same shear coordinates as $\Lambda$, and it remains only to show that the curves in $\Lambda'$ are pairwise compatible.
If all curves in $\Lambda$ are pairwise non-intersecting, then all pairs of curves in $\Lambda'$ are either non-intersecting or are compatible because they arose as in Figure~\ref{replace fig}.
Consider two curves in $\Lambda$ that are compatible because they are identical expect for spiraling in the opposite direction at exactly one of their endpoints.
If they spiral into a point other than an endpoint of $\alpha$, then they are cut into one or more curves in $\Lambda'$ that remain compatible.
If they spiral into an endpoint of $\alpha$, but not just after passing through the triangle $\alpha\beta\gamma$, then they are cut into infinitely many pieces, most of which are discarded, and the remaining pieces are non-intersecting.
However, if they pass through the triangle $\alpha\beta\gamma$ just before spiraling into an endpoint of $\alpha$ (and if we forget the third condition in Definition~\ref{resect compat tri def}), then the resulting non-discarded pieces may intersect, as illustrated in the left two pictures of Figure~\ref{situated spiral}.
\begin{figure}[ht]
\scalebox{0.8}{\includegraphics{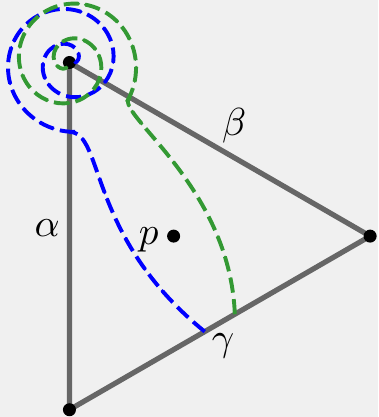}}\qquad\quad\scalebox{0.8}{\includegraphics{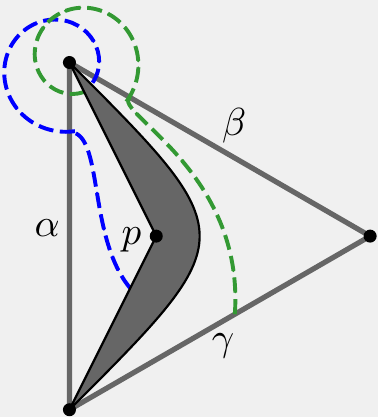}}\qquad\quad\scalebox{0.8}{\includegraphics{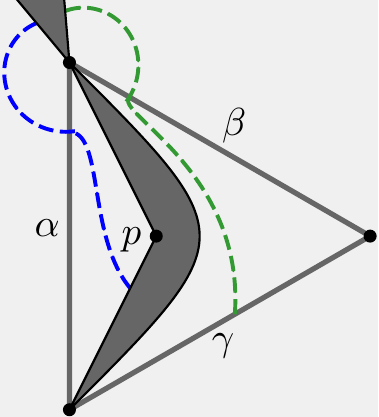}}
\caption{A problem ruled out by Definition~\ref{resect compat tri def}}
\label{situated spiral}
\end{figure}
However, the third condition in Definition~\ref{resect compat tri def} requires that some other arc incident to the endpoint of $\alpha$ is also resected, so this problem never occurs.
See the right picture of Figure~\ref{situated spiral}.

The preceding argument assumed that $\beta$ and $\gamma$ do not coincide.
Assuming now that $\beta$ and $\gamma$ coincide, we can't compute shear coordinates at $\beta=\gamma$ without invoking the special rule for fold edges of self-folded triangles.
The resection cuts off a once-punctured digon (the self-folded triangle), triangulated by the single arc $\beta=\gamma$.
The other triangle (besides the self-folded $\alpha\beta\gamma$) having $\alpha$ as an edge is not self-folded, because if so, $(\S,\M)$ is a three-times-punctured sphere.
Thus the situation is as pictured in Figure~\ref{betagamma}.
\begin{figure}[ht]
\scalebox{0.75}{\includegraphics{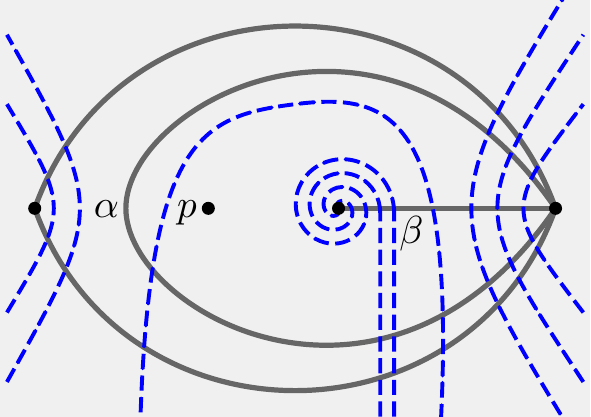}}\qquad\quad\scalebox{0.75}{\includegraphics{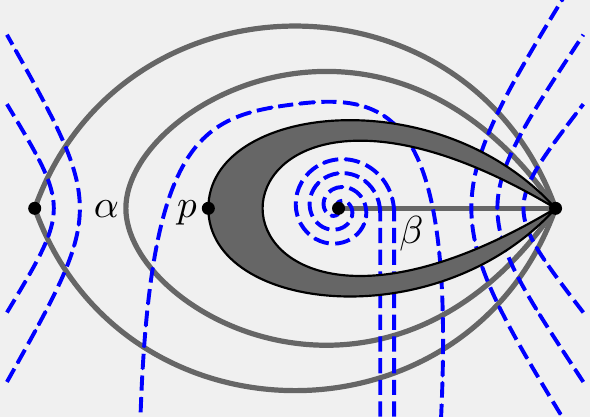}}
\caption{The identity map on quasi-laminations for $\beta=\gamma$}
\label{betagamma}
\end{figure}
In the left picture of the figure, the curves that come from below and spiral counterclockwise into the puncture contribute positively to the shear coordinate of $\beta$ with respect to $T$.
The curves that come from below, go around the puncture, and return downwards also contribute positively.
Curves that come from above and spiral clockwise (not pictured) or curves that come from above, go around the puncture and return upwards (also not pictured) would contribute negatively.
(Curves from below that spiral clockwise and curves from above that spiral counterclockwise contribute zero.)

We place the point $p$ as shown in the figure (or similarly if instead there are curves from above that go around the puncture and return upwards).
By the same construction outlined above, we construct a pairwise compatible collection of allowable curves in the components of $(\S',\M')$ aside from the once-punctured monogon cut off by the resection.
The construction leaves some pieces of curves in the once-punctured monogon, as indicated in the right picture of Figure~\ref{betagamma}.
There are 7 types of curves, and we need to distinguish them according to their behavior in the digon of $T$ from which the monogon was cut.

First, there are 4 types of curves with spirals: two spiral directions, with curves originating from the top or bottom of the digon.
We delete curves from below that spiral clockwise and curves from above that spiral counterclockwise, and retain the other two types.
(Since the original collection of curves was pairwise compatible, at most one of the other two types is present.)
Second, there are 3 types of curves that enter the monogon, cross $\beta$, and leave the digon.
The first type come from curves that cross the digon from top to bottom;  we delete these.
The second type come from curves that come from the bottom of the digon, go around the puncture, and exit the bottom of the digon.
We replace each such curve with a new curve (with the same weight) that enters the monogon and spirals counterclockwise into the puncture.
The third type come from curves that come from the top of the digon, go around the puncture, and exit the top of the digon.
Each of these is replaced with a curve that enters the monogon and spirals clockwise into the puncture.
Again, we delete repetitions of curves and adjust weights accordingly.

These weighted curves, together with the curves we constructed outside of the monogon, are pairwise compatible and have the same shear coordinates as the original weighed collection.
This is the desired new collection $\Lambda'$.

We have proved the first assertion of Theorem~\ref{ref phen resect}.
If also every component of $(\S,\M)$ and $(\S',\M')$ has the Curve Separation Property, then Theorem~\ref{rat FB surfaces} implies the second assertion of Theorem~\ref{ref phen resect}.
\end{proof}


\begin{figure}
\scalebox{0.4}{\includegraphics{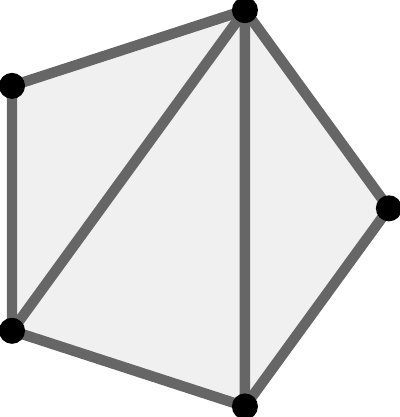}}\quad\raisebox{21pt}{\LARGE$\longrightarrow$}\quad\scalebox{0.4}{\includegraphics{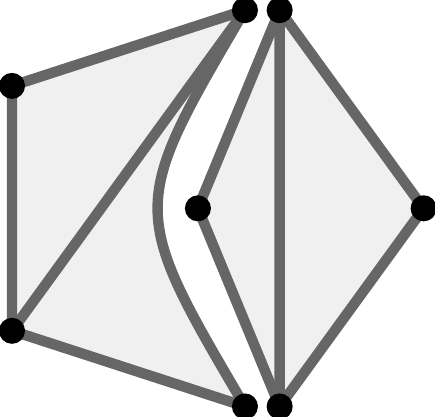}}\qquad\raisebox{21pt}{\LARGE$=$}\qquad\!\scalebox{0.4}{\includegraphics{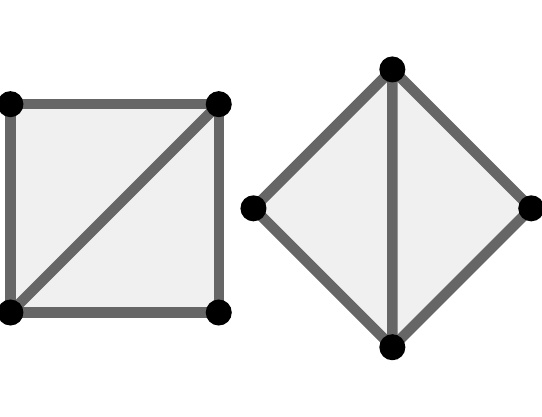}}
\caption{A resection of the pentagon}
\label{pent resect fig}
\end{figure}
\begin{figure}  
\begin{tabular}{cccc}
\scalebox{0.5}{\includegraphics{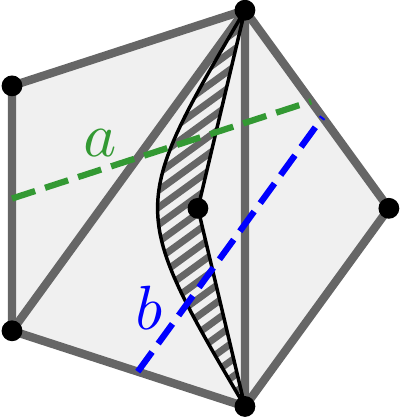}}&&\raisebox{26pt}{\LARGE$\mapsto$}&\raisebox{4pt}{\quad\scalebox{0.6}{\includegraphics{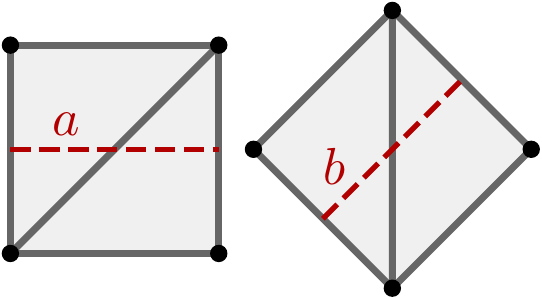}}}\\\hline\\[-8pt]
\scalebox{0.5}{\includegraphics{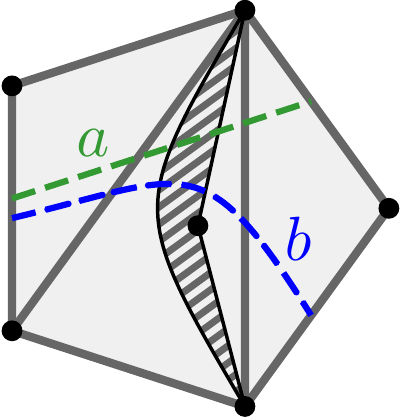}}&&\raisebox{26pt}{\LARGE$\mapsto$}&\raisebox{4pt}{\quad\scalebox{0.6}{\includegraphics{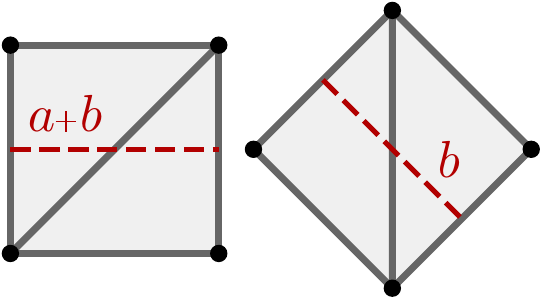}}}\\\hline\\[-8pt]
\scalebox{0.5}{\includegraphics{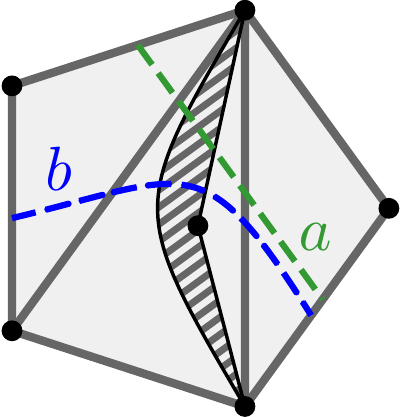}}&&\raisebox{26pt}{\LARGE$\mapsto$}&\raisebox{4pt}{\quad\scalebox{0.6}{\includegraphics{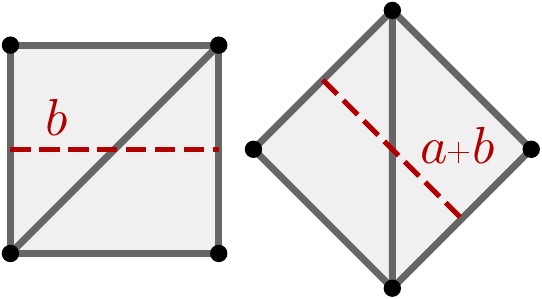}}}\\\hline\\[-8pt]
\scalebox{0.5}{\includegraphics{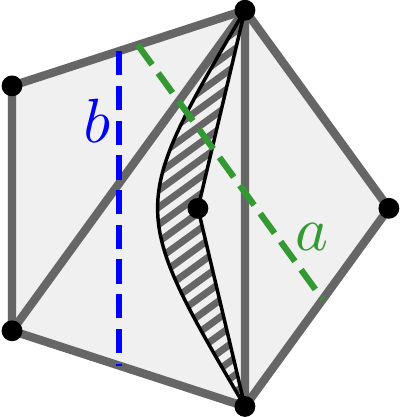}}&&\raisebox{26pt}{\LARGE$\mapsto$}&\raisebox{4pt}{\quad\scalebox{0.6}{\includegraphics{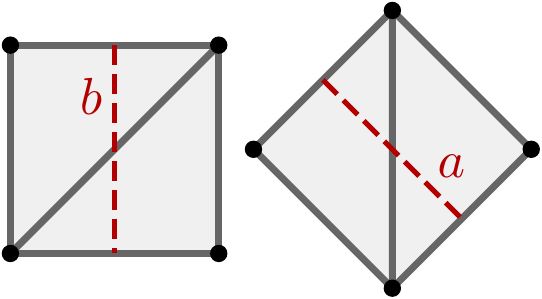}}}\\\hline\\[-8pt]
\scalebox{0.5}{\includegraphics{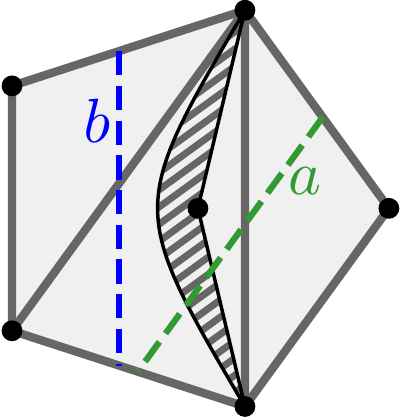}}&&\raisebox{26pt}{\LARGE$\mapsto$}&\raisebox{4pt}{\quad\scalebox{0.6}{\includegraphics{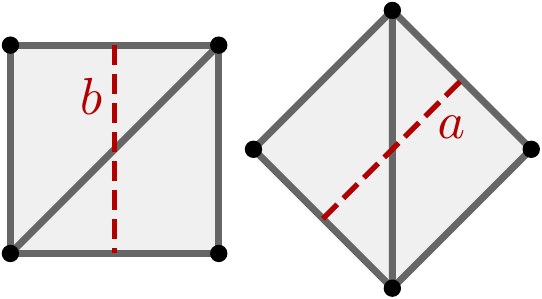}}}\\
\end{tabular}
\caption{The bijection between quasi-laminations on the pentagon and quasi-laminations on the union of two squares}
\label{simple fig}
\end{figure}

\begin{proof}[Proof of Theorem~\ref{id phen resect}]
By Theorem~\ref{null tangle theorem} and Proposition~\ref{positive cone basis}, $\rationals^{B(T)}$ admits a cone basis.
Theorem~\ref{rat FB surfaces} says that $\F_\rationals(T)$ is the rational part of $\F_{B(T)}$ and $\F_\rationals(T')$ is the rational part of $\F_{B(T')}$.
By Theorem~\ref{ref phen resect} and Proposition~\ref{rat part refine mulin}, the identity map $\rationals^B\to\rationals^{B'}$ is mutation-linear.
\end{proof}

\begin{example}\label{simple resection}  
We illustrate the shear-coordinate-preserving bijection from rational quasi-laminations in a surface to rational quasi-laminations in a resected surface for the case of resecting an arc in a triangulated pentagon.
The resection is shown in Figure~\ref{pent resect fig}.
The quasi-laminations of the pentagon are represented in the left column of Figure~\ref{simple fig}.
Each is given by two curves, each with a weight (labeled $a$ or $b$ in the figure).
The right column shows the image of each quasi-lamination under the bijection.
The five rows represent the five maximal cones in the mutation fan for the pentagon.
The mutation fan for the union of two squares has four maximal cones, one of which is the union of two cones of the fan for the pentagon.
\end{example}

\begin{remark}\label{works}
The proof of Theorem~\ref{id phen resect} works because the bijection on quasi-laminations is simple in the direction we described.
Specifically, the curves appearing in the quasi-lamination $L'$ depend only on the curves in $L$.
By contrast, in the inverse direction, weights on curves in $L'$ must be taken into account in order to determine what curves appear in $L$.
For example, it seems nearly hopeless to describe the inverse maps in the following example, which continues Example~\ref{resect ex}.
\end{remark}

\begin{example}\label{resect fan ex}
Figures~\ref{torusfan fig}, \ref{annulusfan fig}, and~\ref{hexagonfan fig} show the mutation fans associated to the three triangulated surfaces of Example~\ref{resect ex}.
\begin{figure}
\scalebox{0.8}{\includegraphics{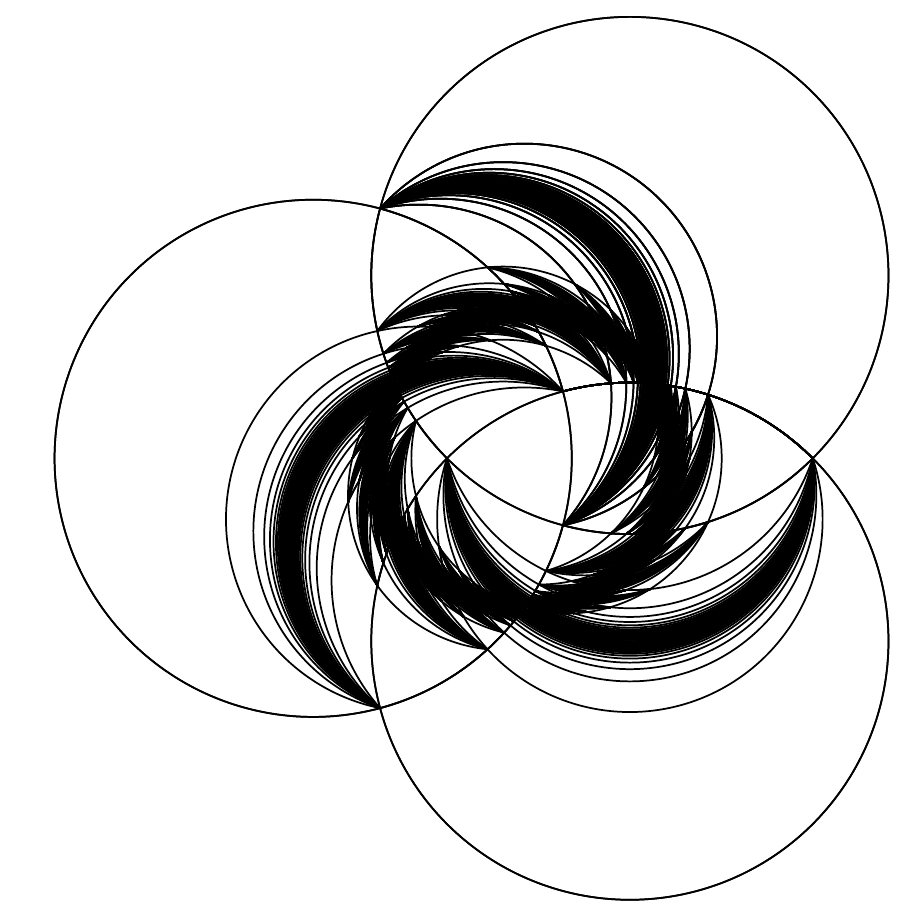}}
\caption{The mutation fan for the once-punctured torus}
\label{torusfan fig}
\end{figure}
Each figure is a projection of a fan in $\reals^3$.
The intersection of the fan with a sphere about the origin is a decomposition of the sphere, which is projected stereographically to the plane to produce the picture shown.
The figures are placed so that paging through the electronic version of this paper shows the refinement relationship guaranteed by Theorem~\ref{id phen resect}.
\end{example}

\section{Refinement of scattering fans}\label{scat sec}
In the introduction, we mentioned some examples of Phenomenon~\ref{ref scat phen} that follow from examples of Phenomenon~\ref{ref phen}.
In this section, we establish an example of Phenomenon~\ref{ref scat phen} that does not follow from an example of Phenomenon~\ref{ref phen}.

We will not define the scattering fan completely, but instead, we will indicate what kind of an object it is and quote results that describe its properties.
Scattering fans are defined by scattering diagrams, as we now describe.
For more on scattering diagrams, see \cite{GHKK}.
For an exposition more suited to the goals of this paper, see~\cite{scatfan}.

\begin{figure}
\scalebox{0.8}{\includegraphics{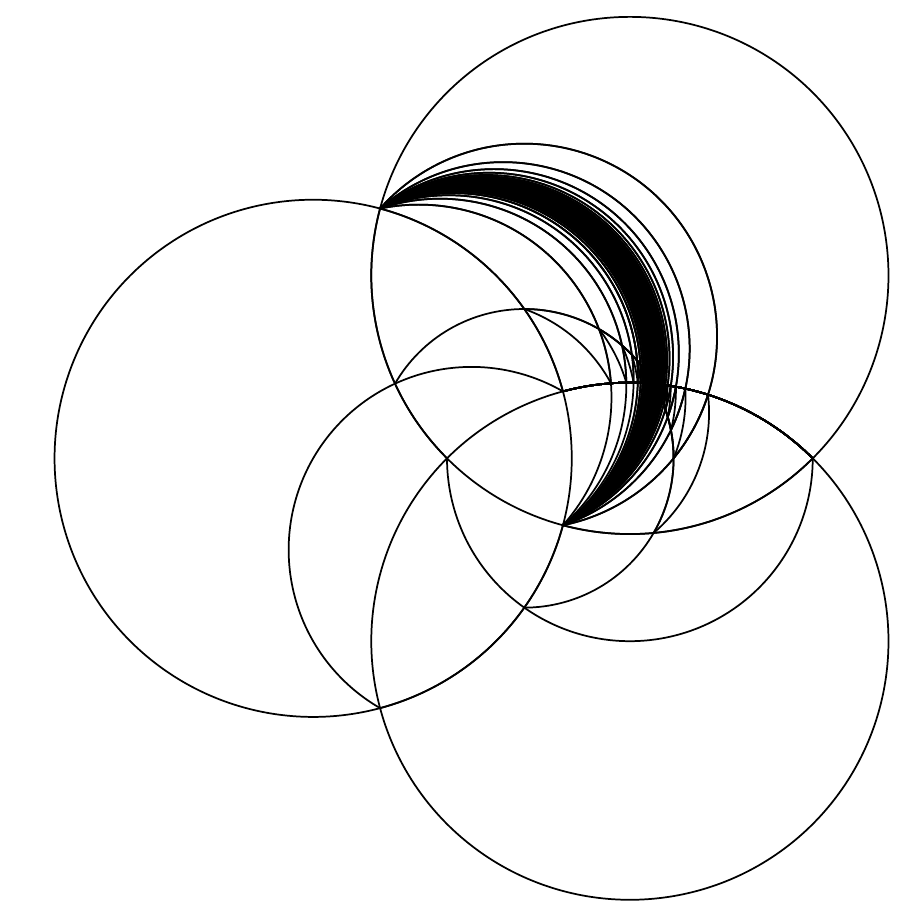}}
\caption{The mutation fan for the annulus}
\label{annulusfan fig}
\end{figure}

A scattering diagram $\D$ is a collection of \newword{walls}.
Each wall is a pair $(\d,f_\d)$, where $\d$ is a codimension-$1$ rational cone in $\reals^n$, normal to a nonzero vector with nonnegative entries and $f_\d$ is a formal power series in $n$ variables with constant term~$1$.
The scattering diagram may have infinitely many walls, but for each $k\ge 1$ only finitely many walls have nonzero coefficients on terms of degree $d$ with $0<d\le k$.
To each sufficiently generic path $\gamma$ in $\reals^n$, we associate the sequence of walls crossed by $\gamma$, and using this sequence and the exchange matrix $B$, we define a certain automorphism of a multivariate formal power series ring.
(Some issues arise when $\D$ is infinite, but these are resolved using the finiteness requirement for each $k$ and taking a limit.)
The \newword{support} of $\D$ is the union of its walls.
A scattering diagram is \newword{consistent} if the automorphism arising from a path depends only on the endpoints of the path.
Two scattering diagrams $\D$ and $\D'$ are \newword{equivalent} if any (sufficiently generic) path defines the same automorphism with respect to both scattering diagrams.
We only care about scattering diagrams up to equivalence.

Two equivalent scattering diagrams may have different supports.
(For example, one can add any wall $(\d,f_\d)$ with $f_\d=1$ to a scattering diagram, and the new scattering diagram is equivalent to the old one, but the support may have changed.
For less trivial ways that two scattering diagrams may have different supports, see the paragraph before \cite[Proposition~2.8]{scatfan}.)
A scattering diagram has \newword{minimal support} if its support is minimal (under containment) among scattering diagrams in its equivalence class.
Every scattering diagram is equivalent to a scattering diagram with minimal support \cite[Proposition~2.8]{scatfan}.
Two different scattering diagrams in the same equivalence class may have minimal support, but if so, their support is the same.

Given a scattering diagram $\D$ with minimal support and a vector~$\n$, the \newword{rampart} associated to $\n$ is the union of all walls of $\D$ that are normal to~$\n$.
Given $p\in\reals^n$, let $\Ram_\D(p)$ be the set of ramparts of $\D$ containing $p$.
We say points $p$ and $q$ in $\reals^n$ are \newword{$\D$-equivalent} if there is a path from $p$ to $q$ on which $\Ram_\D(\,\cdot\,)$ is constant.
The closures of $\D$-equivalence classes are called \newword{$\D$-cones}.
The set $\Fan(\D)$ of all $\D$-cones and their faces is a complete fan \cite[Theorem~3.1]{scatfan}.  

The \newword{cluster scattering diagram} associated to an exchange matrix $B$ is the unique (up to equivalence) consistent scattering diagram satisfying certain conditions.
We do not need full details here, but one of the conditions is that each coordinate hyperplane in $\reals^n$ is a wall of the cluster scattering diagram.
The other condition uses $B$ to place limits on the other walls that can appear.
The \newword{scattering fan} $\ScatFan(B)$ is the fan defined by the cluster scattering diagram for~$B$.

\begin{figure}
\scalebox{0.8}{\includegraphics{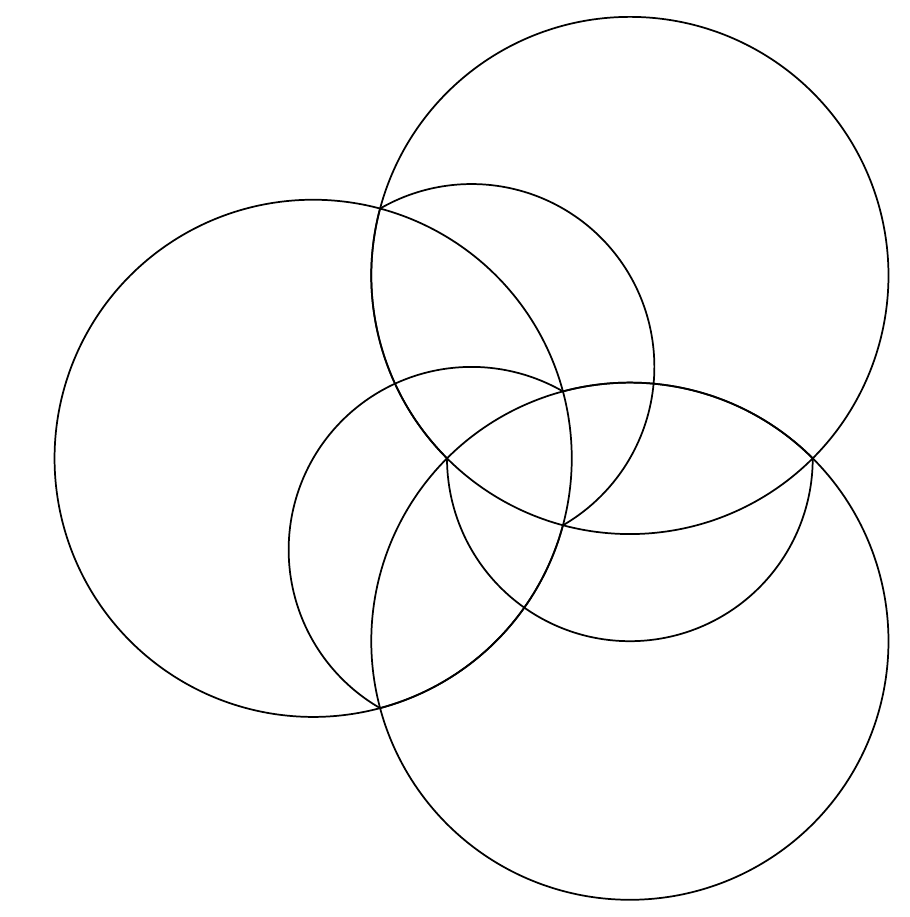}}
\caption{The mutation fan for the hexagon}
\label{hexagonfan fig}
\end{figure}

We now prove Theorem~\ref{ref scat phen 2x2}. 
We make use of the background from \cite[Section~9]{universal} (already quoted in Section~\ref{2x2 ref sec}) on the mutation fan $\F_B$ for a $2\times2$ exchange matrix $B=\begin{bsmallmatrix*}[r]0&a\\b&0\end{bsmallmatrix*}$. 
A slope that is not between $s_\infty(a,b)$ and $t_\infty(a,b)$ is a relevant slope for $\ScatFan(B)$ if and only if it is a relevant slope for $\F_B$.
The Discreteness Hypothesis is the assertion that, for $ab<-4$, every slope weakly between $s_\infty(a,b)$ and $t_\infty(a,b)$ is a relevant slope (in the sense of Section~\ref{2x2 ref sec}) for the scattering fan $\ScatFan(B)$.

\begin{proof}[Proof of Theorem~\ref{ref scat phen 2x2}]
Suppose $B=\B ab$ and $B'=\B cd$ are $2\times2$ exchange matrices.
If neither is wild, then since the scattering fan coincides with the mutation fan in the non-wild $2\times2$ case, Theorem~\ref{ref phen 2x2}\eqref{neither wild} implies that $\ScatFan(B)$ refines $\ScatFan(B')$ if and only if $B$ dominates $B'$.
If $B'$ is wild and $B$ is not, then $\ScatFan(B')$ has irrational relevant slopes, while $\ScatFan(B)$ has only rational slopes, so $\ScatFan(B)$ does not refine $\ScatFan(B')$.

If neither is of finite type, we first show that dominance is necessary for refinement.
Indeed, suppose that $B$ does not dominate $B'$, so that either $a<c$ or $b>d$.
If $a<c$, then the smallest relevant slope of $\ScatFan(B)$ is $s_0(a,b)=-a$, while the smallest relevant slope of $\ScatFan(B')$ is $s_0(c,d)=-c$.
Thus the smallest relevant slope of $\ScatFan(B')$ is not a relevant slope of $\ScatFan(B)$, so $\ScatFan(B)$ does not refine $\ScatFan(B')$.
Similarly, if $b>d$, then the two largest relevant slopes are $t_0(a,b)=\frac1b$ and $t_0(c,d)=\frac1d$, and thus $\ScatFan(B)$ does not refine $\ScatFan(B')$.

Continuing in the case where neither $B$ nor $B'$ is of finite type, we show that dominance is sufficient for refinement.
It is enough to consider two cases:  the case where $c=a$ and $d=b+1$ and the case where $c=a-1$ and $d=b$.
First, consider the case where $c=a$ and $d=b+1$.
In this case, $s_0(a,b)=-a=-c=s_0(c,d)$.
If $a>1$, then one can show that $s_\infty(a,b)<s_1(c,d)$ and $t_0(c,d)<t_\infty(a,b)$.
(Using the fact that $a>0$ and $b<b+1<0$, both of these inequalities can be shown to be equivalent to $(1-a)b>a$, which holds because $a>1$ and $b<-1$.)
Now the Discreteness Hypothesis implies that every relevant slope of $\ScatFan(B')$ is a relevant slope of $\ScatFan(B)$ as desired.

If on the other hand, $a=1$, then $s_2(a,b)=-a\frac{-ab-2}{-ab-1}=\frac{-b-2}{b+1}$ and $s_1(c,d)=\frac{-cd-1}{d}=\frac{-a(b+1)-1}{b+1}=\frac{-b-2}{b+1}$.
Furthermore, $t_1(a,b)=-a\frac1{-ab-1}=\frac1{b+1}$ and $t_0(c,d)=\frac1{b+1}$.
To complete the proof in the case where $a>1$, we show that $s_\infty(a,b)<s_2(c,d)$ and $t_1(c,d)<t_\infty(a,b)$.
(Both equations are shown to be equivalent to $b<-4$, which holds because $a=1$ and $a(b+1)\le-4$.)
Again the Discreteness Hypothesis now implies that every relevant ray of $\ScatFan(B')$ is a relevant ray of $\ScatFan(B)$.

In the case where $c=a-1$ and $d=b$, we argue similarly or appeal to symmetry.

We have proven the theorem in the case where both $B$ and $B'$ are non-wild, in the case where both $B$ and $B'$ are of infinite type, and in the case where $B'$ is of wild type but $B$ is not. 
It thus remains to prove the theorem in the case where~$B'$ is of finite type and $B$ is of wild type.
If $B$ dominates $B'$, then there is an exchange matrix $B''$ of affine type such that $B''$ dominates $B'$ and $B$ dominates~$B''$.
Concatenating the results we have proved, we see that $\ScatFan(B)$ refines $\ScatFan(B')$.
If $B=\B ab$ does not dominate $B'=\B cd$, then there are two possibilities:
Either $a<c$ or $b>d$.
Notice that in every case (as shown in Table~\ref{fin slopes}), the smallest relevant slope of $\ScatFan(B')$ is $-c$.
The smallest relevant slope of $\ScatFan(B)$ is $s_0(a,b)=-a$.
Thus if $a<c$, $\ScatFan(B)$ does not refine $\ScatFan(B')$.
Similarly, the largest relevant slope of $\ScatFan(B')$ is $\frac1d$ and the largest relevant slope of $\ScatFan(B)$ is $\frac1b$.
Thus also if $b>d$, $\ScatFan(B)$ does not refine $\ScatFan(B')$.
We have checked that in every $2\times2$ case, $\ScatFan(B)$ refines $\ScatFan(B')$ if and only if $B$ dominates $B'$.
\end{proof}

\section{Morphisms of cluster algebras}\label{morph sec}
In this section, we prove results and describe computer verifications that provide examples of Phenomenon~\ref{hom phen}.

\subsection{General considerations}\label{easy obs}
Let $B$ be an $n\times n$ exchange matrix and let $\x=x_1,\ldots,x_n$ and $\y=y_1,\ldots,y_n$ be indeterminates.
Let $L$ stand for $\integers[\x^{\pm1},\y]=\integers[x_1^{\pm1},\ldots,x_n^{\pm1},y_1,\ldots,y_n]$, the ring of Laurent polynomials in $\x$ with coefficients integer polynomials in $\y$.
Let $K$ stand for the field of rational functions in $\x$ with coefficients integer polynomials in~$\y$.
Importantly to what follows, we do \emph{not} invert the $\y$.
(That is, we work with ordinary polynomials in $\y$, not Laurent polynomials.)

A \newword{seed} is a pair $(\tB,(v_1,\ldots,v_n))$, where $\tB$ is a $2n\times n$ integer matrix (an \newword{extended exchange matrix}) whose top $n$ rows are an exchange matrix and $(v_1,\ldots,v_n)$ is an ordered $n$-tuple of elements of $K$.
The tuple $(v_1,\ldots,v_n)$ is called a \newword{cluster}.
(This is a special case of the definition of a seed of geometric type \cite{ca4}.)
Given a seed $(\tB,(v_1,\ldots,v_n))$ and $k\in\set{1,\ldots,n}$, we can \newword{mutate} to obtain a new seed $\mu_k(\tB,(v_1,\ldots,v_n))=(\tB',(v'_1,\ldots,v'_n))$, where $\tB'=\mu_k(\tB)$ is described in \eqref{b mut}, $v'_i=v_i$ for $i\neq k$, and $v_k$ is described by the \newword{exchange relation}
\begin{equation}\label{x mut}
v_kv'_k=\prod_{i=1}^nv_i^{[b_{ik}]_+}\prod_{i=1}^ny_i^{[b_{(n+i)k}]_+}+\prod_{i=1}^nv_i^{[-b_{ik}]_+}\prod_{i=1}^ny_i^{[-b_{(n+i)k}]_+},
\end{equation}
where the $b_{ik}$ are entries of $B_t$ and $[b]_+$ means $\max(0,b)$.

We take as an \newword{initial seed} the pair $\bigl(\begin{bsmallmatrix}B\\I\end{bsmallmatrix},(x_1,\ldots,x_n)\bigr)$, where $B$ is an exchange matrix and $I$ is the identity matrix.
A \newword{cluster variable} with respect to this initial seed is an entry in a cluster that can be obtained from $\bigl(\begin{bsmallmatrix}B\\I\end{bsmallmatrix},(x_1,\ldots,x_n)\bigr)$ by an arbitrary sequence of mutations.
Typically there are infinitely many cluster variables. 
If not, then $B$ is of \newword{finite type}.
The \newword{principal coefficients cluster algebra} $\A_\bullet(B)$ is the subring of $K$ generated by all cluster variables and the $\y$.
By the Laurent Phenomenon \cite[Theorem~3.1]{ca1}, the cluster algebra $\A_\bullet(B)$ is a subring of $L$.
Write $\Var_\bullet(B)$ for the set of cluster variables in $\A_\bullet(B)$.

Consider a $\integers^n$-grading of $L=\integers[\x^\pm,\y]$ given by setting the degree of each $x_k$ to be the standard basis vector $\e_k$ and setting the degree of each $y_k$ to be the negative of the $k\th$ column of $B$. 
If an element $x$ of $L$ is homogeneous with respect to this grading, then its $\integers^n$-degree is called the \newword{$\g$-vector} of $x$.
Each cluster variable in $\A_\bullet(B)$ is homogeneous in this grading and thus has a $\g$-vector \cite[Proposition~6.1]{ca4}.

We construct another cluster algebra, just as above, but with primes on every symbol.
Thus $B'$ is an $n'\times n'$ matrix, $L'$ is the ring of Laurent polynomials in $\x'=x'_1,\ldots,x'_{n'}$ with coefficients polynomials in $\y'=y'_1,\ldots,y'_{n'}$, and we define $\A_\bullet(B')$ with initial seed $\bigl(\begin{bsmallmatrix}B\\I\end{bsmallmatrix},\x\bigr)$ and cluster variables $\Var_\bullet(B')$, etc.
 
Given a set map $\nu$ from $\set{\x',\y'}$ to $L$, there is a unique ring homomorphism from $L'$ to $K$ agreeing with $\nu$ on $\set{\x',\y'}$.
We reuse the symbol $\nu$ for the homomorphism.
The following observation is immediate from the definition of~$\A_\bullet(B')$.

\begin{prop}\label{hom from initial}
Suppose $\nu:\set{\x',\y'}\to L$ is a set map and reuse the symbol $\nu$ for the extension to a ring homomorphism from $L'$ to $K$.
If $\nu(\Var_\bullet(B'))\subset\A_\bullet(B)$, then $\nu$ restricts to a ring homomorphism from $\A_\bullet(B')$ to $\A_\bullet(B)$.
\end{prop}


Given a set map $\nu:\Var_\bullet(B')\cup\set{\y}\to\A_\bullet(B)$, each exchange relation in $\A_\bullet(B')$ is mapped to a equation (valid or not) relating elements of $\A_\bullet(B)$.

\begin{prop}\label{hom from cl vars}
Suppose $\nu$ is a set map from $\Var_\bullet(B')\cup\set{\y'}$ to $\A_\bullet(B)$.
If $\nu$ maps every exchange relation of $\A_\bullet(B')$ to a valid relation in $\A_\bullet(B)$, then $\nu$ extends to a ring homomorphism from $\A_\bullet(B')$ to $\A_\bullet(B)$.
\end{prop}
\begin{proof}
Restrict the set map $\nu$ to $\set{\x',\y'}$ and consider the extension of $\nu$ to a ring homomorphism $\overline\nu:L'\to F$.
We claim that each $z'\in\Var_\bullet(B')$ has $\nu(z')=\overline\nu(z')$.
We argue by induction on the smallest number $k$ of mutation steps required to obtain $z'$ from the initial cluster $\x'$.
Given a sequence of $k$ mutations producing $z'$ from $\x'$, the last mutation defines an exchange relation writing $z'$ in terms of the $y_i$ and in terms of cluster variables $w_i'$ that can be reached from $\x'$ in fewer mutations.
Since $\nu$ maps each exchange relation to a valid relation, we obtain an expression for $\nu(z')$ in terms of the quantities $\nu(w_i')$.
By induction, each $w_i$ has $\nu(w_i')=\overline\nu(w_i')$.
Thus $\nu(w_i')$ can be obtained from any expression for $w_i'$ by replacing each $x'_i$ and $y'_j$ by $\nu(x_i)$ and $\nu(y_j)$.
We have found an expression for $z'$ such that replacing each $x'_i$ and $y'_j$ by $\nu(x_i)$ and $\nu(y_j)$ yields $\nu(z')$.
Therefore $\nu(z')=\overline\nu(z')$.

Since we know that $\nu$ sends every cluster variable of $\A_\bullet(B')$ to an element of $\A_\bullet(B)$, we have verified the hypotheses of Proposition~\ref{hom from initial} and thus obtained a homomorphism from $\A_\bullet(B')$ to $\A_\bullet(B)$.
Furthermore, this homomorphism agrees with the given set map from $\Var_\bullet(B')\cup\set{\y'}$ to $\A_\bullet(B)$.
\end{proof}

We will often use the following criterion to checking that $\nu_\z$ is injective.
\begin{prop}\label{one antip}
Suppose $B$ dominates $B'$.
If there is at most one index $k$ such that there exists $i$ with $b_{ik}<b'_{ik}\le0$, then the homomorphism $\nu_\z$ is injective.
\end{prop}
\begin{proof}
Up to reindexing, we may as well assume that $n$ is the unique index such that there exists $i$ with $b_{in}<b'_{in}\le0$.
(If there is no pair $ik$ such that $b_{ik}<b'_{ik}\le0$, then an easier version of the following argument works.)
For all $k<n$, we can take $z_k$ to be a cluster monomial, specifically, a monomial in the initial cluster variables $x_1,\ldots,x_n$.
Writing $x_{n+i}$ for $y_i$ and $x_{n+i}'$ for $y_i'$, the Jacobian matrix $\Bigl[\pd{\nu(x_i')}{x_j}\Bigr]$ is upper-triangular.
The first $n$ diagonal entries are $1$.
The next $n-1$ diagonal entries are (possibly trivial) monomials in $x_1,\ldots,x_n$.
The last diagonal entry is $\pd{}{y_n}(y_nz_n)$. 
That entry is never zero because $z_n$ is Laurent in $\x$ but polynomial in $\y$.
The Jacobian determinant is the product of these diagonal entries, and is therefore not zero.
Non-vanishing of the Jacobian determinant is a well-known criterion for injectivity (see, for example, \cite[Theorem~2.2]{ER}).
\end{proof}

\subsection{\texorpdfstring{The $2\times2$ case}{The 2 by 2 case}}\label{2x2 hom sec}
We now prove Theorems~\ref{hom phen 2x2 part 2} and~\ref{hom phen 2x2 part 1}, which provide examples of Phenomenon~\ref{hom phen} in the $2\times2$ case.
Although there is some overlap in the two theorems, we prove them mostly separately.
The advantage is that we prove Theorem~\ref{hom phen 2x2 part 2} (almost) entirely using the basic definition of cluster variables, with (almost) no theta functions (and without the need for the Discreteness Hypothesis).

Theorem~\ref{hom phen 2x2 part 2} establishes Phenomenon~\ref{hom phen} when $B$ is $2\times2$ and of finite or affine type.
We will reduce Theorem~\ref{hom phen 2x2 part 2} to the following proposition.

\begin{prop}\label{implies hom phen 2x2 part 2}
Under the hypotheses of Theorem~\ref{hom phen 2x2 part 2}, the map $\nu_\z$ sends each cluster variable of $\A_\bullet(B')$ to a cluster variable of $\A_\bullet(B)$ or, in the case where $B$ is of affine type, possibly to the theta function of the limiting ray.
\end{prop}

When $B$ is of affine type, the only ray theta function that is not a cluster variable is the ray theta function for the limiting ray.
Furthermore, since $B'$ is strictly dominated by $B$, it is of finite type, so all of its ray theta functions are cluster variables.
Thus Proposition~\ref{implies hom phen 2x2 part 2} verifies that $\nu_\z$ takes each ray theta function in $\ScatFan^T(B')$ to the theta function for the same ray in $\ScatFan^T(B)$.
Necessarily, each $z_k$ is a power of a cluster variable with $\g$-vector $\begin{bsmallmatrix*}[r]\pm1&0\end{bsmallmatrix*}$or $\begin{bsmallmatrix*}[r]0&\pm1\end{bsmallmatrix*}$.
Assuming Proposition~\ref{implies hom phen 2x2 part 2}, Proposition~\ref{hom from initial} implies that the map $\nu_\z$ is a ring homomorphism from $\A_\bullet(B')$ to $\A_\bullet(B)$ preserving $\g$-vectors and Proposition~\ref{one antip} implies that $\nu_\z$ is injective.
We have verified that Proposition~\ref{implies hom phen 2x2 part 2} implies Theorem~\ref{hom phen 2x2 part 2}.

To prove Proposition~\ref{implies hom phen 2x2 part 2}, we can, by symmetry (transposing both $B$ and $B'$), assume that $B=\B ba$ with $a<0$ and $b>0$ and $B'=\B dc$ with $c\le0$ and $d\ge0$.
Furthermore, if $B$ dominates $B'$, which dominates $B''$ and we check the proposition for $B$ and $B'$, and also check it for $B'$ and $B''$, then in particular the map from $\A_\bullet(B')$ to $\A_\bullet(B)$ will take the cluster variables of $\A_\bullet(B')$ with $\g$-vectors $\begin{bsmallmatrix*}[r]\pm1&0\end{bsmallmatrix*}$ and $\begin{bsmallmatrix*}[r]0&\pm1\end{bsmallmatrix*}$ to the cluster variables of $\A_\bullet(B)$ with the same $\g$-vectors.
Thus the composition of the maps from $\A_\bullet(B'')$ to $\A_\bullet(B')$ and from $\A_\bullet(B')$ to $\A_\bullet(B)$ will coincide with the map directly from $\A_\bullet(B'')$ to $\A_\bullet(B)$, and will affect cluster variables of $\A_\bullet(B'')$ as desired.
In light of these facts and Theorem~\ref{ref phen 2x2}, we need to check the following cases.
\begin{enumerate}[\,\,{Case }1:]
\item $B'=\B00$ and $B=\B ba$.
\item $B'=\B1{-1}$ and $B=\B b{-1}$.
\item $B'=\B1{-1}$ and $B=\B 1a$.
\item $B'=\B2{-1}$ and $B=\B 3{-1}$.
\item $B'=\B2{-1}$ and $B=\B 2{-2}$.
\item $B'=\B1{-2}$ and $B=\B 1{-3}$.
\item $B'=\B1{-2}$ and $B=\B 2{-2}$.
\item $B'=\B3{-1}$ and $B=\B 4{-1}$.
\item $B'=\B1{-3}$ and $B=\B 1{-4}$.
\end{enumerate}

We define cluster variables indexed by integers, starting with $x_1$ and $x_2$ and defining the remaining $x_i$ so that each pair $x_i,x_{i+1}$ forms a cluster.
In particular, each $x_i$ is related to $x_{i+2}$ by the exchange relation \ref{x mut}.

{\allowdisplaybreaks
Since we are assuming that $a<0$ and $b>0$, we have the following formulas for cluster variables:
\begin{align}
\label{x0 eq}
x_0&=\frac{x_1^by_2+1}{x_2}\displaybreak[1]\\
\label{x-1 eq}
x_{-1}&=\frac{x_0^{-a}y_1+1}{x_1}\displaybreak[1]\\
\label{x-2 eq}
x_{-2}&=\frac{x_{-1}^b+y_2}{x_0}\displaybreak[1]\\
\label{x-3 eq}
x_{-3}&=\frac{x_{-2}^{-a}+y_1y_2^{-a}}{x_{-1}}\displaybreak[1]\\
\label{x-4 eq}
x_{-4}&=\frac{x_{-3}^b+y_1^by_2^{-ab-1}}{x_{-2}}\displaybreak[1]\\
\label{x3 eq}
x_3&=\frac{y_1+x_2^{-a}}{x_1}\displaybreak[1]\\
\label{x4 eq}
x_4&=\frac{y_1^by_2+x_3^b}{x_2}\displaybreak[1]\\
\label{x5 eq}
x_5&=\frac{y_1^{-ab-1}y_2^{-a}+x_4^{-a}}{x_3}
\end{align}
In fact, all of these formulas except the formulas for $x_{-4}$ and $x_5$ are valid under the weaker assumption that $a\le0$ and $b\ge0$.
Thus, since we also have $c\le0$ and $d\ge0$, those formulas hold for cluster variables $x'_i$ with $-3\le i\le4$, replacing $a$ by~$c$, replacing $b$ by $d$, and replacing each $x_j$ by $x'_j$ in the formulas.
}

We will use the notation $x_\infty(a,b)$ for the ray theta function of the limiting ray when $ab=-4$.
Simple computations (quoted later as Propositions~\ref{theta m1 b} and~\ref{theta m2 a}) in the transposed cluster scattering diagram yields the following formulas.
\begin{align}
\label{x-22 eq}
x_\infty(-2,2)&=\frac{y_1+y_1y_2x_1^2+x_2^2}{x_1x_2}  
\displaybreak[1]\\
\label{x-14 eq}
x_\infty(-1,4)&=\frac{x_1^4y_1^2y_2+x_2^2+2x_2y_1+y_1^2}{x_1^2x_2} 
\displaybreak[1]\\
\label{x-41 eq}
x_\infty(-4,1)&=\frac{y_1y_2^2x_1^2+y_1+2y_1y_2x_1+x_2^4}{x_1x_2^2} 
\end{align}
(The formula for $x_\infty(-2,2)$ is the subject of \cite[Example~3.8]{CGMMRSW}, but due to the global transpose, we must switch the indices $1$ and $2$ to relate \eqref{x-22 eq} to \cite[Example~3.8]{CGMMRSW}.)

The proof of Proposition~\ref{implies hom phen 2x2 part 2} (and thus Theorem~\ref{hom phen 2x2 part 2}) follows from a sequence of lemmas that describe where $\nu_\z$ sends various cluster variables.
Each lemma follows from a computation that can be checked by hand or with a computer-algebra system.
We omit the details.
In every case, $z_1=x_0^{c-a}$ and $z_2=x_1^{b-d}$.

\begin{lemma}\label{00 lemma}
If $B=\B ba$ and $B'=\B dc$ with $a\le c\le0$ and $b\ge d\ge0$, then $\nu_\z(x'_0)=x_0$ and $\nu_\z(x'_{-1})=x_{-1}$.
\end{lemma}

Lemma~\ref{00 lemma} is enough for Case 1 because for $B'=\B00$, the cluster variables in $\A_\bullet(B')$ are $\set{x'_{-1},x'_0,x'_1,x'_2}$.
Lemma~\ref{00 lemma} also reduces the checking for other cases.

\begin{lemma}\label{a same lemma}
If $B=\B ba$ and $B'=\B da$ with $b\ge d\ge0$, then $\nu_\z(x'_3)=x_3$.
\end{lemma}

\begin{lemma}\label{b same lemma}
If $B=\B ba$ and $B'=\B bc$ with $a\le c\le0$, then $\nu_\z(x'_{-2})=x_{-2}$.
\end{lemma}

When $B'=\B1{-1}$, the only cluster variable not accounted for in Lemma~\ref{00 lemma} is $x'_{-2}=x'_3$.
Thus Cases 2--3 are handled by Lemmas~\ref{00 lemma}, \ref{a same lemma}, and~\ref{b same lemma}.

The cluster variables for $B'=\B2{-1}$ not accounted for by Lemma~\ref{00 lemma} are $x'_{-2}$ and $x'_3$.
In Case 4, where $B=\B3{-1}$, the variable $x'_3$ is accounted for by Lemma~\ref{a same lemma}.
The variable $x_{-2}$ is accounted for by the following lemma, and we are finished with Case 4.

\begin{lemma}\label{d+1 1 lemma}
If $B=\begin{bsmallmatrix*}[r]0&d+1\\-1&0\,\,\,\end{bsmallmatrix*}$ and $B'=\B d{-1}$ with $d\ge 1$, then $\nu_\z(x'_{-2})=x_{-3}$.
\end{lemma}

In Case 5, the variable $x'_{-2}$ is accounted for by Lemma~\ref{b same lemma}, and we complete the verification of Case 5 by proving the following lemma.
\begin{lemma}\label{22 21 lemma}
If $B=\B2{-2}$ and $B'=\B2{-1}$, then $\nu_\z(x'_3)=x_\infty(2,-2)$.
\end{lemma}

Similarly, the cluster variables for $B'=\B1{-2}$ not accounted for by Lemma~\ref{00 lemma} are $x'_{-2}$ and $x'_3$.
In Case 6, $B=\B1{-3}$, so $x'_{-2}$ is accounted for by Lemma~\ref{b same lemma}.
The following lemma finishes Case 6 by accounting for $x'_3$.
\begin{lemma}\label{1 c-1 lemma}
If $B=\begin{bsmallmatrix*}[c]0&1\\c-1&0\end{bsmallmatrix*}$ and $B'=\B 1c$ with $c\le-1$, then $\nu_\z(x'_3)=x_4$.
\end{lemma}

In Case 7, the variable $x'_3$ is accounted for by Lemma~\ref{a same lemma}, and we complete the case with the following lemma.
\begin{lemma}\label{22 12 lemma}
If $B=\B2{-2}$ and $B'=\B1{-2}$, then $\nu_\z(x'_{-2})=x_\infty(2,-2)$.
\end{lemma}

In Case 8, when $B'=\B3{-1}$ and $B=\B4{-1}$, the only cluster variables in $\A_\bullet(B')$ not accounted for by Lemmas~\ref{00 lemma}, \ref{a same lemma}, and~\ref{d+1 1 lemma} are $x'_{-3}$ and $x'_4$.
The following two lemmas account for these.

\begin{lemma}\label{31 41 lemma 1}
If $B=\B4{-1}$ and $B'=\B3{-1}$, then $\nu_\z(x'_4)=x_5$.
\end{lemma}

\begin{lemma}\label{31 41 lemma 2}
If $B=\B4{-1}$ and $B'=\B3{-1}$, then $\nu_\z(x'_{-3})=x_\infty(4,-1)$.
\end{lemma}

Finally, in Case 9, when $B'=\B1{-3}$ and $B=\B1{-4}$, Lemmas~\ref{00 lemma}, \ref{b same lemma} and~\ref{1 c-1 lemma} account for all of the cluster variables except $x'_{-3}$ and $x'_4$.
We complete the case with the following two lemmas.
\begin{lemma}\label{13 14 lemma 1}
If $B=\B1{-4}$ and $B'=\B1{-3}$, then $\nu_\z(x'_{-3})=x_{-4}$.
\end{lemma}

\begin{lemma}\label{13 14 lemma 2}
If $B=\B1{-4}$ and $B'=\B1{-3}$, then $\nu_\z(x'_4)=x_\infty(1,-4)$.
\end{lemma}

We have completed the proof of Proposition~\ref{implies hom phen 2x2 part 2} and thus of Theorem~\ref{hom phen 2x2 part 2}.
We now turn to Theorem~\ref{hom phen 2x2 part 1}, which establishes the first part of Phenomenon~\ref{hom phen} for $B'$ of finite type and, assuming the Discreteness Hypothesis, describes when the second part holds.
We will reduce Theorem~\ref{hom phen 2x2 part 1} to the following proposition.

\begin{proposition}\label{implies hom phen 2x2 part 1}
If $B$ and $B'$ are $2\times2$ exchange matrices such that $B$ dominates $B'$, with $B'$ of finite type, then $\nu_\z$ takes every cluster variable of $\A_\bullet(B')$ to a theta function for $B$ plus an element of $\A_\bullet(B)$.
In fact, $\nu_\z$ takes every cluster variable of $\A_\bullet(B')$ to a theta function for $B$ unless $B=\B ba$ and $B'=\B dc$ with $cd=-3$ and $1\not\in\set{|a|,|b|}$
\end{proposition}

Suppose we have proven Proposition~\ref{implies hom phen 2x2 part 1}.
It is not true in general that every ray theta function for $B$ is an element of $\A_\bullet(B)$, but combining \cite[Theorem~1.18]{ca3}, \cite[Theorem~0.12]{GHKK}, and \cite[Proposition~0.14]{GHKK}, we see that it is true in this case.
Thus in particular $\nu_\z$ takes every cluster variable of $\A_\bullet(B')$ to an element of $\A_\bullet(B)$.
We apply Proposition~\ref{hom from initial} to see that $\nu_\z$ is a ring homomorphism from $\A_\bullet(B')$ to $\A_\bullet(B)$ preserving $\g$-vectors.
Again, Proposition~\ref{one antip} implies that $\nu_\z$ is injective.
To prove the assertion about ray theta functions mapping to ray theta functions, first note that, since $B'$ is of finite type, all of its ray theta functions are cluster variables.
Thus Proposition~\ref{implies hom phen 2x2 part 1} verifies (except in the excluded cases) that $\nu_\z$ takes each ray theta function in $\ScatFan^T(B')$ to the theta function for the same ray in $\ScatFan^T(B)$.
Since we assume the Discreteness Hypothesis, Theorem~\ref{ref scat phen 2x2} implies that these theta functions are in fact \emph{ray} theta functions.
We see that Proposition~\ref{implies hom phen 2x2 part 1} implies Theorem~\ref{hom phen 2x2 part 1}.

To prove Proposition~\ref{implies hom phen 2x2 part 1}, again by symmetry, we take $B=\B ba$ with $a<0$, $b>0$ and $B'=\B dc$ with $c\le0$ and $d\ge0$ and $cd\ge-3$.
We must check six cases:
\begin{enumerate}[\,\,{Case }1:]
\item $B'=\B00$.
\item $B'=\B1{-1}$.
\item $B'=\B2{-1}$.
\item $B'=\B1{-2}$.
\item $B'=\B3{-1}$.
\item $B'=\B1{-3}$.
\end{enumerate}
We write $\thet_{[m_1,m_2]}$ for the theta function with respect to $B$ associated to the integer vector $[m_1,m_2]$.
We write $\thet'_{[m_1,m_2]}$ for the theta function with respect to $B'$.
In each of the six cases, the cluster variables $x'_{-1},x'_0,x'_1,x'_2$, are sent (according to Lemma~\ref{00 lemma} and by definition) by $\nu_\z$ to  $x_{-1},x_0,x_1,x_2$, which in turn are $\thet_{[\pm1,0]}$ and $\thet_{[0,\pm1]}$.
This observation completes Case 1 and leaves from 1 to 4 remaining cluster variables to check in the other cases.
For convenience, we call these remaining cluster variables \newword{diagonal cluster variables} here.
We handle the remaining cases in several lemmas, some of which are slightly more general:  Unless explicitly stated, the lemmas do not need the assumption that $cd\ge-3$.

The lemmas are proved using the following theta function computations, which are \cite[Proposition~3.16]{scatcomb}, \cite[Proposition~3.18]{scatcomb}, and \cite[Example~3.20]{scatcomb}.
As before, the notation $[b]_+$ means $\max(0,b)$.

\begin{prop}\label{theta m1 b}
If $-b\le m_1\le0$ and $m_2\ge0$, then 
\[\thet_{[m_1,m_2]}=x_1^{m_1}\sum_{i=0}^{-m_1}\binom{-m_1}{i}y_1^ix_0^{[-m_2-ai]_+}x_2^{[m_2+ai]_+}\,.\]
\end{prop}

\begin{prop}\label{theta m2 a}
If $m_1<-b$ and $0\le m_2<-a$, then 
\[\thet_{[m_1,m_2]}=x_1^{m_1}x_2^{m_2}+\sum_{i=1}^{-m_1}\sum_{j=0}^{m_2}\binom{-m_1-bj}i\binom{m_2}j y_1^iy_2^jx_0^{-m_2-ai}x_1^{m_1+bj}.\]
\end{prop}

\begin{prop}\label{specific}
If $a=-3$ and $b=1$, then 
\[\thet_{[-2,3]}=x_1^{-2}x_2^3+2y_1x_1^{-2}+3y_1y_2x_1^{-1}+y_1^2x_0^3x_1^{-2}.\]
\end{prop}

We use $\nu_\z(x'_0)=x_0$ (Lemma~\ref{00 lemma}) throughout the proofs.
Most of the lemmas below are simple theta function computations using Proposition~\ref{theta m1 b}, and we omit the details of those lemmas.
Again, in every case, $z_1=x_0^{c-a}$ and $z_2=x_1^{b-d}$.

The only diagonal cluster variable for $B'=\B1{-1}$ has $\g$-vector $[-1,1]$ and this equals $\thet'_{[-1,1]}$.
The following lemma thus completes Case 2.

\begin{lemma}\label{thet -11 lem}
If $B=\B ba$ and $B'=\B dc$ with $a\le c\le-1$ and $b\ge d\ge1$, then $\nu_\z(\thet'_{[-1,1]})=\thet_{[-1,1]}$.
\end{lemma}

When $B'=\B2{-1}$, there is exactly one diagonal cluster variable not covered by Lemma~\ref{thet -11 lem}, with $\g$-vector $[-2,1]$.
The following lemma completes Case 3.

\begin{lemma}\label{thet -21 lem}
If $B=\B ba$ and $B'=\B dc$ with $a\le c\le-1$ and $b\ge d\ge2$, then $\nu_\z(\thet'_{[-2,1]})=\thet_{[-2,1]}$.
\end{lemma}

When $B'=\B1{-2}$, the one diagonal cluster variable not covered by Lemma~\ref{thet -11 lem} has $\g$-vector $[-1,2]$.
The following lemma completes Case 4.
\begin{lemma}\label{thet -12 lem}
If $B=\B ba$ and $B'=\B dc$ with $a\le c\le-2$ and $b\ge d\ge1$, then $\nu_\z(\thet'_{[-1,2]})=\thet_{[-1,2]}$.
\end{lemma}

When $B'=\B3{-1}$, the diagonal cluster variables have $\g$-vectors $[-1,1]$, $[-2,1]$, $[-3,1]$, and $[-3,2]$.
The first two of these are covered by Lemmas~\ref{thet -11 lem} and~\ref{thet -21 lem}, and the other two are covered by the following two lemmas, where we also encounter the first excluded case of Theorem~\ref{hom phen 2x2 part 1}.

\begin{lemma}\label{thet -31 lem}
If $B=\B ba$ and $B'=\B dc$ with $a\le c\le-1$ and $b\ge d\ge3$, then $\nu_\z(\thet'_{[-3,1]})=\thet_{[-3,1]}$.
\end{lemma}

\begin{lemma}\label{thet -32 lem}
If $B=\B ba$ and $B'=\B d{-1}$ with $a\le-1$ and $b\ge d\ge3$, then 
\[\nu_\z(\thet'_{[-3,2]})=\begin{cases}
\thet_{[-3,2]}&\text{if }a=-1,\text{ or}\\
\thet_{[-3,2]}+3y_1y_2x_0^{-2-a}x_1^{b-3}&\text{if }a\le-2.
\end{cases}\]
\end{lemma}
\begin{proof}
We compute $\thet'_{[-3,2]}=(x'_1)^{-3}\bigl((x'_2)^2+3y'_1x'_2+3(y'_1)^2+(y'_1)^3x'_0\bigr)$, so that 
\[\nu_\z(\thet'_{[-3,2]})=x_1^{-3}(x_2^2+3y_1x_0^{-1-a}x_2+3y_1^2x_0^{-2-2a}+y_1^3x_0^{-2-3a}).\]
We also compute
\begin{align*}
\thet_{[-3,2]}&=x_1^{-3}\sum_{i=0}^{3}\binom{3}{i}y_1^ix_0^{[-2-ai]_+}x_2^{[2+ai]_+}\\
&=\begin{cases}
x_1^{-3}(x_2^2+3y_1x_2+3y_1^2+y_1^3x_0)&\text{if }a=-1,\text{ or}\\
x_1^{-3}(x_2^2+3y_1x_0^{-2-a}+3y_1^2x_0^{-2-2a}+y_1^3x_0^{-2-3a})&\text{if }a\le-2.
\end{cases}
\end{align*}
We see that if $a=-1$ then $\nu_\z(\thet'_{[-3,2]})=\thet_{[-3,2]}$ and if $a\le-2$ then $\nu_\z(\thet'_{[-3,2]})-\thet_{[-3,2]}$
is 
\[3y_1x_1^{-3}(x_0^{-1-a}x_2-x_0^{-2-a})=3y_1x_0^{-2-a}x_1^{-3}(x_0x_2-1)=3y_1y_2x_0^{-2-a}x_1^{b-3}.\qedhere\]
\end{proof}

This completes Case 5, leaving only Case 6:  $B'=\B1{-3}$.
In this case, the $\g$-vectors of diagonal cluster variables are $[-1,1]$, $[-1,2]$, $[-1,3]$, and $[-2,3]$.
Lemmas~\ref{thet -11 lem} and~\ref{thet -12 lem}, together with the following two lemmas, take care of Case 6 and complete the proof of Proposition~\ref{implies hom phen 2x2 part 1} and thus Theorem~\ref{hom phen 2x2 part 1}.

\begin{lemma}\label{thet -13 lem}
If $B=\B ba$ and $B'=\B dc$ with $a\le c\le-3$ and $b\ge d\ge1$, then $\nu_\z(\thet'_{[-1,3]})=\thet_{[-1,3]}$.
\end{lemma}

\begin{lemma}\label{thet -23 lem}
If $B=\B ba$ and $B'=\B 1{-3}$ with $a\le-3$ and $b\ge1$, then 
\[\nu_\z(\thet'_{[-2,3]})=\begin{cases}
\thet_{[-2,3]}&\text{if }b=1,\text{ or}\\
\thet_{[-2,3]}+3y_1y_2x_0^{-3-a}x_1^{b-2}&\text{if }b\ge2.\end{cases}\]
\end{lemma}
\begin{proof}
Proposition~\ref{specific} says that 
\[\thet'_{[-2,3]}=(x'_1)^{-2}(x'_2)^3+2y'_1(x'_1)^{-2}+3y'_1y'_2(x'_1)^{-1}+(y'_1)^2(x'_0)^3(x'_1)^{-2}\text{, so}\]
\[\nu_\z(\thet'_{[-2,3]})=x_1^{-2}x_2^3+2y_1x_0^{-3-a}x_1^{-2}+3y_1y_2x_0^{-3-a}x_1^{b-2}+y_1^2x_0^{-3-2a}x_1^{-2}.\]
If $b=1$, then (since the case where $a=-3$ is a tautology) we apply Proposition~\ref{theta m2 a} to compute
\begin{align*}
\thet_{[-2,3]}&=x_1^{-2}x_2^{3}+\sum_{i=1}^{2}\sum_{j=0}^{3}\binom{2-j}i\binom{3}j y_1^iy_2^jx_0^{-3-ai}x_1^{-2+j}\\
&=x_1^{-2}x_2^{3}+2y_1x_0^{-3-a}x_1^{-2}+3y_1y_2x_0^{-3-a}x_1^{-1}+y_1^2x_0^{-3-2a}x_1^{-2}\\
&=\nu_\z(\thet'_{[-2,3]}).
\end{align*}
If $b\ge2$, then we apply Proposition~\ref{theta m1 b} to compute
\[\thet_{[-2,3]}=x_1^{-2}(x_2^3+2y_1x_0^{-3-a}+y_1^2x_0^{-3-2a})=\nu_\z(\thet'_{[-2,3]})-3y_1y_2x_0^{-3-a}x_1^{b-2}.\qedhere\]
\end{proof}

\subsection{Acyclic finite type}\label{acyc fin sec}
In this section, we point out a simplification in the description of the map $\nu_\z$ in the cases where $B$ is acyclic.
We then describe how the simpler description of $\nu_\z$ allows for shortcuts in the computations that verify Theorem~\ref{hom phen fin conj} (Phenomenon~\ref{hom phen} for acyclic finite-type exchange matrices) up to $8\times8$ matrices.
The following easy observation is a well-known feature of the acyclic case.

\begin{prop}\label{feature}
Suppose $B$ is an acyclic exchange matrix.
For each ${i=1,\ldots,n}$ there exists a cluster variable $a_i$ with $\g$-vector $-\e_i$.
For each ${k=1,\ldots,n}$, there exists a cluster $X$ with $a_i\in X$ if $b_{ik}>0$ and $x_i\in X$ if $b_{ik}<0$.
\end{prop}

Here $\e_i$ is the $i\th$ standard unit basis vector in $\reals^n$.
Recall that the $\g$-vector is the $\integers^n$-grading of $\A_\bullet(B)$ such that the $\g$-vector of $x_k$ is $\e_k$ and the $\g$-vector of $y_k$ is the negative of the $k\th$ column of $B$.
A \newword{cluster monomial} is a monomial in the cluster variables in some cluster.

Using Proposition~\ref{feature}, we prove the following fact, which lets us factor the map $\nu_\z$ in the acyclic case.
Our notation $\nu_\z$ does not explicitly show the dependence on $B$ and $B'$.
We temporarily make the notation more explicit by writing $\nu^{B',B}_\z$.

\begin{prop}\label{factor nu z}
Suppose $B$ is acyclic and dominates $B'$, which dominates $B''$.
If $\nu^{B'',B'}_\z$ sends cluster variables of $\A_\bullet(B'')$ to cluster variables of $\A_\bullet(B')$ and $\nu^{B',B}_\z$ sends cluster variables of $\A_\bullet(B'')$ to cluster variables of $\A_\bullet(B')$, then the composition $\A_\bullet(B'')\toname{\nu_\z}\A_\bullet(B')\toname{\nu_\z}\A_\bullet(B)$ sends cluster variables of $\A_\bullet(B'')$ to cluster variables of $\A_\bullet(B)$.
This composition equals $\nu_\z^{B'',B}:\A_\bullet(B'')\to\A_\bullet(B)$.
\end{prop}
\begin{proof}
The first statement is immediate.
Proposition~\ref{feature} implies that the elements $z_i^{B'',B'}$ in the definition of $\nu_\z^{B'',B'}$ are cluster monomials, and similarly for the elements $z_i^{B',B}$ in the definition of $\nu_\z^{B',B}$.
Furthermore, $\nu_\z^{B',B}$ acts on the cluster monomials $z_i^{B'',B'}$ simply by mapping each cluster variable to the corresponding cluster variable.
\end{proof}

Suppose that $B$ is an acyclic and of finite type and that $B$ dominates~$B'$.
Theorem~\ref{hom phen fin conj} asserts that $\nu_\z$ is an injective, $\g$-vector-preserving ring homomorphism from $\A_\bullet(B')$ to $\A_\bullet(B)$, sending each cluster variable in $\A_\bullet(B')$ to a cluster variable in $\A_\bullet(B)$.
To verify the theorem, we can first reduce to the case of an irreducible exchange matrix $B$.
(An exchange matrix is reducible if it can be written in block-diagonal form with more than one block.)
Furthermore, we will see that it is enough to check the special case where $B$ dominates $B'$ but they differ either in exactly one position or they differ in that $B$ has $\pm1$ in a pair of symmetric positions where $B'$ has zero.
To check the special case, we need only check that $\nu_\z$ sends each cluster variable to a cluster variable and then apply Propositions~\ref{hom from initial} and~\ref{one antip}.

If we have checked the special case, then in general, we can find a sequence of exchange matrices, starting at $B'$ and ending at $B$ such that each two adjacent matrices in the sequence belong to the special case.  
Proposition~\ref{factor nu z} ensures that, composing the maps $\nu_\z$ at each step in the sequence, we obtain the map $\nu_\z^{B',B}$, which therefore has the desired properties.

We have verified the special case computationally (and thus we deduce the general case) for exchange matrices $B$ that are $8\times8$ or smaller.  
Recall that acyclic exchange matrices of finite type are exactly those such that $\Cart(B)$ is a Cartan matrix of finite type.
Since we have proved the theorem up to $8\times8$ exchange matrices, the theorem is proved whenever $\Cart(B)$ is of exceptional finite type (E, F, or G).
Below, we prove the theorem in types A and D using the surfaces model.
As already mentioned,  \cite[Theorem~4.1.5]{Viel} completes the proof of Theorem~\ref{hom phen fin conj} by handling types B and C.

\subsection{Resection of surfaces}\label{resect hom sec}
In this section, we give more background on the surfaces model (building on Section~\ref{bij surf sec}) and prove Theorem~\ref{hom phen surf thm}.

A \newword{tagged arc} is an arc that does not cut out a once-punctured monogon and that, at each endpoint incident to a puncture, is marked (or ``tagged'') either \newword{notched} or \newword{plain}, with the condition that if both ends of the arc are at the same puncture, they must have the same tagging.
Two tagged arcs are \newword{compatible} if either (1) the corresponding untagged arcs are distinct and compatible and the two arcs have the same tagging at any endpoint they have in common, or (2) both correspond to the same untagged arc, which has two distinct endpoints and is not contained in a component of $(\S,\M)$ that is a once-punctured monogon, and the taggings of the two tagged arcs disagree at exactly one endpoint.
(The two tagged arcs in a once-punctured monogon are not compatible.)
A \newword{tagged triangulation} is a maximal collection of pairwise compatible tagged arcs.
Each tagged triangulation has the same number of tagged arcs.
The operation of reversing all taggings at a given puncture sends a tagged triangulation to a tagged triangulation.
By iterating this operation, we can remove all notched taggings except that some punctures are incident to two tagged arcs with opposite taggings at that puncture (compatible as in (2) above).
When notches have been thus maximally removed, the resulting tagged triangulation corresponds to an ordinary triangulation as follows:
Each arc without notches becomes an ordinary arc.
Each arc with a notch (necessarily only on one end) becomes a loop with endpoints at the unnotched end, tracing around the tagged arc.
This loop becomes the non-fold edge of a self-folded triangle, as shown in Figure~\ref{tau fig}.
\begin{figure}[ht]
\scalebox{0.8}{\includegraphics{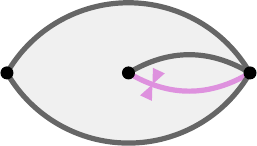}}\qquad\raisebox{16pt}{\LARGE$\longrightarrow$}\qquad\scalebox{0.8}{\includegraphics{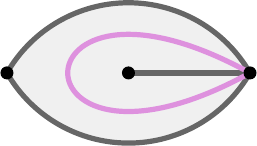}}
\caption{The last step in making an ordinary triangulation from a tagged triangulation}
\label{tau fig}
\end{figure}

Passing from a tagged triangulation to an ordinary triangulation as described above does not change the associated exchange matrix \cite[Definition~9.6]{cats1}.
Thus in proving statements about exchange matrices arising from surfaces, we may as well assume that $T$ is an ordinary triangulation, but we pass freely between $T$ and the corresponding tagged triangulation as in Figure~\ref{tau fig}.

The definition of shear coordinates can be extended to define shear coordinates of a quasi-lamination $L$ with respect to a tagged triangulation $T$.  
To compute $b(T,L)$, we perform the operation of reversing all taggings at a given puncture until as many notches are removed as possible.
For each reversal of taggings, the quasi-lamination $L$ is also altered by reversing the directions of all spirals into that puncture, obtaining a new quasi-lamination $L'$.
The new tagged triangulation corresponds to an ordinary triangulation $T'$ as in Figure~\ref{tau fig}.
The shear coordinate $b_\gamma(T,L)$ for a tagged arc $\gamma$ in $T$ is defined to be the shear coordinate $b_{\gamma'}(T',L')$, where $\gamma'$ is the arc in $T'$ corresponding to $\gamma$.

In a marked surface $(\S,\M)$ with tagged triangulation $T$, the cluster variables in $\A_\bullet(B(T))$ are in bijection with tagged arcs.
(In fact, when $\S$ has no boundary and $|\M|=1$, the cluster variables are in bijection with the tagged arcs having only plain tags, but we will not need to consider that case in this paper.)
We write $x_\gamma$ for the cluster variable associated to a tagged arc $\gamma$.

The coefficients $y_i$ can be identified with the certain allowable curves obtained from the arcs in the triangulation $T$.
Specifically, if $\gamma$ is a tagged arc in $T$, then the \newword{elementary lamination} $L_\gamma$ associated to $\gamma$ is a curve that agrees with $\gamma$ except very close to the endpoints.
If $\gamma$ has an endpoint that is on a boundary component, then as one follows $L_\gamma$ toward the endpoint, it misses the endpoint by veering slightly to the right.  
If $\gamma$ is tagged plain at a puncture, then $L_\gamma$ again misses to the right and then spirals counterclockwise into the puncture.
If $\gamma$ is tagged notched at a puncture, then $L_\gamma$ misses to the left and spirals clockwise.
We write $L_T$ for the set $\set{L_\gamma:\gamma\in T}$ of elementary laminations associated to $T$.
It is easily verified that, for each tagged arc $\gamma$ in $T$, the shear coordinates of $L_\gamma$ are $1$ in the $\gamma$ position and $0$ elsewhere.

Since cluster variables are in bijection with tagged arcs, elements of $\A_\bullet(B(T))$ can be represented (non-uniquely) as sums of terms, each term being a monomial in the $x_\alpha$ for tagged arcs $\alpha$ (tagged plain when $(\S,\M)$ is a once-punctured surface with no boundary components) times a monomial in the $L_T$.

The tagged arcs are also in bijection with non-closed allowable curves, and we write $\kappa$ for this bijection.
The map $\kappa$ is like the map taking $\gamma$ to the elementary lamination $L_\gamma$, except with right and left reversed.
The curve $\kappa(\gamma)$ agrees with the tagged arc $\gamma$ except very close to the endpoints.
At an endpoint of $\gamma$ on the boundary or at a puncture tagged plain, $\kappa(\gamma)$ misses to the left and either hits the boundary or spirals clockwise.
At a puncture tagged notched, $\kappa(\gamma)$ misses to the right and spirals counterclockwise.
The map $\kappa$ induces a bijection from tagged triangulations to maximal sets of pairwise-compatible non-closed allowable curves, which index the full-dimensional cones in the rational quasi-lamination fan $\F_\rationals(T)$).

By \cite[Proposition~5.2]{unisurface}, for a tagged arc $\gamma$, the $\g$-vector of the cluster variable~$x_\gamma$ is $-\b(T,\kappa(\gamma))$, the negative of the shear coordinates of $\kappa(\gamma)$ with respect to $T$.
Suppose $(\S,\M)$ is a marked surface with a triangulation $T$, suppose $(\S',\M')$ is obtained by a resection compatible with $T$, and suppose $T'$ is the triangulation induced by $T$.
Supposing also that every component of $(\S,\M)$ and $(\S',\M')$ either has the Null Tangle Property or is a null surface, Theorem~\ref{ref phen resect} implies that every rational ray of the rational quasi-lamination fan $\F_\rationals(T')$ is also a ray of $\F_\rationals(T)$.
Since rational rays of the rational quasi-lamination fan are spanned by the shear coordinates of allowable curves, we see that for every allowable curve $\lambda'$ in $(\S',\M')$, there is a (unique) allowable curve $\lambda$ in $(\S,\M)$ such that $\b(T,\lambda)=\b(T',\lambda')$.

Our goal is to prove Theorem~\ref{hom phen surf thm}, and accordingly we now restrict our attention to the case where $(\S,\M)$ is a once-punctured or unpunctured disk and $T$ is a triangulation such that $B(T)$ is acyclic.
Thus also $(\S',\M')$ is a union of once-punctured or unpunctured disks and $B(T')$ is acyclic.
This simplifies the situation considerably, not least because in this case there are no closed allowable curves.
In particular, the allowable curves are in bijection with the tagged arcs by the map $\kappa$ described above.
We define a map $\chi$ on tagged arcs by letting $\chi$ send a tagged arc $\gamma'$ in $(\S',\M')$ to the tagged arc $\gamma$ in $(\S,\M)$ such that $\kappa(\gamma')$ and $\kappa(\gamma)$ have the same shear coordinates.
Equivalently, $\chi$ sends $\gamma'$ to the tagged arc $\gamma$ such that the $\g$-vector of the cluster variable $x_\gamma$ in $\A_\bullet(B(T))$ equals the $\g$-vector of $x_{\gamma'}$ in $\A_\bullet(B(T'))$.
Accordingly, we re-use the symbol $\chi$ for the map on cluster variables in $\A_\bullet(B(T))$ sending $x_\gamma'$ to $x_{\chi(\gamma')}$ for a tagged arc $\gamma'$ in $(\S',\M')$.
Furthermore, we extend $\chi$ to a map on $\Var_\bullet(B(T'))\cup\set{\y'}$ by letting it agree with $\nu_\z$ on $\set{\y'}$.

The requirement that $B(T)$ is acyclic implies that (1) the puncture (if there is one) is contained in a once-punctured digon of $T$, with one side of the digon on the boundary and (2) every triangle of $T$ has at least one edge on the boundary or has two edges inside the once-punctured digon.
To prove Theorem~\ref{hom phen surf thm}, we do not need to check every possible resection.
Instead, since every resection gives a dominance relation and since Theorem~\ref{hom phen surf thm} is an assertion about matrices dominated by $B(T)$, it is enough to check every matrix $B'$ dominated by $B(T)$.
Furthermore, by Proposition~\ref{factor nu z}, we need only check for each pair $b_{\alpha\beta}=-b_{\beta\alpha}=\pm1$ of entries in $B(T)$, that the theorem holds when $B'$ is obtained by setting $b_{\alpha\beta}$ and $b_{\beta\alpha}$ to zero.
To accomplish this, it is enough to consider three cases:

\medskip\noindent 
\textbf{Case 1.}
$(\S',\M')$ is obtained by resecting a single arc $\alpha\in T$ with two distinct endpoints, both on the boundary and the puncture in $(\S',\M')$ is in the component not containing $\alpha$.

\medskip\noindent 
\textbf{Case 2.} 
$(\S',\M')$ is obtained by resecting an arc $\alpha\in T$ that is the non-fold edge of a self-folded triangle in $T$ and the point $p_\alpha$ is in the self-folded triangle.

\medskip\noindent 
\textbf{Case 3.} 
$(\S',\M')$ is obtained by resecting two arcs incident to the puncture in the once-punctured digon of $T$.
(The resection must be as pictured in Figure~\ref{digon resect config}, or the mirror image of that picture.)
\medskip

In each of the three cases, we accomplish more than what is stated in Theorem~\ref{hom phen surf thm}:
We make no global requirement of acyclicity, but rather disallow erasing edges that are contained in oriented cycles.
In each case, the outline is as follows:
\begin{enumerate}
\item Explicitly describe the map $\chi$ on tagged arcs.
(In every case, $\chi$ maps each tagged arc in the triangulation $T'$ to the corresponding tagged arc in $T$, and we won't mention these cases separately in each proof.)
\item Explicitly describe the map $\chi$ (i.e.\ $\nu_\z$) on $\set{\y'}=\set{L'_\zeta:\zeta\in T'}$.
Here $L'_\zeta$ means the elementary lamination in $(\S',\M')$ as opposed to in $(\S,\M)$.
Recall that $\nu_\z(L'_\zeta)=L_\zeta z_\zeta$, where $z_\zeta$ is the cluster monomial whose $\g$-vector is the $\zeta$-column of $B(T)$ minus the $\zeta$-column of $B(T')$.
\item Check that $\chi$ takes every exchange relation in $\A_\bullet(B(T'))$ to an exchange relation in $\A_\bullet(B(T))$.
\item Conclude by Proposition~\ref{hom from cl vars} that $\chi$ extends to a ring homomorphism.  (Since $\chi$ agrees with $\nu_\z$ on $\Var_\bullet(B(T'))\cup\set{\y'}$, it coincides with $\nu_\z$.)
\item Apply Proposition~\ref{one antip} to conclude that $\nu_\z$ is injective.
\end{enumerate}

We begin with the following proposition, which proves Theorem~\ref{hom phen surf thm} in Case 1.

\begin{prop}\label{case 1 prop}
Suppose $(\S,\M)$ is a once-punctured or unpunctured disk and suppose $(\S',\M')$ and $T'$ are obtained from $(\S,\M)$ and $T$ by a resection compatible with $T$, along a single arc $\alpha$ with distinct endpoints, both on the boundary. 
Suppose also that the point $p_\alpha$ used to construct the resection is in a triangle of $T$ having exactly one edge on the boundary and that the puncture in $(\S',\M')$ is in the component not containing the arc $\alpha$.
Then both parts of Phenomenon~\ref{hom phen} occur.
\end{prop}
\begin{proof}  
We will refer to the labels on some arcs in $T$, and corresponding arcs in $T'$, shown in the first row of Figure~\ref{chi fig}.
For generality, neither $\beta$ nor $\gamma$ is shown as a boundary segment, but by hypothesis exactly one of them them is.
\begin{figure}[ht]
\begin{tabular}{ccc}
\scalebox{0.75}{\includegraphics{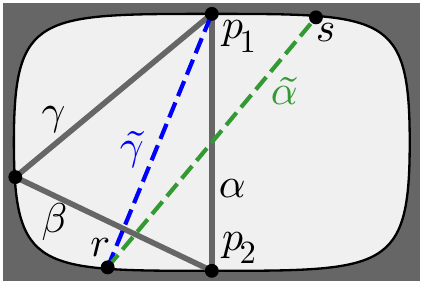}}&\raisebox{29pt}{\LARGE$\xrightarrow{\text{\small resect}}$}&\scalebox{0.75}{\includegraphics{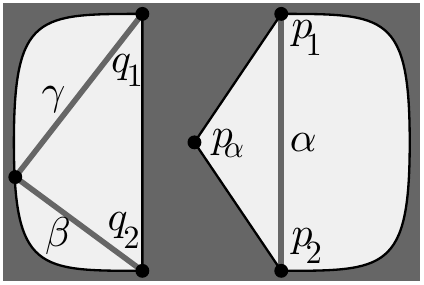}}\\[2pt]
\scalebox{0.75}{\includegraphics{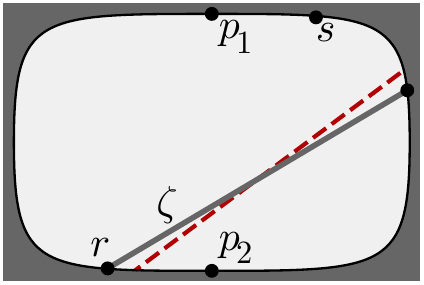}}&\raisebox{29pt}{\LARGE$\xleftarrow{\chi}$}&\scalebox{0.75}{\includegraphics{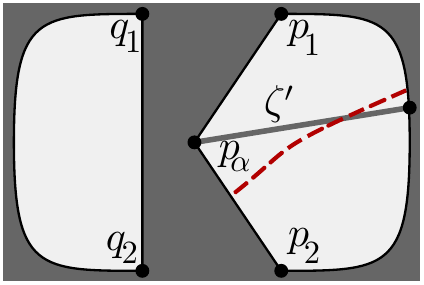}}\\[2pt]
\scalebox{0.75}{\includegraphics{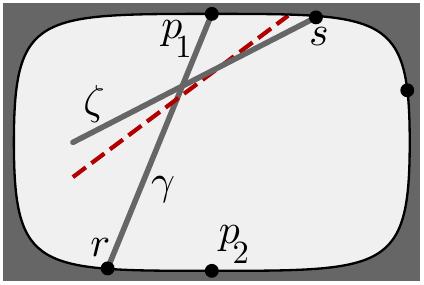}}&\raisebox{29pt}{\LARGE$\xleftarrow{\chi}$}&\scalebox{0.75 }{\includegraphics{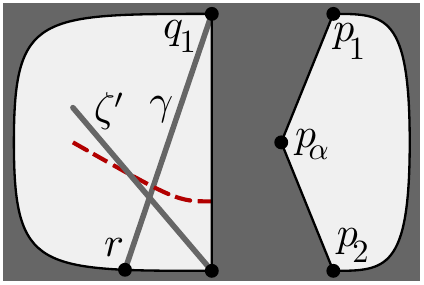}}
\end{tabular}
\caption{Illustrations of the map $\chi$ on tagged arcs, Case 1}
\label{chi fig}
\end{figure}

Let $\zeta'$ be a tagged arc in $(\S',\M')$.
If neither endpoint of $\zeta'$ is $p_\alpha$ or~$q_2$, then $\zeta$ is obtained from $\zeta'$ by the natural inclusion (preserving taggings) taking $(\S',\M')$ minus the triangle $p_1p_2p_\alpha$ into $(\S,\M)$.
If an endpoint of $\zeta'$ is $p_\alpha$, then we delete the part of $\zeta'$ contained in the triangle $p_1p_2p_\alpha$ and include what remains of $\zeta'$ into $(\S,\M)$
The included arc is attached to the point $r$ that is closest to $p_2$ (moving from $p_2$ while keeping the surface on the right), as illustrated at the second row of Figure~\ref{chi fig}.
The curves $\kappa(\zeta')$ and $\kappa(\zeta)$ are shown dotted in the figure.
If an endpoint of $\zeta'$ is at $q_2$, then there are two cases, depending on whether $\beta$ or $\gamma$ is a boundary segment.
If $\gamma$ is a boundary segment, then $\zeta$ is the inclusion of $\zeta'$.
If $\beta$ is a boundary segment, then the inclusion of $\zeta'$ is cut where it crosses $\beta$, the piece incident to $p_2$ is discarded, and the remaining piece is connected to $s$, the point that is closest to $p_1$, moving from $p_1$ keeping the surface on the right.
(This case is illustrated in the third row of Figure~\ref{chi fig}.
In the pictures, the curves $\zeta$, $\chi(\zeta)$, $\zeta'$, and $\chi(\zeta')$ are shown ``dangling'' because they might end at the boundary or at the marked point.)

We next describe how $\chi$ (or $\nu_\z$) acts on $\set{\y'}=\set{L'_\zeta:\zeta\in T'}$ by describing the elements $z_\zeta$.
As before, $r\in\M$ is the marked point closest to $p_2$ along the boundary and $s\in\M$ is the marked point closest to $p_1$ along the boundary.
If $\beta$ is a boundary segment, then $\tilde\gamma=\gamma$ and $z_\alpha=x_\gamma$ and $z_\gamma=x_{\tilde\alpha}$.
If $\gamma$ is a boundary segment, then $z_\alpha=x_{\tilde\gamma}$ and $z_\beta=x_\alpha$.

\begin{figure}[ht]
\begin{tabular}{ccc}
\scalebox{0.75}{\includegraphics{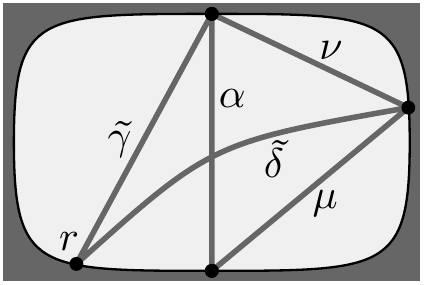}}&\raisebox{29pt}{\LARGE$\xleftarrow{\chi}$}&\scalebox{0.75}{\includegraphics{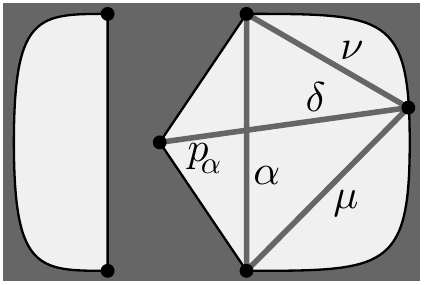}}
\end{tabular}
\caption{Arcs in exchange relations, Case 1}
\label{chi single}
\end{figure}
We now show that $\chi$ takes every exchange relation of $\A_\bullet(B(T'))$ to an exchange relation in $\A_\bullet(B(T))$.
We can deal with such an exchange relation without specifying whether $\beta$ or $\gamma$ is a boundary segment, because in any case, neither $L'_\beta$ nor $L'_\gamma$ is involved in the exchange relation.
If neither of the tagged arcs being exchanged is $\alpha$, then $L'_\alpha$ is also not involved.
Mapping the exchange relation to $(\S,\M)$ by $\chi$, we obtain precisely the exchange relation that exchanges the corresponding arcs in $(\S,\M)$.
If one of the arcs being exchanged is $\alpha$, then the situation is as illustrated in the right picture of Figure~\ref{chi single}. 
The exchange relation is
\[x'_\alpha x'_\delta=x'_\mu\prod_{\zeta\in T'}(L'_\zeta)^{[b_\alpha(\widebar T',L'_\zeta)]_+}+x'_\nu\prod_{\zeta\in T'}(L'_\zeta)^{[-b_\alpha(\widebar T',L'_\zeta)]_+},\]
where $\widebar T'$ is any triangulation of $(\S',\M')$ containing the arcs $\alpha$ and whichever of $\mu$ and $\nu$ are not boundary segments.
(If $\mu$ and/or $\nu$ is a boundary segment, then we set $x_\mu$ and/or $x_\nu$ equal to $1$.)
We have $b_\alpha(\widebar T',L'_\alpha)=1$.
Since the allowable curves in $L_{T'}$ are pairwise compatible, no curve $\zeta\in L'_{T'}$ has $b_\alpha(\widebar T',L'_\zeta)<0$.
Thus we can rewrite the exchange relation as
\[x'_\alpha x'_\delta=x'_\mu L'_\alpha \prod_{\zeta\in T'\setminus\set{\alpha}}\!\!(L'_\zeta)^{[b_\alpha(\widebar T',L'_\zeta)]_+}\,+\,x'_\nu.\]
Since $L'_\beta$ and $L'_\gamma$ do not appear in the product, applying $\chi$ yields
\[x_\alpha x_{\tilde\delta}=x_\mu x_{\tilde\gamma} L_\alpha \prod_{\zeta\in T'\setminus\set{\alpha}}\!\!L_\zeta^{[b_\alpha(\widebar T',L'_\zeta)]_+}\,+\,x_\nu.\]
Take $\widebar T$ to be a triangulation of $(\S,\M)$ agreeing with $\widebar T'$ on the inclusion of the right component of $(\S',\M')$ minus the triangle $p_1p_2p$ into $(\S,\M)$ and containing $\gamma'$.
We see that $b_\alpha(\widebar T',L'_\zeta)=b_\alpha(\widebar T,L_\zeta)$ for every $\zeta\in T'$, so that the relation becomes 
\[x_\alpha x_{\tilde\delta}=x_\mu x_{\tilde\gamma}\prod_{\zeta\in T}L_\zeta^{[b_\alpha(\widebar T,L_\zeta)]_+}\,+\,x_\nu,\]
which is an exchange relation in $\A_\bullet(B(T))$, as shown in the left picture of Figure~\ref{chi single}.

To assist in considering further cases, we recast the above treatment pictorially.
The top row of Figure~\ref{exch pic} shows the exchange relation in $(\S',\M')$ in the case where one of the arcs being exchanged is $\alpha$, with each element $x'_\zeta$ represented by the arc $\zeta$.
\begin{figure}
\begin{tabular}{ccccc}
\scalebox{0.75}{\includegraphics{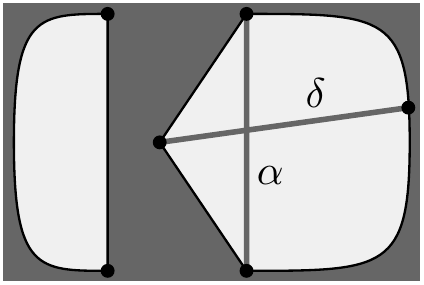}}&\raisebox{30pt}{$=$}&\scalebox{0.75}{\includegraphics{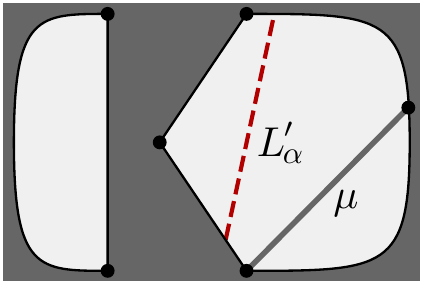}}&\raisebox{30pt}{$+$}&\scalebox{0.75}{\includegraphics{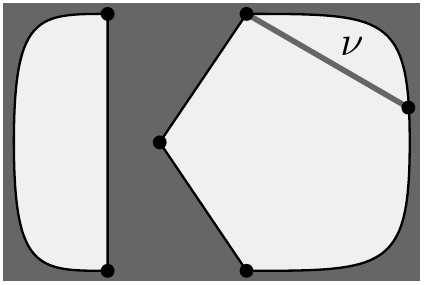}}\\[-3pt]
$\chi\downarrow$&&$\chi\downarrow$&&$\chi\downarrow$\\[1pt]
\scalebox{0.75}{\includegraphics{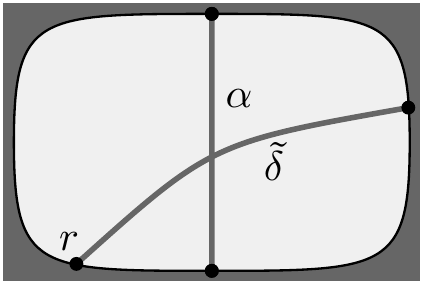}}&\raisebox{30pt}{$=$}&\scalebox{0.75}{\includegraphics{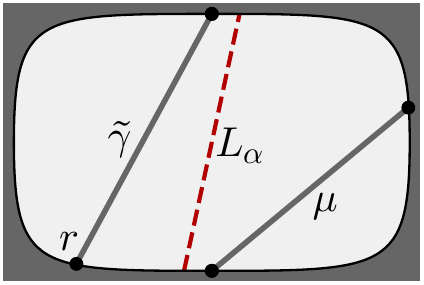}}&\raisebox{30pt}{$+$}&\scalebox{0.75}{\includegraphics{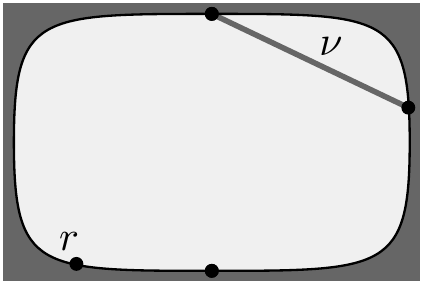}}\\[3pt]\hline\hline\\[-5pt]
\scalebox{0.75}{\includegraphics{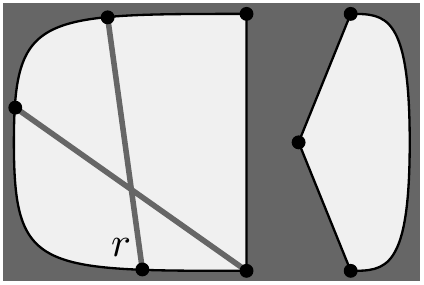}}&\raisebox{30pt}{$=$}&\scalebox{0.75}{\includegraphics{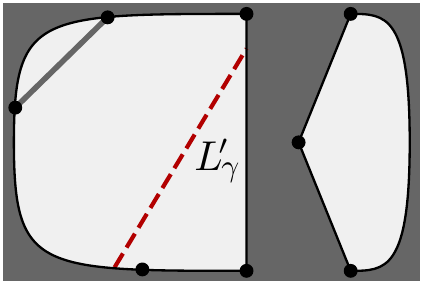}}&\raisebox{30pt}{$+$}&\scalebox{0.75}{\includegraphics{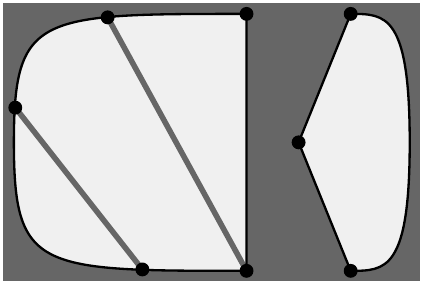}}\\[-3pt]
$\chi\downarrow$&&$\chi\downarrow$&&$\chi\downarrow$\\[1pt]
\scalebox{0.75}{\includegraphics{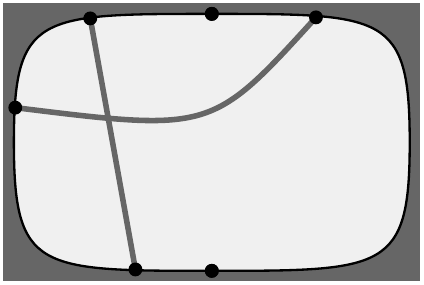}}&\raisebox{30pt}{$=$}&\scalebox{0.75}{\includegraphics{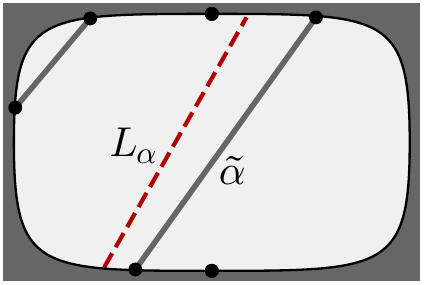}}&\raisebox{30pt}{$+$}&\scalebox{0.75}{\includegraphics{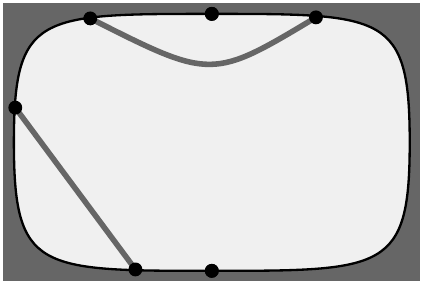}}\\[3pt]\hline\hline\\[-5pt]
\scalebox{0.75}{\includegraphics{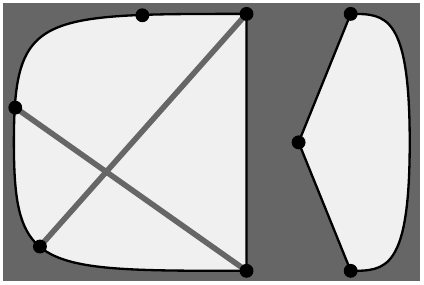}}&\raisebox{30pt}{$=$}&\scalebox{0.75}{\includegraphics{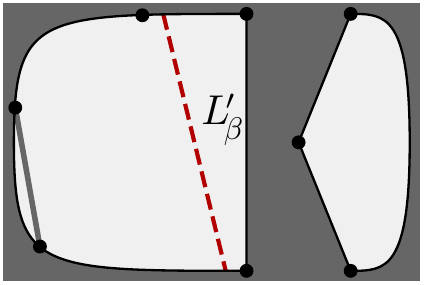}}&\raisebox{30pt}{$+$}&\scalebox{0.75}{\includegraphics{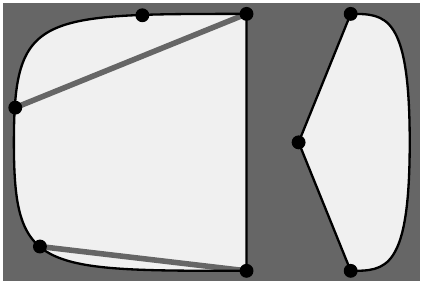}}\\[-3pt]
$\chi\downarrow$&&$\chi\downarrow$&&$\chi\downarrow$\\[1pt]
\scalebox{0.75}{\includegraphics{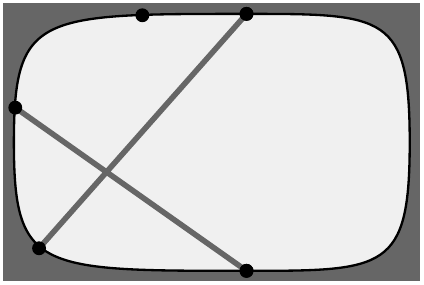}}&\raisebox{30pt}{$=$}&\scalebox{0.75}{\includegraphics{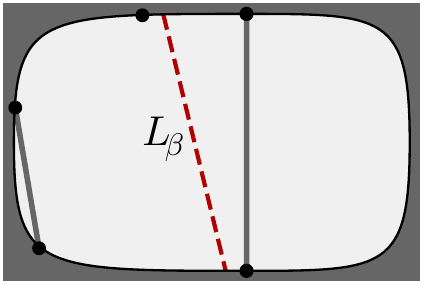}}&\raisebox{30pt}{$+$}&\scalebox{0.75}{\includegraphics{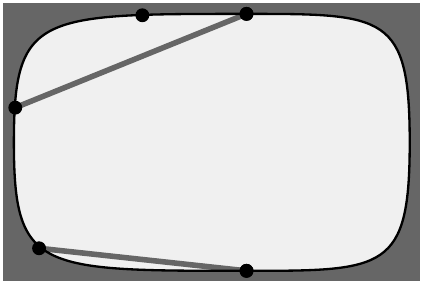}}
\end{tabular}
\caption{The action of $\chi$ on some exchange relations, Case 1}
\label{exch pic}
\end{figure}
Aside from $L'_\alpha$ (shown dotted), the elementary laminations involved in the relation are not shown.
The next row of the figure shows the analogous representation of the image of the relation under $\chi$.

Now consider an exchange relation in the component of $(\S',\M')$ not containing~$p$.

First suppose that $\beta$ is a boundary segment.
If neither tagged arc being exchanged is incident to $p_2$, then $L'_\gamma$ does not appear in the exchange relation and $\chi$ acts by inclusion on the tagged arcs involved in the exchange relation and acts trivially on each elementary lamination in the exchange relation.
The result is an exchange relation in $\A_\bullet(B(T))$.
If one of the tagged arcs being exchanged is incident to $p_2$, then $L'_\gamma$ appears in the exchange relation if and only if the other arc being exchanged is incident to $r$.
When the other arc being exchanged is not incident to $r$, as $\chi$ moves the endpoints of arcs at $p_2$ to $s$, the exchange relation is taken to an exchange relation in $\A_\bullet(B(T))$.
(The exchange relation in $\A_\bullet(B(T))$ still does not involve $L_\gamma$, and also does not involve $L_\zeta$ for any arc $\zeta$ from the other component of $\S'$.)
When the other arc being exchanged \emph{is} incident to $r$, as $\chi$ moves the endpoints of arcs at $p_2$ to $s$, the exchange relation picks up a factor $x_{\tilde\alpha}$ in one of its right-side terms, as illustrated in the third and fourth rows of Figure~\ref{exch pic}.
(In $(\S',\M')$ the boundary segment $\beta$ is positioned relative to the exchange relation exactly where $\tilde\alpha$ is positioned relative to the exchange relation in $(\S,\M)$.)
In all of these cases where $\beta$ is a boundary segment, there are various forms the exchange relations can take, depending on how the exchanged arcs are positioned relative to the puncture and the boundary.
(For example, one or more of the arcs appearing on the right side may be a boundary segment, thus disappearing from the relation.
Or, one of the arcs on the right side may be replaced by a pair of arcs to the puncture coinciding except for opposite taggings at the puncture.
The latter happens when the two arcs being exchanged share an endpoint.)
These details are preserved when $\chi$ is applied.
We see (as illustrated in the fourth row of Figure~\ref{exch pic} that $\chi$ maps the exchange relation in $\A_\bullet(B(T'))$ to an exchange relation in $\A_\bullet(B(T))$.

Next, suppose $\gamma$ is a boundary segment.
If neither tagged arc being exchanged is incident to $p_2$, then $L_\beta$ does not appear in the exchange relation and $\chi$ acts by inclusion on the tagged arcs involved in the exchange relation and acts trivially on each elementary lamination in the exchange relation.
The result is again an exchange relation in $\A_\bullet(B(T))$.
If one of the tagged arcs being exchanged is incident to $p_2$, then $L_\gamma$ appears in the exchange relation if and only if the other arc being exchanged is incident to $p_1$.
This time, $\chi$ does not move the endpoints of arcs at $p_2$, but again $\chi$ takes the exchange relation to an exchange relation in $\A_\bullet(B(T))$, as illustrated in the fifth and sixth rows of Figure~\ref{exch pic}.

\begin{figure}
\begin{tabular}{ccc}
\scalebox{0.75}{\includegraphics{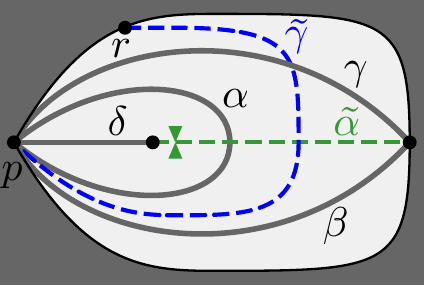}}&\raisebox{29pt}{\LARGE$\xrightarrow{\text{\small resect}}$}&\scalebox{0.75}{\includegraphics{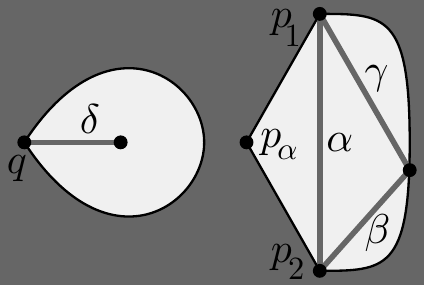}}\\[2pt]
\scalebox{0.75}{\includegraphics{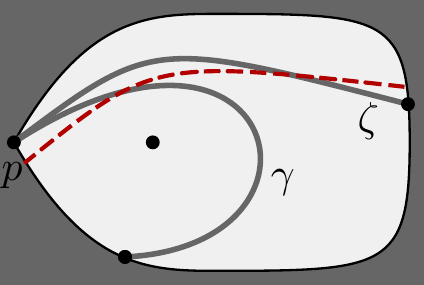}}&\raisebox{29pt}{\LARGE$\xleftarrow{\chi}$}&\scalebox{0.75}{\includegraphics{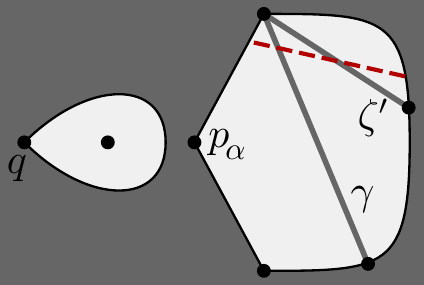}}\\[2pt]
\scalebox{0.75}{\includegraphics{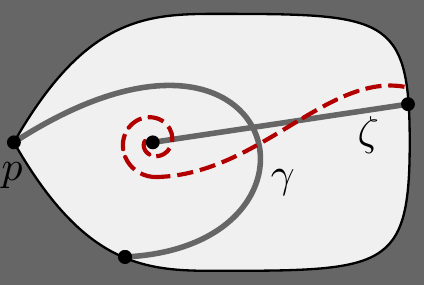}}&\raisebox{29pt}{\LARGE$\xleftarrow{\chi}$}&\scalebox{0.75}{\includegraphics{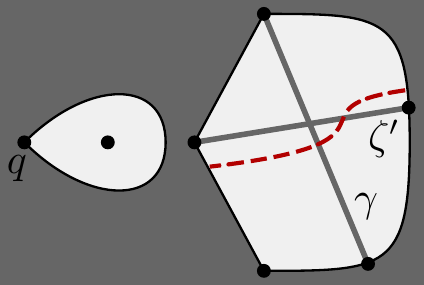}}\\[2pt]
\scalebox{0.75}{\includegraphics{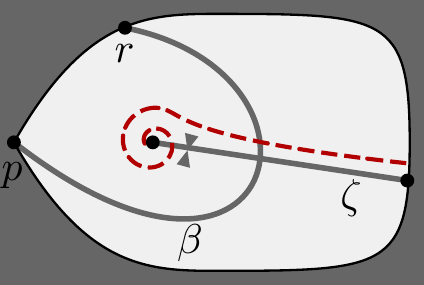}}&\raisebox{29pt}{\LARGE$\xleftarrow{\chi}$}&\scalebox{0.75}{\includegraphics{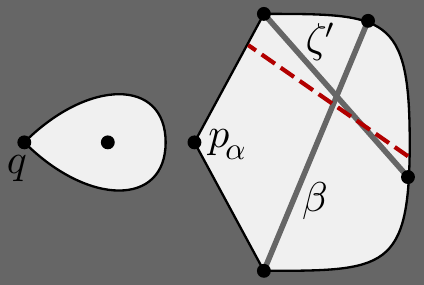}}\\[2pt]
\scalebox{0.75}{\includegraphics{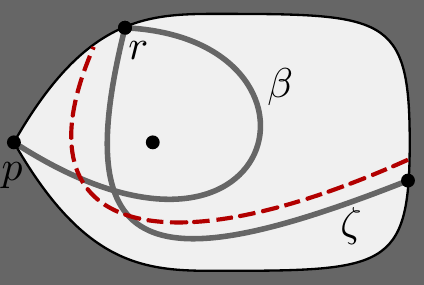}}&\raisebox{29pt}{\LARGE$\xleftarrow{\chi}$}&\scalebox{0.75}{\includegraphics{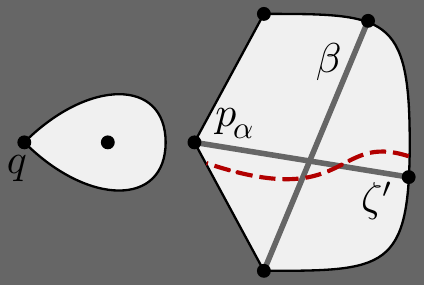}}\\[2pt]
\scalebox{0.75}{\includegraphics{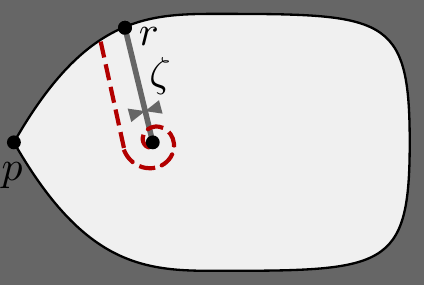}}&\raisebox{29pt}{\LARGE$\xleftarrow{\chi}$}&\scalebox{0.75}{\includegraphics{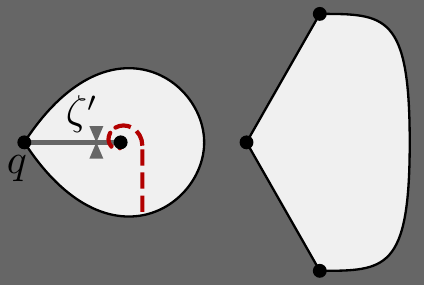}}
\end{tabular}
\caption{Illustrations of the map $\chi$ on tagged arcs, Case 2}
\label{chi mon fig}
\end{figure}

The rest of the proof follows the outline given above.
(We can apply Proposition~\ref{one antip} because of the hypothesis that either $\beta$ or $\gamma$ is a boundary segment.)
\end{proof}

The following proposition proves Theorem~\ref{hom phen surf thm} in Case 2.

\begin{prop}\label{case 2 prop}
Suppose $(\S,\M)$ is a once-punctured disk and $(\S',\M')$ is obtained by resecting an arc $\alpha\in T$ that is the non-fold edge of a self-folded triangle in $T$ that is in turn contained in a once-punctured digon having exactly one edge on the boundary.
Suppose also that the point $p_\alpha$ is in the self-folded triangle.
Then both parts of Phenomenon~\ref{hom phen} occur.
\end{prop}
\begin{proof}  
We describe $\chi$ on tagged arcs.
The situation is shown in the first row of Figure~\ref{chi mon fig}, 
where again for generality neither $\gamma$ nor $\beta$ is pictured as a boundary arc, although by hypothesis, one of them is. 
The point labeled $r$ is the closest marked point to $p$ (in that direction) on the boundary.

A tagged arc $\zeta'$ in the component of $(\S',\M')$ containing $\alpha$ is mapped into $(\S,\M)$ by the natural inclusion unless it has an endpoint at $p_1$ or $p_\alpha$.
If $\zeta'$ has an endpoint at $p_1$ or $p_\alpha$, then the inclusion is extended to $p$ or $r$ or the puncture, tagged plain or notched there.
The details depend also on whether $\beta$ or $\gamma$ is a boundary segment, as illustrated in the second through fifth rows of Figure~\ref{chi mon fig}.
(In the fifth row, if $\beta$ and $\zeta'$ share an endpoint, then $\zeta$ is an arc from $r$ to the puncture, tagged plain.)

There are two tagged arcs in the once-punctured monogon.
The one that is tagged plain maps into $(\S,\M)$ by the natural inclusion and remains tagged plain at the puncture.
The one that is tagged notched maps to an arc from $r$ to the puncture, still tagged notched, as shown in the last row of Figure~\ref{chi mon fig}.

If $\beta$ is a boundary segment, then $z_\gamma=x_\delta$ and $z_\delta=x_{\tilde\gamma}$.
If $\gamma$ is a boundary segment, then $\tilde\gamma=\beta$, so that $z_\beta=x_{\tilde\alpha}$ and $z_\delta=x_{\tilde\gamma}=x_\beta$.

\begin{figure}
\begin{tabular}{ccccc}
\scalebox{0.75}[0.69]{\includegraphics{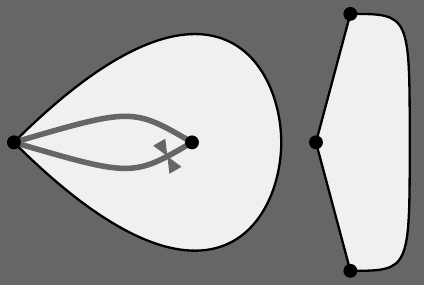}}&\raisebox{30pt}{$=$}&\scalebox{0.75}[0.69]{\includegraphics{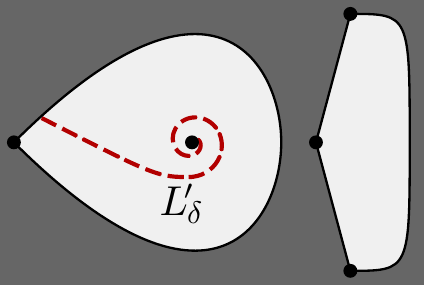}}&\raisebox{30pt}{$+$}&\scalebox{0.75}[0.69]{\includegraphics{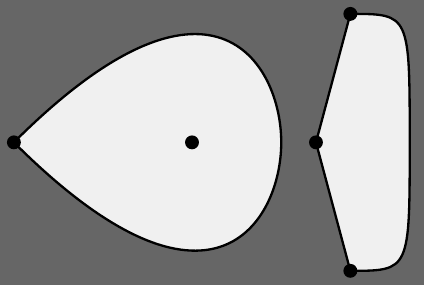}}\\[-3pt]
$\chi\downarrow$&&$\chi\downarrow$&&$\chi\downarrow$\\[1pt]
\scalebox{0.75}[0.69]{\includegraphics{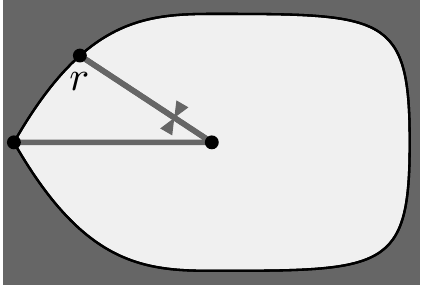}}&\raisebox{30pt}{$=$}&\scalebox{0.75}[0.69]{\includegraphics{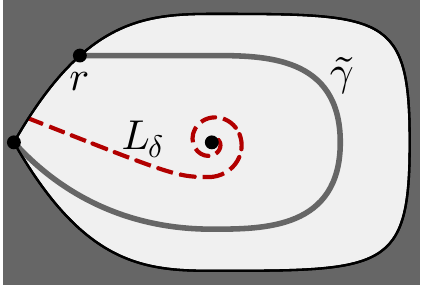}}&\raisebox{30pt}{$+$}&\scalebox{0.75}[0.69]{\includegraphics{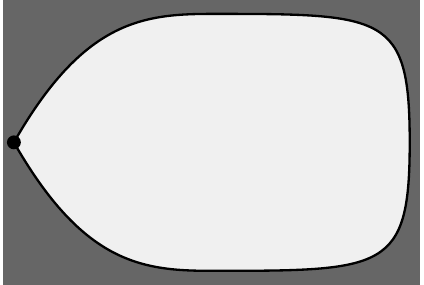}}\\[3pt]\hline\hline\\[-5.3pt]
\scalebox{0.75}[0.69]{\includegraphics{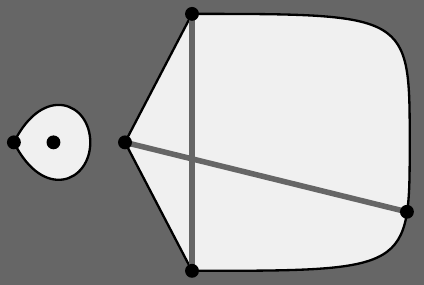}}&\raisebox{30pt}{$=$}&\scalebox{0.75}[0.69]{\includegraphics{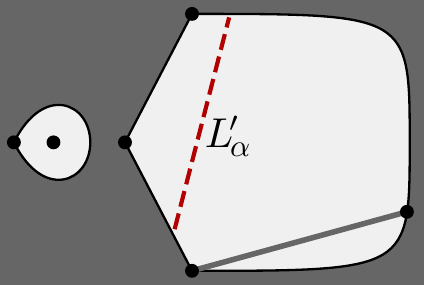}}&\raisebox{30pt}{$+$}&\scalebox{0.75}[0.69]{\includegraphics{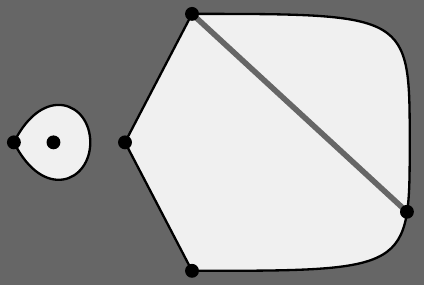}}\\[-3pt]
$\chi\downarrow$&&$\chi\downarrow$&&$\chi\downarrow$\\[1pt]
\scalebox{0.75}[0.69]{\includegraphics{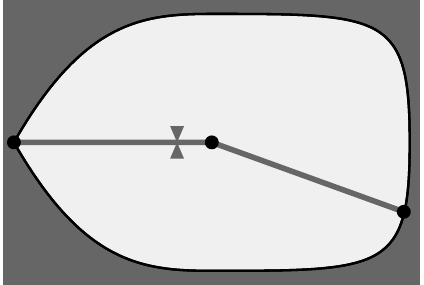}}&\raisebox{30pt}{$=$}&\scalebox{0.75}[0.69]{\includegraphics{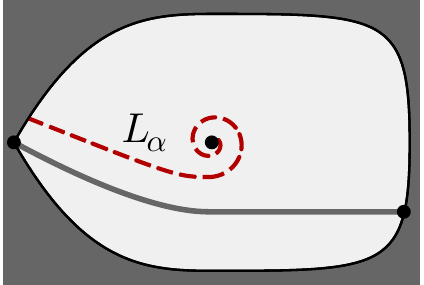}}&\raisebox{30pt}{$+$}&\scalebox{0.75}[0.69]{\includegraphics{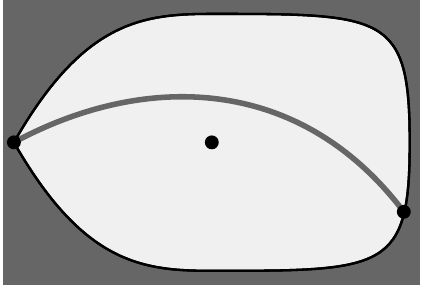}}\\[3pt]\hline\hline\\[-5.3pt]
\scalebox{0.75}[0.69]{\includegraphics{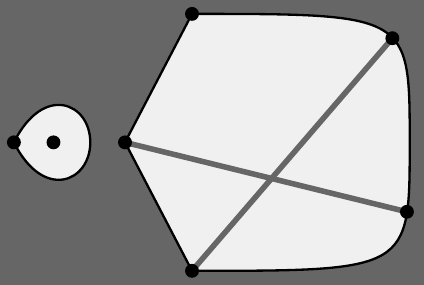}}&\raisebox{30pt}{$=$}&\scalebox{0.75}[0.69]{\includegraphics{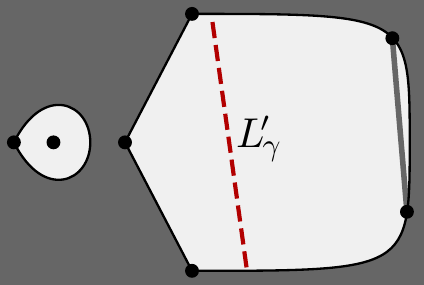}}&\raisebox{30pt}{$+$}&\scalebox{0.75}[0.69]{\includegraphics{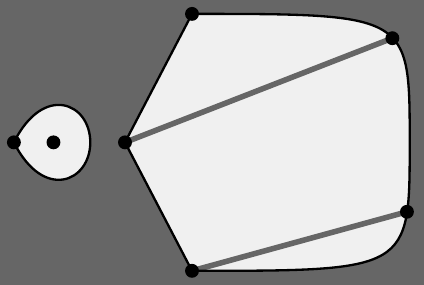}}\\[-3pt]
$\chi\downarrow$&&$\chi\downarrow$&&$\chi\downarrow$\\[1pt]
\scalebox{0.75}[0.69]{\includegraphics{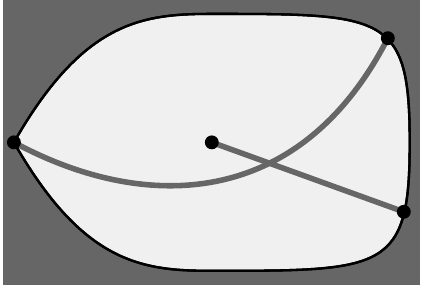}}&\raisebox{30pt}{$=$}&\scalebox{0.75}[0.69]{\includegraphics{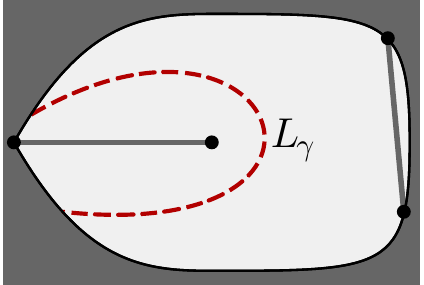}}&\raisebox{30pt}{$+$}&\scalebox{0.75}[0.69]{\includegraphics{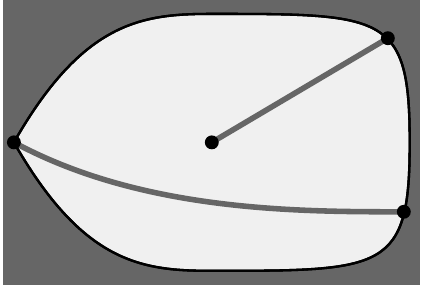}}\\[3pt]\hline\hline\\[-5.3pt]
\scalebox{0.75}[0.69]{\includegraphics{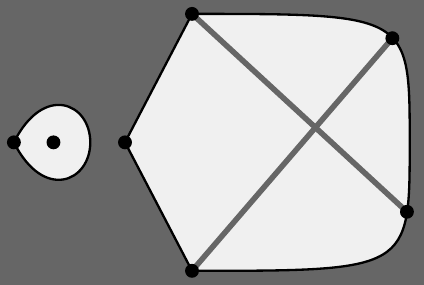}}&\raisebox{30pt}{$=$}&\scalebox{0.75}[0.69]{\includegraphics{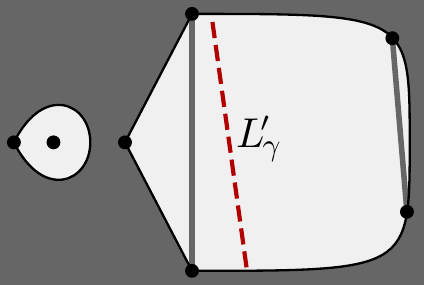}}&\raisebox{30pt}{$+$}&\scalebox{0.75}[0.69]{\includegraphics{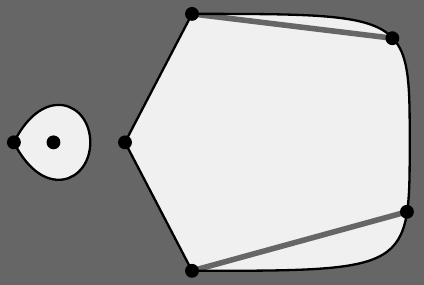}}\\[-3pt]
$\chi\downarrow$&&$\chi\downarrow$&&$\chi\downarrow$\\[1pt]
\scalebox{0.75}[0.69]{\includegraphics{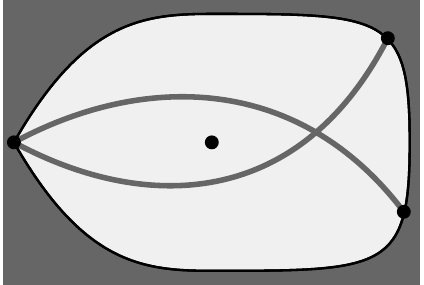}}&\raisebox{30pt}{$=$}&\scalebox{0.75}[0.69]{\includegraphics{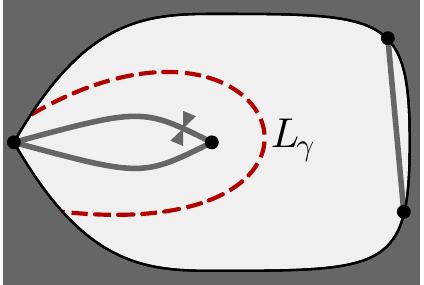}}&\raisebox{30pt}{$+$}&\scalebox{0.75}[0.69]{\includegraphics{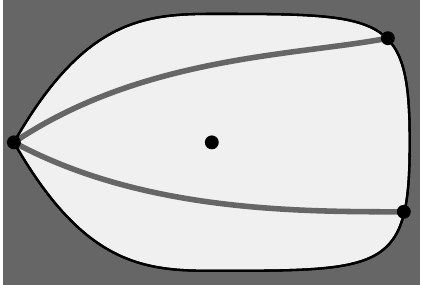}} 
\end{tabular}
\caption{Exchange relations, Case 2}
\label{exch case 2 1}
\end{figure}
\begin{figure}
\begin{tabular}{ccccc}
\scalebox{0.75}[0.69]{\includegraphics{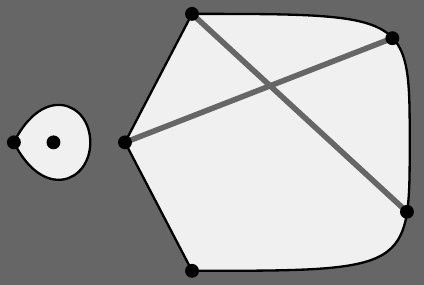}}&\raisebox{30pt}{$=$}&\scalebox{0.75}[0.69]{\includegraphics{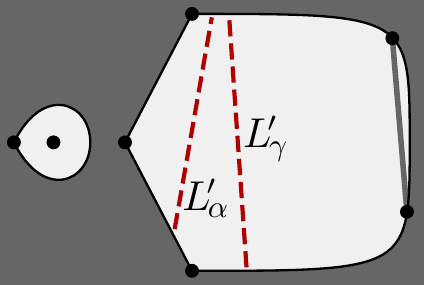}}&\raisebox{30pt}{$+$}&\scalebox{0.75}[0.69]{\includegraphics{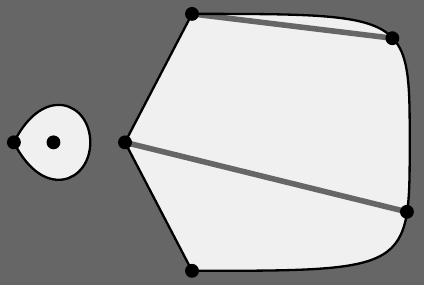}}\\[-3pt]
$\chi\downarrow$&&$\chi\downarrow$&&$\chi\downarrow$\\[1pt]
\scalebox{0.75}[0.69]{\includegraphics{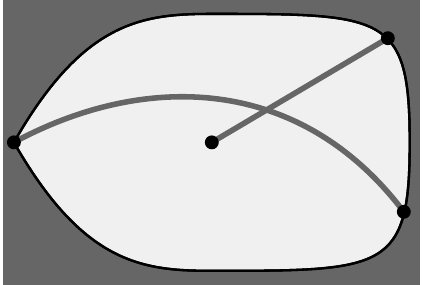}}&\raisebox{30pt}{$=$}&\scalebox{0.75}[0.69]{\includegraphics{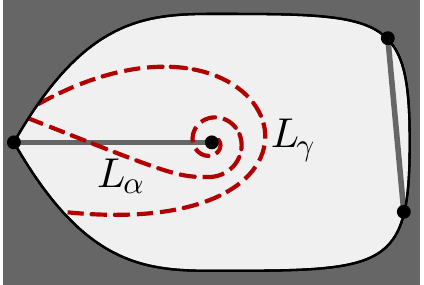}}&\raisebox{30pt}{$+$}&\scalebox{0.75}[0.69]{\includegraphics{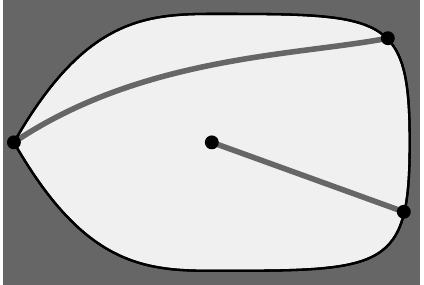}}\\[3pt]\hline\hline\\[-5.3pt]
\scalebox{0.75}[0.69]{\includegraphics{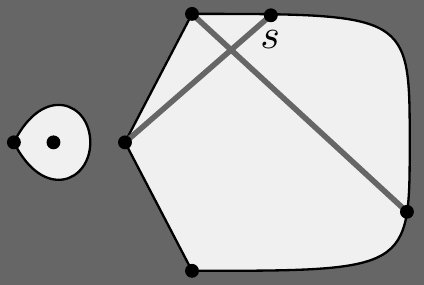}}&\raisebox{30pt}{$=$}&\scalebox{0.75}[0.69]{\includegraphics{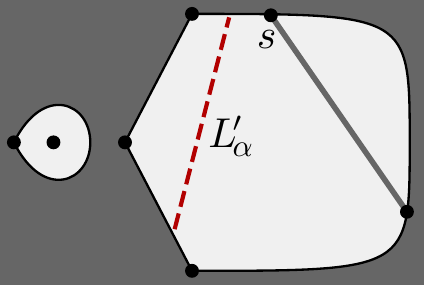}}&\raisebox{30pt}{$+$}&\scalebox{0.75}[0.69]{\includegraphics{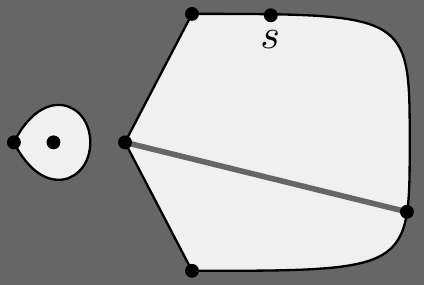}}\\[-3pt]
$\chi\downarrow$&&$\chi\downarrow$&&$\chi\downarrow$\\[1pt]
\scalebox{0.75}[0.69]{\includegraphics{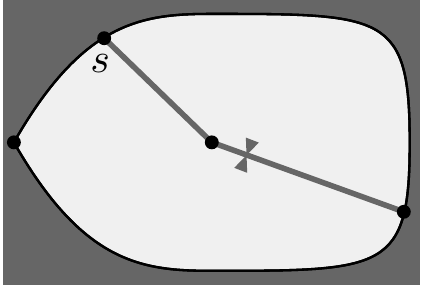}}&\raisebox{30pt}{$=$}&\scalebox{0.75}[0.69]{\includegraphics{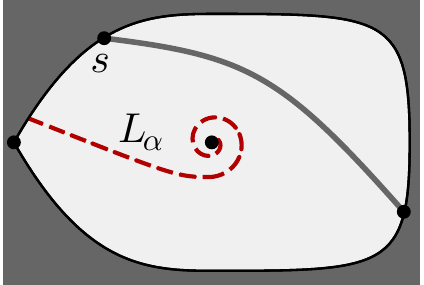}}&\raisebox{30pt}{$+$}&\scalebox{0.75}[0.69]{\includegraphics{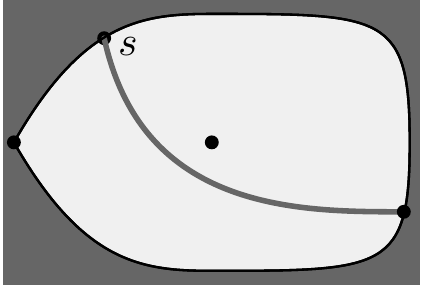}}\\[3pt]\hline\hline\\[-5.3pt]
\scalebox{0.75}[0.69]{\includegraphics{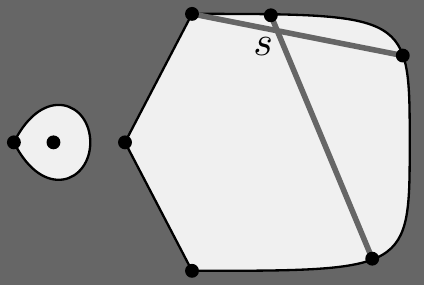}}&\raisebox{30pt}{$=$}&\scalebox{0.75}[0.69]{\includegraphics{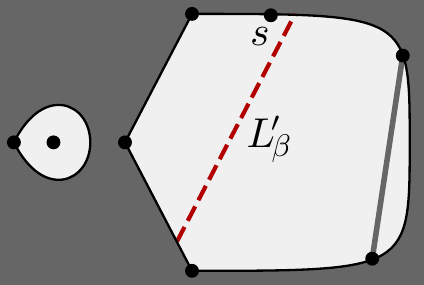}}&\raisebox{30pt}{$+$}&\scalebox{0.75}[0.69]{\includegraphics{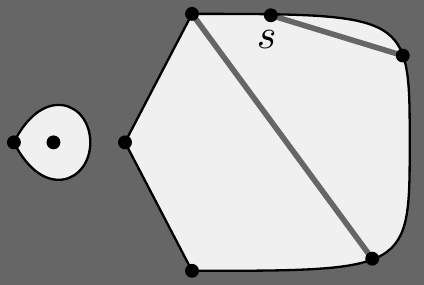}}\\[-3pt]
$\chi\downarrow$&&$\chi\downarrow$&&$\chi\downarrow$\\[1pt]
\scalebox{0.75}[0.69]{\includegraphics{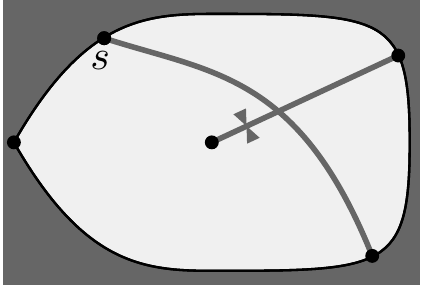}}&\raisebox{30pt}{$=$}&\scalebox{0.75}[0.69]{\includegraphics{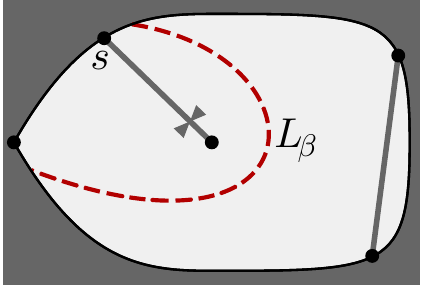}}&\raisebox{30pt}{$+$}&\scalebox{0.75}[0.69]{\includegraphics{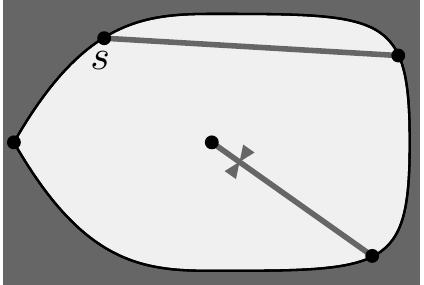}}\\[3pt]\hline\hline\\[-5.3pt]
\scalebox{0.75}[0.69]{\includegraphics{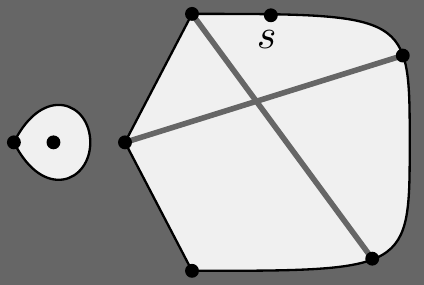}}&\raisebox{30pt}{$=$}&\scalebox{0.75}[0.69]{\includegraphics{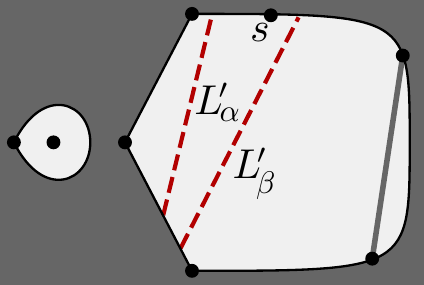}}&\raisebox{30pt}{$+$}&\scalebox{0.75}[0.69]{\includegraphics{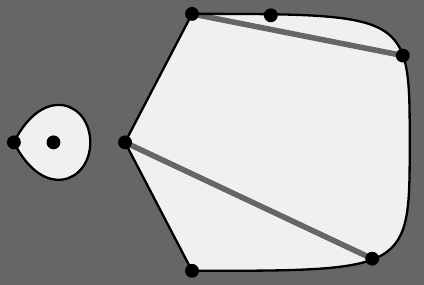}}\\[-3pt]
$\chi\downarrow$&&$\chi\downarrow$&&$\chi\downarrow$\\[1pt]
\scalebox{0.75}[0.69]{\includegraphics{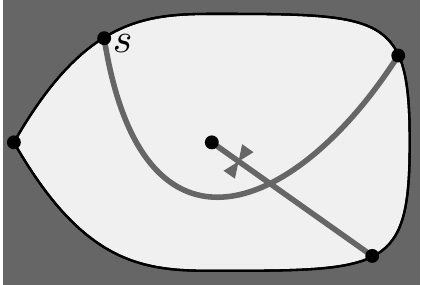}}&\raisebox{30pt}{$=$}&\scalebox{0.75}[0.69]{\includegraphics{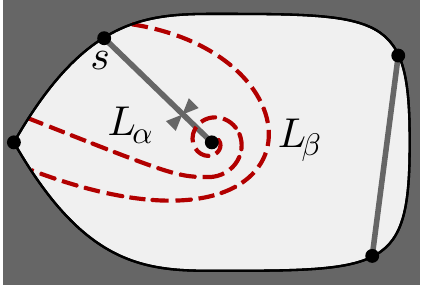}}&\raisebox{30pt}{$+$}&\scalebox{0.75}[0.69]{\includegraphics{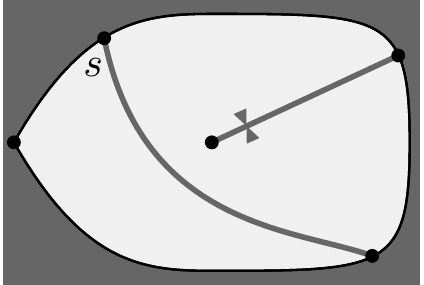}}\end{tabular}
\caption{More exchange relations, Case 2}
\label{exch case 2 2}
\end{figure}
We now show that $\chi$ takes every exchange relation of $\A_\bullet(B(T'))$ to an exchange relation in $\A_\bullet(B(T))$.
The only exchange relation in the once-punctured monogon is shown in the first row of Figure~\ref{exch case 2 1}.
In either case (whether $\beta$ or $\gamma$ is a boundary segment), $\chi$ sends this exchange relation to an exchange relation as illustrated in the second row of Figure~\ref{exch case 2 1}.

Now suppose $\beta$ is a boundary segment and consider an exchange relation in the component containing $\alpha$.
The elementary lamination $L'_\delta$ does not participate in the exchange relation.
There are four cases, based on whether $L'_\alpha$ and (independently) $L'_\gamma$ participate in the exchange relation.
We observe that $L'_\alpha$ participates if and only if an arc incident to $p_\alpha$ is exchanged with an arc incident to $p_1$.
Similarly, $L'_\gamma$ participates if and only if exactly two of $p_1$, $p_\alpha$, and $p_2$ occur as endpoints of the two arcs being exchanged.
The cases where neither $L'_\alpha$ nor $L'_\gamma$ appears is easy. 
The remaining cases are illustrated in the third through eighth rows of Figure~\ref{exch case 2 1} and in the first two rows of Figure~\ref{exch case 2 2}.
\begin{figure}
\begin{tabular}{ccccc}
\scalebox{0.75}[0.69]{\includegraphics{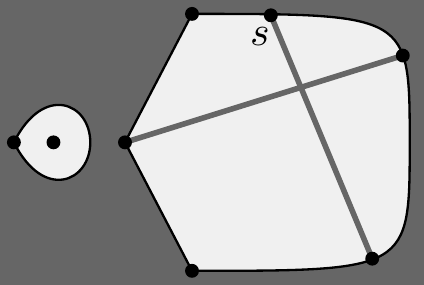}}&\raisebox{30pt}{$=$}&\scalebox{0.75}[0.69]{\includegraphics{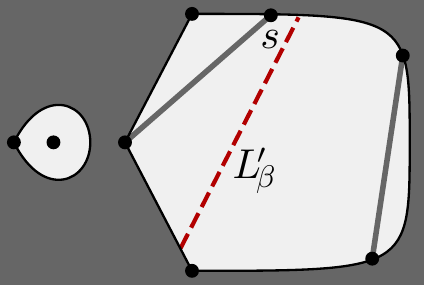}}&\raisebox{30pt}{$+$}&\scalebox{0.75}[0.69]{\includegraphics{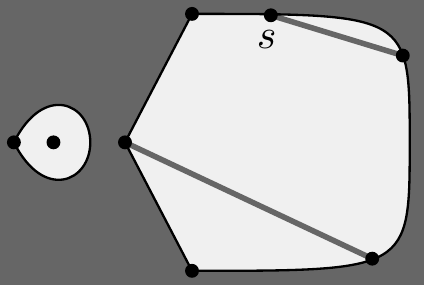}}\\[-2pt]
$\chi\downarrow$&&$\chi\downarrow$&&$\chi\downarrow$\\[1pt]
\scalebox{0.75}[0.69]{\includegraphics{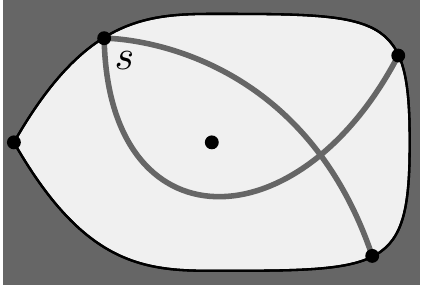}}&\raisebox{30pt}{$=$}&\scalebox{0.75}[0.69]{\includegraphics{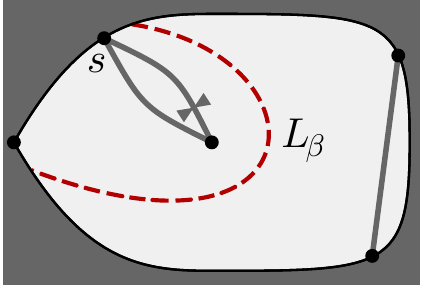}}&\raisebox{30pt}{$+$}&\scalebox{0.75}[0.69]{\includegraphics{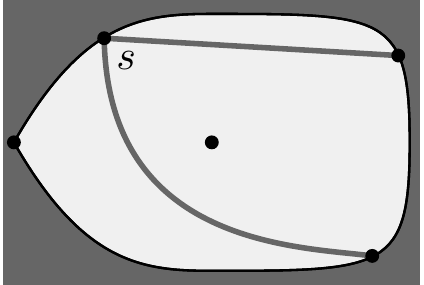}}
\end{tabular}
\caption{More exchange relations, Case 2}
\label{exch case 2 3}
\end{figure}

Finally, suppose $\gamma$ is a boundary segment and consider an exchange relation in the component containing $\alpha$.
Again, $L'_\delta$ does not participate in the exchange relation, and we break into cases based on whether $L'_\alpha$ and $L'_\beta$ participate in the exchange relation.
Again, $L'_\alpha$ participates if and only if an arc incident to $p_\alpha$ is exchanged with an arc incident to $p_1$.
We observe that $L'_\gamma$ participates if and only if exactly two of $s$, $p_1$, and $p_\alpha$, occur as endpoints of the two arcs being exchanged, where as before $s$ is the marked point closest to $p_1$.
Again, the case where neither $L'_\alpha$ nor $L'_\beta$ appears is easy, and the remaining cases are illustrated in the third through eighth rows of Figure~\ref{exch case 2 2} and in Figure~\ref{exch case 2 3}.

The rest of the proof follows the outline given above.
\end{proof}

Finally, we prove a proposition that covers Case 3.
In the proof, we will consider two subcases, depending on which of the two arcs in the digon is cut off from the rest of the disk.
We will refer to the case where $\alpha$ is cut off as Case 3a and refer to the other case as Case 3b, as illustrated in Figure~\ref{case 3a 3b}.
\begin{figure}[ht]
\begin{tabular}{ccc}
\scalebox{0.86}{\includegraphics{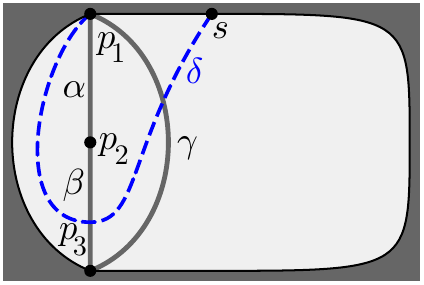}}&\raisebox{29pt}{\LARGE$\xrightarrow{\text{\small resect}}$}&\scalebox{0.86}{\includegraphics{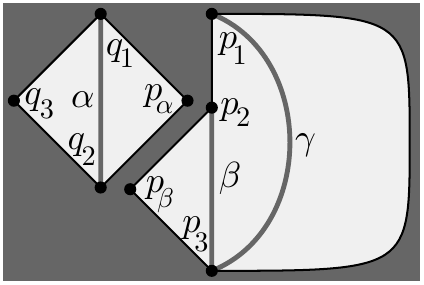}}\\[5pt]
\scalebox{0.86}{\includegraphics{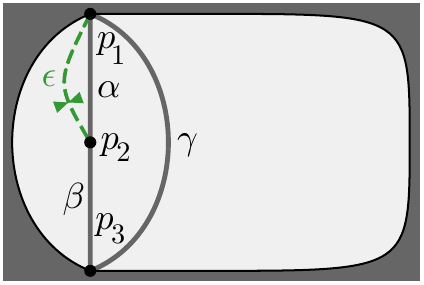}}&\raisebox{29pt}{\LARGE$\xrightarrow{\text{\small resect}}$}&\scalebox{0.86}{\includegraphics{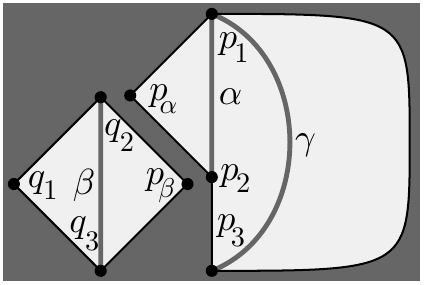}}
\end{tabular}
\caption{Resections, Cases 3a and 3b}
\label{case 3a 3b}
\end{figure}
(Compare Figure~\ref{digon resect config}.)


\begin{prop}\label{case 3 prop}
Suppose $(\S,\M)$ is a once-punctured disk with at least 3 marked points on the boundary with triangulation $T$.
Suppose $(\S',\M')$ is obtained by a resection compatible with $T$, resecting two arcs incident to the puncture in a once-punctured digon of $T$, with one of the edges of the digon being a boundary segment.
Then both parts of Phenomenon~\ref{hom phen} occur.
\end{prop}

\begin{proof}
The map on tagged arcs in Cases 3a and 3b is illustrated in Figure~\ref{case3 chi}.
In Case 3a, the element $z_\alpha$ is $x_\delta$ and $z_\gamma$ is $x_\alpha$.
In Case 3b, $z_\beta$ is $x_\gamma$ and $z_\gamma$ is $x_\epsilon$.
\begin{figure}[ht]
\begin{tabular}{ccc}
\scalebox{0.75}[0.69]{\includegraphics{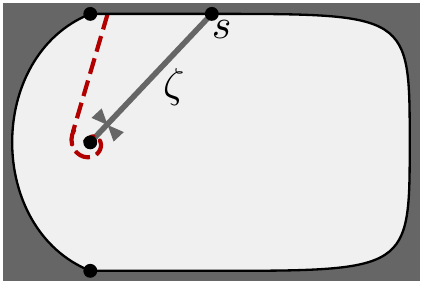}}&\raisebox{29pt}{\LARGE$\xleftarrow{\chi}$}&\scalebox{0.75}[0.69]{\includegraphics{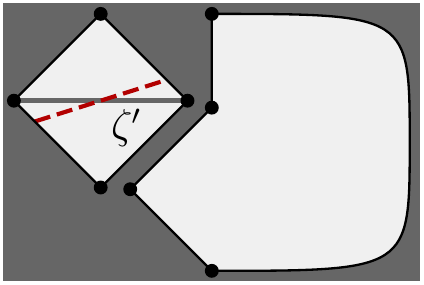}}\\[5pt]
\scalebox{0.75}[0.69]{\includegraphics{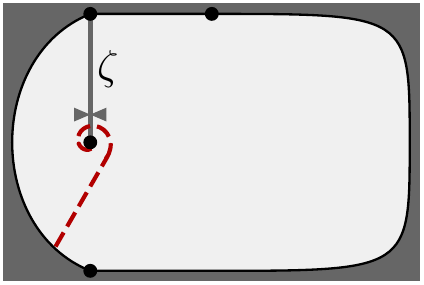}}&\raisebox{29pt}{\LARGE$\xleftarrow{\chi}$}&\scalebox{0.75}[0.69]{\includegraphics{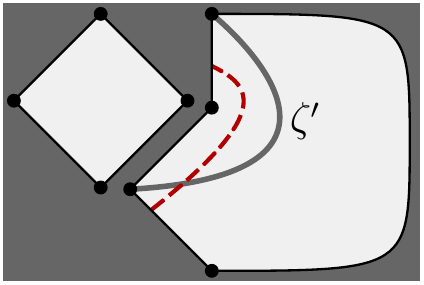}} \\[5pt]
\scalebox{0.75}[0.69]{\includegraphics{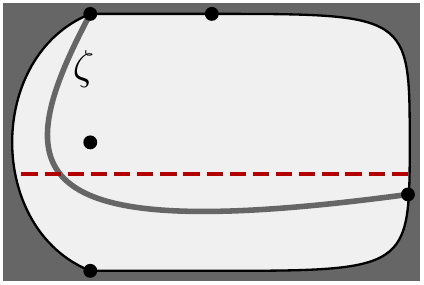}}&\raisebox{29pt}{\LARGE$\xleftarrow{\chi}$}&\scalebox{0.75}[0.69]{\includegraphics{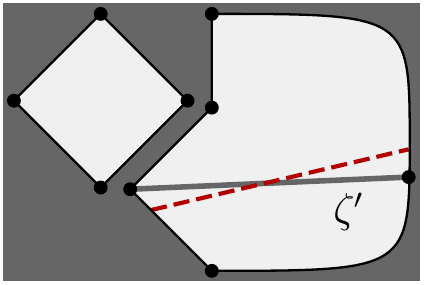}}\\[5pt]
\scalebox{0.75}[0.69]{\includegraphics{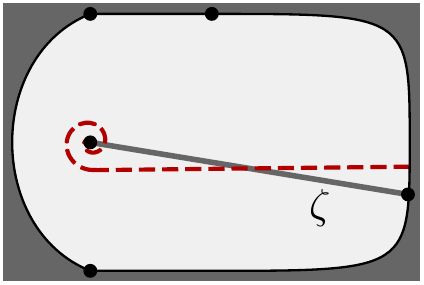}}&\raisebox{29pt}{\LARGE$\xleftarrow{\chi}$}&\scalebox{0.75}[0.69]{\includegraphics{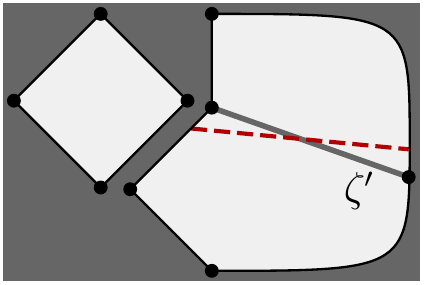}}\\[5pt]
\scalebox{0.75}[0.69]{\includegraphics{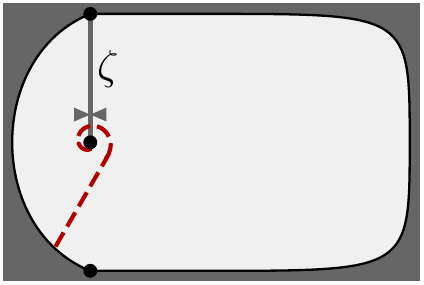}}&\raisebox{29pt}{\LARGE$\xleftarrow{\chi}$}&\scalebox{0.75}[0.69]{\includegraphics{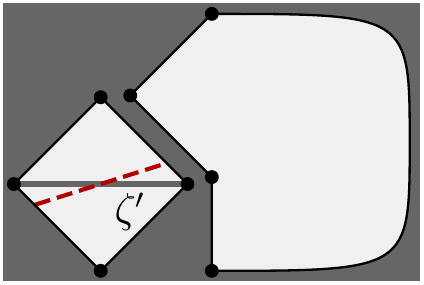}}\\[5pt]
\scalebox{0.75}[0.69]{\includegraphics{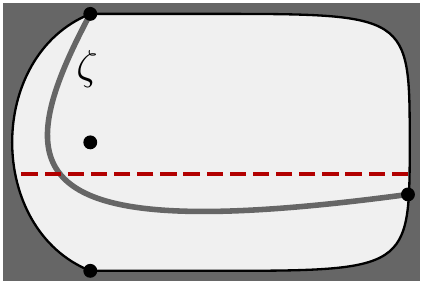}}&\raisebox{29pt}{\LARGE$\xleftarrow{\chi}$}&\scalebox{0.75}[0.69]{\includegraphics{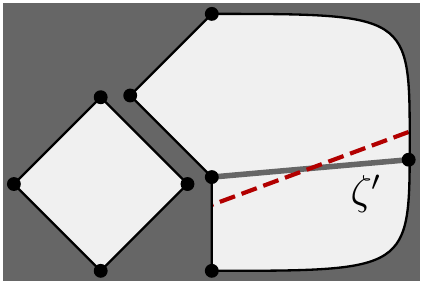}}\\[5pt]
\scalebox{0.75}[0.69]{\includegraphics{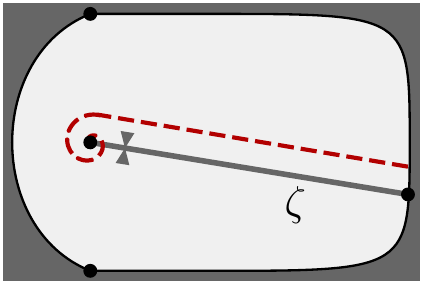}}&\raisebox{29pt}{\LARGE$\xleftarrow{\chi}$}&\scalebox{0.75}[0.69]{\includegraphics{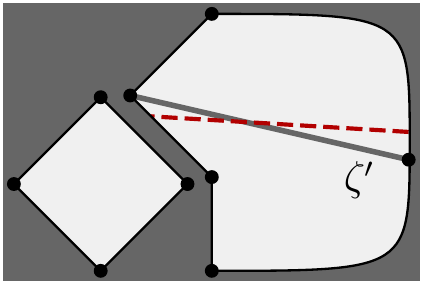}}
\end{tabular}
\caption{Illustrations of the map $\chi$ on tagged arcs, Case 3}
\label{case3 chi}
\end{figure}

\begin{figure}
\begin{tabular}{ccccc}
\scalebox{0.75}[0.69]{\includegraphics{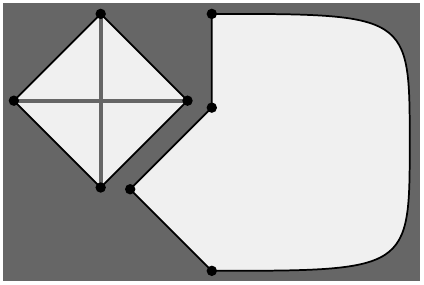}}&\raisebox{30pt}{$=$}&\scalebox{0.75}[0.69]{\includegraphics{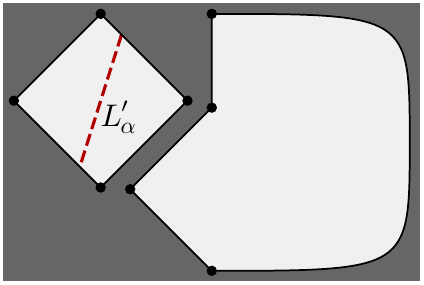}}&\raisebox{30pt}{$+$}&\scalebox{0.75}[0.69]{\includegraphics{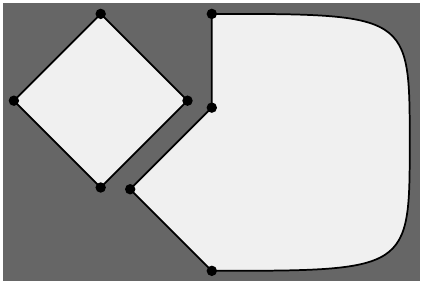}}\\[-3pt]
$\chi\downarrow$&&$\chi\downarrow$&&$\chi\downarrow$\\[1pt]
\scalebox{0.75}[0.69]{\includegraphics{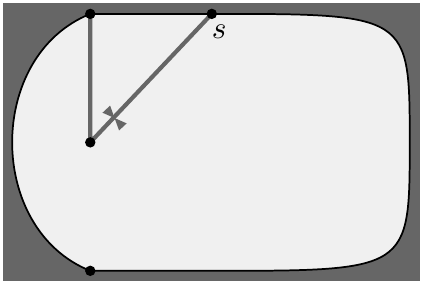}}&\raisebox{30pt}{$=$}&\scalebox{0.75}[0.69]{\includegraphics{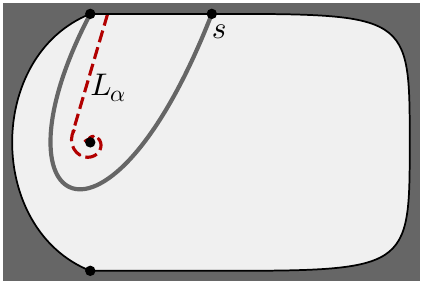}}&\raisebox{30pt}{$+$}&\scalebox{0.75}[0.69]{\includegraphics{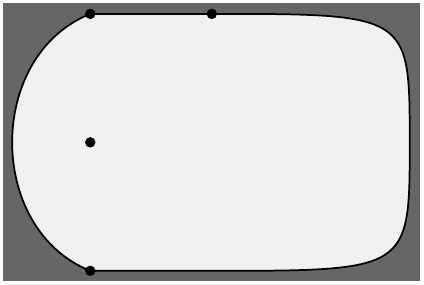}}\\[3pt]\hline\hline\\[-5.3pt]
\scalebox{0.75}[0.69]{\includegraphics{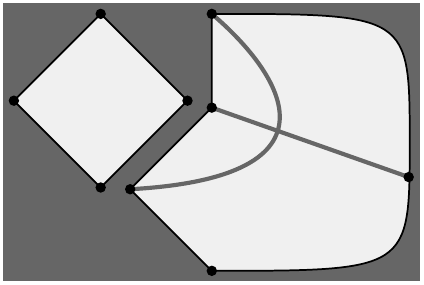}}&\raisebox{30pt}{$=$}&\scalebox{0.75}[0.69]{\includegraphics{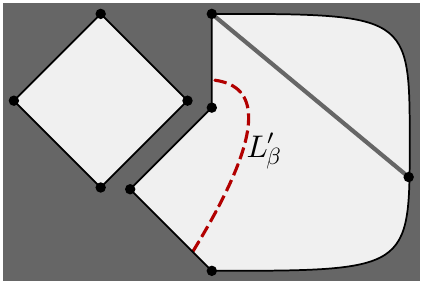}}&\raisebox{30pt}{$+$}&\scalebox{0.75}[0.69]{\includegraphics{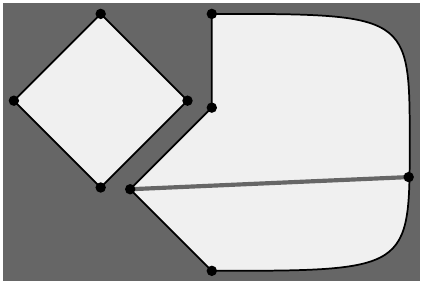}}\\[-3pt]
$\chi\downarrow$&&$\chi\downarrow$&&$\chi\downarrow$\\[1pt]
\scalebox{0.75}[0.69]{\includegraphics{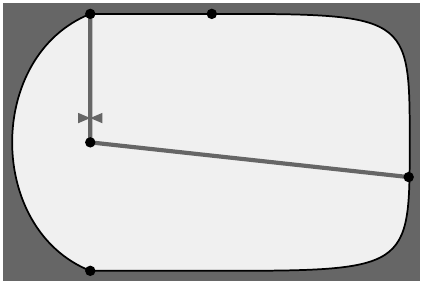}}&\raisebox{30pt}{$=$}&\scalebox{0.75}[0.69]{\includegraphics{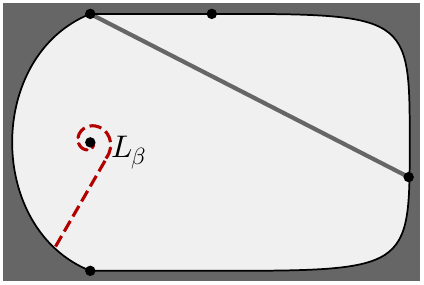}}&\raisebox{30pt}{$+$}&\scalebox{0.75}[0.69]{\includegraphics{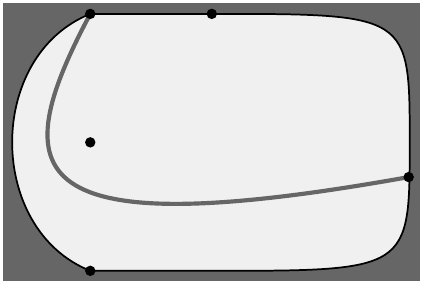}}\\[3pt]\hline\hline\\[-5.3pt]
\scalebox{0.75}[0.69]{\includegraphics{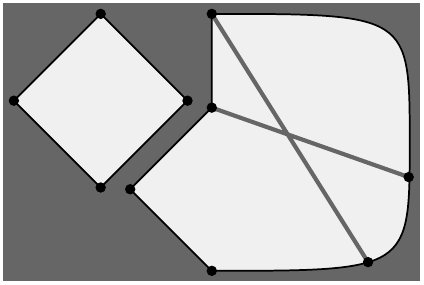}}&\raisebox{30pt}{$=$}&\scalebox{0.75}[0.69]{\includegraphics{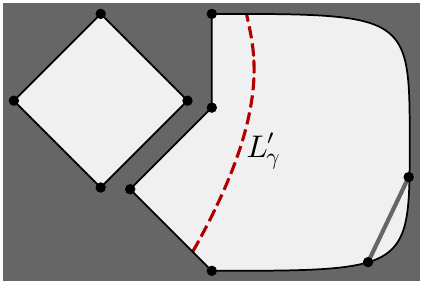}}&\raisebox{30pt}{$+$}&\scalebox{0.75}[0.69]{\includegraphics{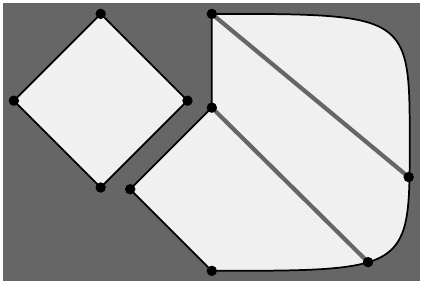}}\\[-3pt]
$\chi\downarrow$&&$\chi\downarrow$&&$\chi\downarrow$\\[1pt]
\scalebox{0.75}[0.69]{\includegraphics{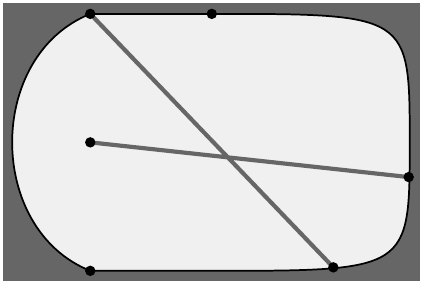}}&\raisebox{30pt}{$=$}&\scalebox{0.75}[0.69]{\includegraphics{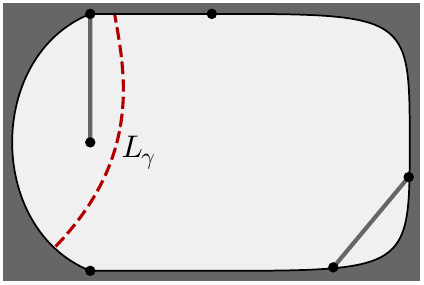}}&\raisebox{30pt}{$+$}&\scalebox{0.75}[0.69]{\includegraphics{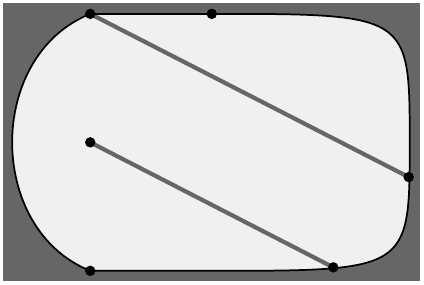}}\\[3pt]\hline\hline\\[-5.3pt]
\scalebox{0.75}[0.69]{\includegraphics{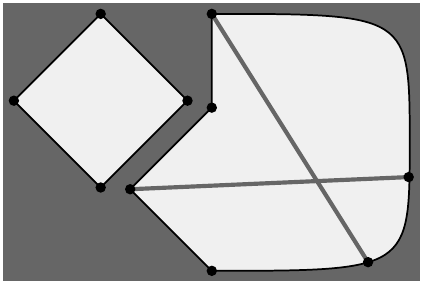}}&\raisebox{30pt}{$=$}&\scalebox{0.75}[0.69]{\includegraphics{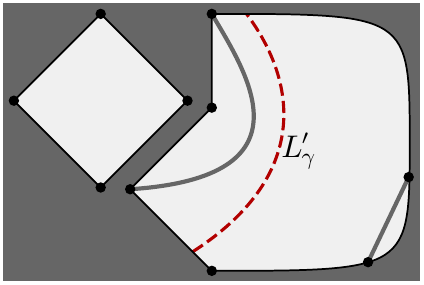}}&\raisebox{30pt}{$+$}&\scalebox{0.75}[0.69]{\includegraphics{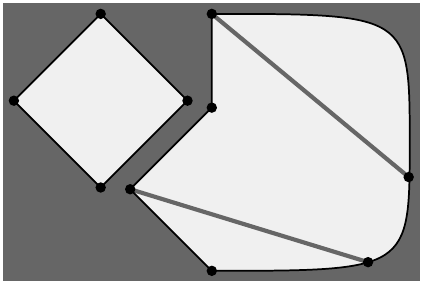}}\\[-3pt]
$\chi\downarrow$&&$\chi\downarrow$&&$\chi\downarrow$\\[1pt]
\scalebox{0.75}[0.69]{\includegraphics{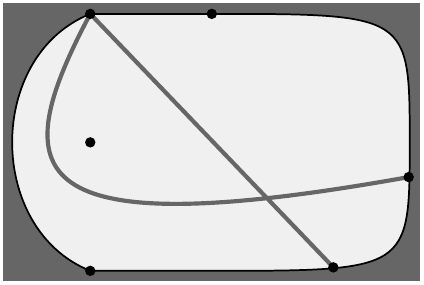}}&\raisebox{30pt}{$=$}&\scalebox{0.75}[0.69]{\includegraphics{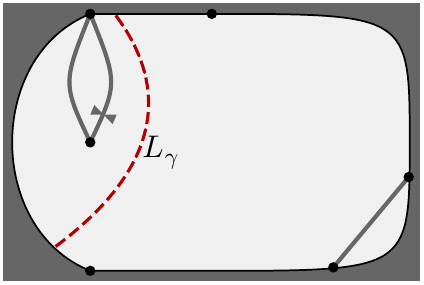}}&\raisebox{30pt}{$+$}&\scalebox{0.75}[0.69]{\includegraphics{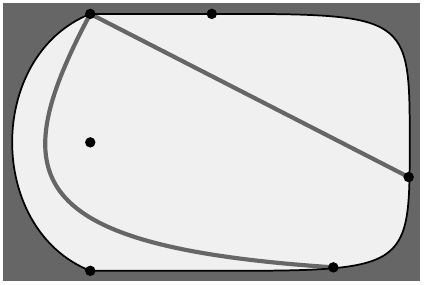}}
\end{tabular}
\caption{Exchange relations, Case 3}
\label{exch3}
\end{figure}
In Case 3a, $\chi$ takes the unique exchange relation involving arcs in the quadrilateral component to an exchange relation in $\A_\bullet(B(T))$, as illustrated in the first two lines of Figure~\ref{exch3}.
An exchange relation in $\A_\bullet(B(T'))$ involves $L'_\beta$ if and only if it exchanges an arc with endpoint $p_2$ and an arc with endpoint $p_\beta$.
An exchange relation involves $L'_\gamma$ if and only if exactly two of $p_1$, $p_2$, and $p_\beta$ occur as endpoints of the two arcs being exchanged.
In each case, $\chi$ maps the exchange relation to an exchange relation, as illustrated in the last six lines of Figure~\ref{exch3} and the first two lines of Figure~\ref{exch3 2}.
\begin{figure}
\begin{tabular}{ccccc}
\scalebox{0.75}[0.69]{\includegraphics{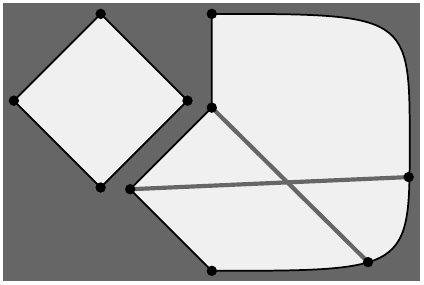}}&\raisebox{30pt}{$=$}&\scalebox{0.75}[0.69]{\includegraphics{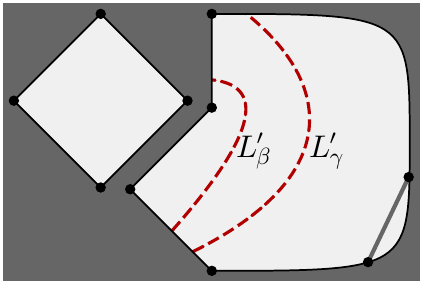}}&\raisebox{30pt}{$+$}&\scalebox{0.75}[0.69]{\includegraphics{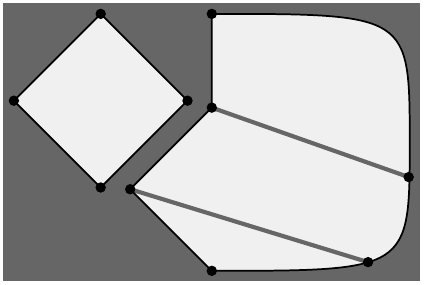}}\\[-3pt]
$\chi\downarrow$&&$\chi\downarrow$&&$\chi\downarrow$\\[1pt]
\scalebox{0.75}[0.69]{\includegraphics{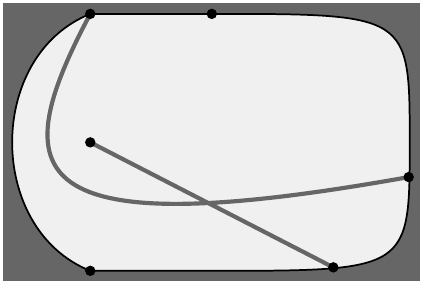}}&\raisebox{30pt}{$=$}&\scalebox{0.75}[0.69]{\includegraphics{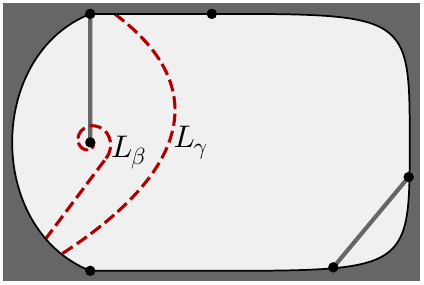}}&\raisebox{30pt}{$+$}&\scalebox{0.75}[0.69]{\includegraphics{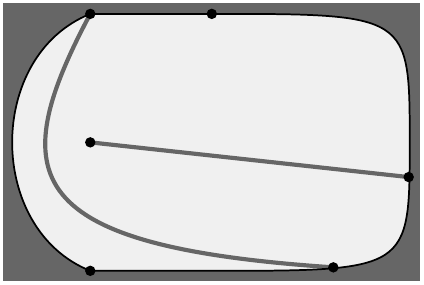}}\\[3pt]\hline\hline\\[-5.3pt]
\scalebox{0.75}[0.69]{\includegraphics{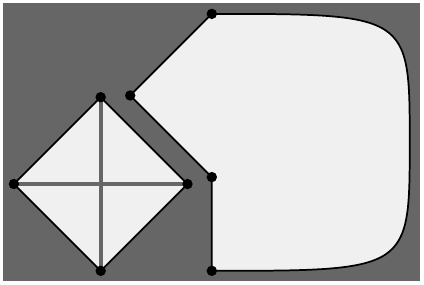}}&\raisebox{30pt}{$=$}&\scalebox{0.75}[0.69]{\includegraphics{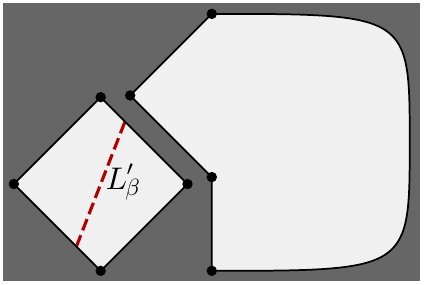}}&\raisebox{30pt}{$+$}&\scalebox{0.75}[0.69]{\includegraphics{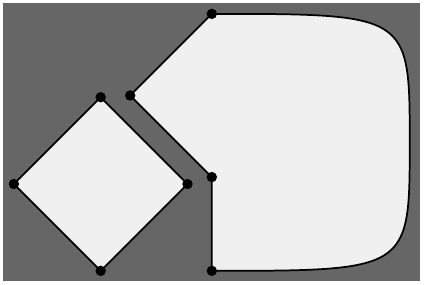}}\\[-3pt]
$\chi\downarrow$&&$\chi\downarrow$&&$\chi\downarrow$\\[1pt]
\scalebox{0.75}[0.69]{\includegraphics{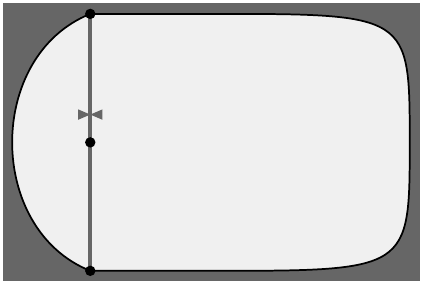}}&\raisebox{30pt}{$=$}&\scalebox{0.75}[0.69]{\includegraphics{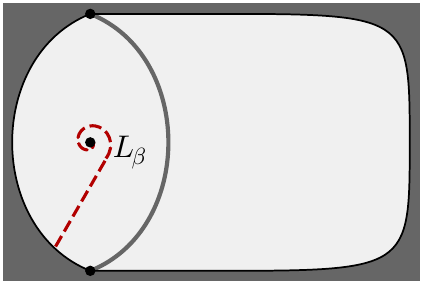}}&\raisebox{30pt}{$+$}&\scalebox{0.75}[0.69]{\includegraphics{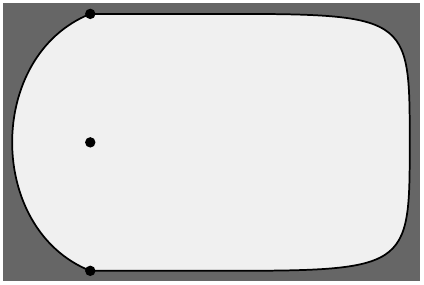}} \\[3pt]\hline\hline\\[-5pt]
\scalebox{0.75}[0.69]{\includegraphics{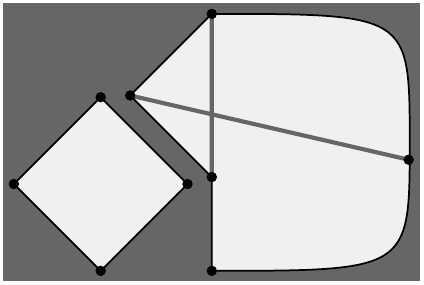}}&\raisebox{30pt}{$=$}&\scalebox{0.75}[0.69]{\includegraphics{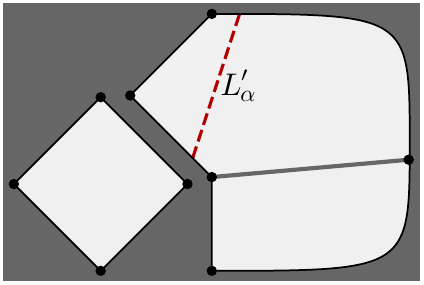}}&\raisebox{30pt}{$+$}&\scalebox{0.75}[0.69]{\includegraphics{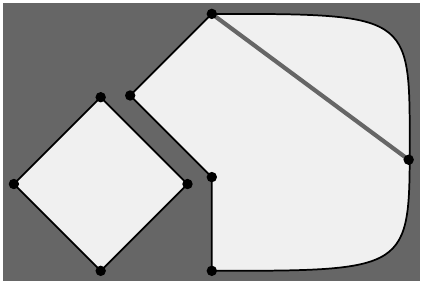}}\\[-3pt]
$\chi\downarrow$&&$\chi\downarrow$&&$\chi\downarrow$\\[1pt]
\scalebox{0.75}[0.69]{\includegraphics{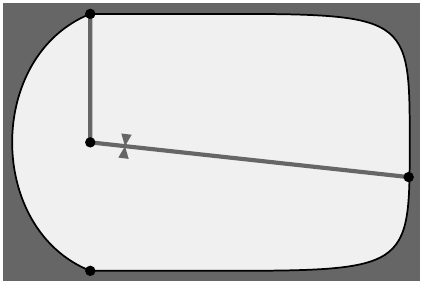}}&\raisebox{30pt}{$=$}&\scalebox{0.75}[0.69]{\includegraphics{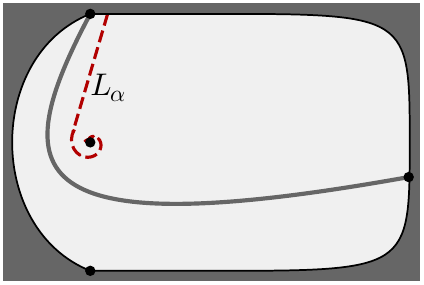}}&\raisebox{30pt}{$+$}&\scalebox{0.75}[0.69]{\includegraphics{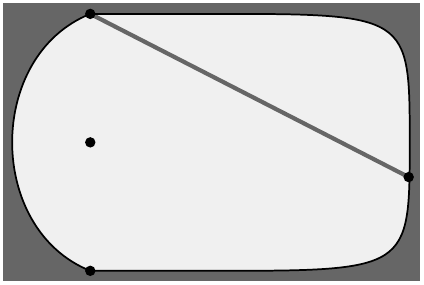}}\\[3pt]\hline\hline\\[-5.3pt]
\scalebox{0.75}[0.69]{\includegraphics{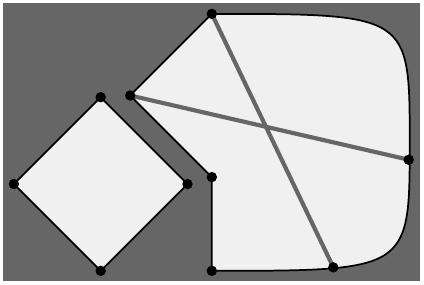}}&\raisebox{30pt}{$=$}&\scalebox{0.75}[0.69]{\includegraphics{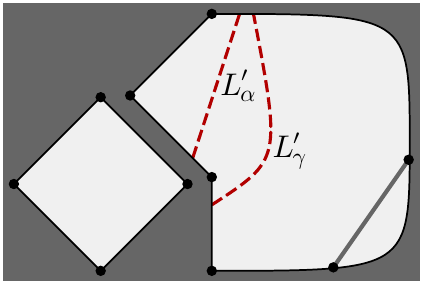}}&\raisebox{30pt}{$+$}&\scalebox{0.75}[0.69]{\includegraphics{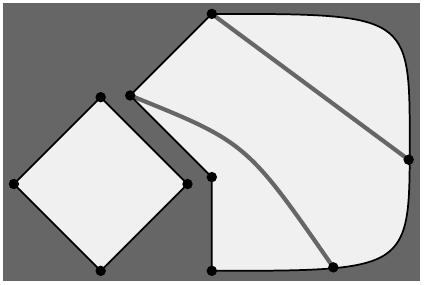}}\\[-3pt]
$\chi\downarrow$&&$\chi\downarrow$&&$\chi\downarrow$\\[1pt]
\scalebox{0.75}[0.69]{\includegraphics{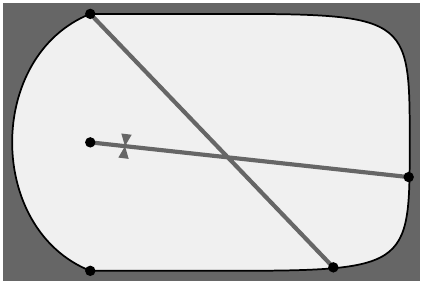}}&\raisebox{30pt}{$=$}&\scalebox{0.75}[0.69]{\includegraphics{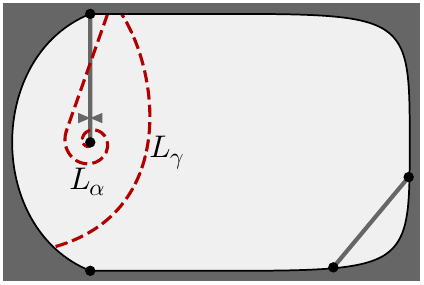}}&\raisebox{30pt}{$+$}&\scalebox{0.75}[0.69]{\includegraphics{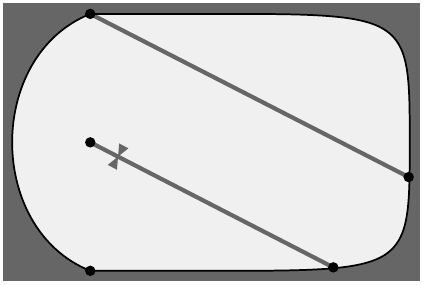}}
\end{tabular}
\caption{More exchange relations, Case 3}
\label{exch3 2}
\end{figure}
\begin{figure}
\begin{tabular}{ccccc}
\scalebox{0.75}[0.69]{\includegraphics{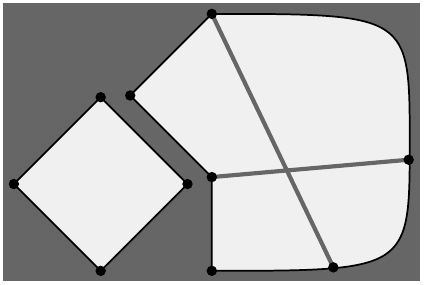}}&\raisebox{30pt}{$=$}&\scalebox{0.75}[0.69]{\includegraphics{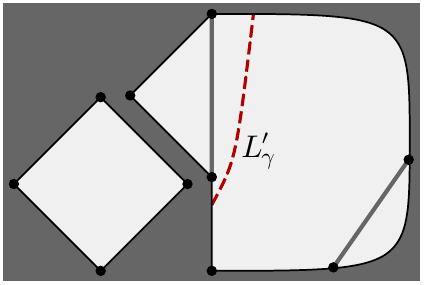}}&\raisebox{30pt}{$+$}&\scalebox{0.75}[0.69]{\includegraphics{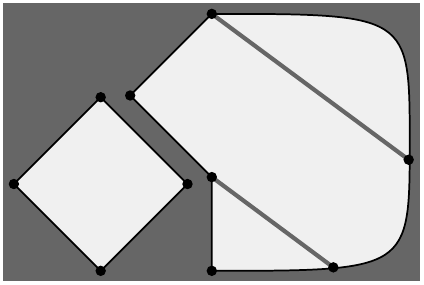}}\\[-3pt]
$\chi\downarrow$&&$\chi\downarrow$&&$\chi\downarrow$\\[1pt]
\scalebox{0.75}[0.69]{\includegraphics{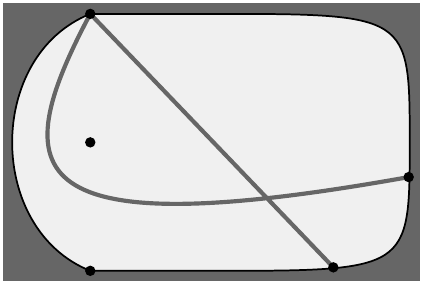}}&\raisebox{30pt}{$=$}&\scalebox{0.75}[0.69]{\includegraphics{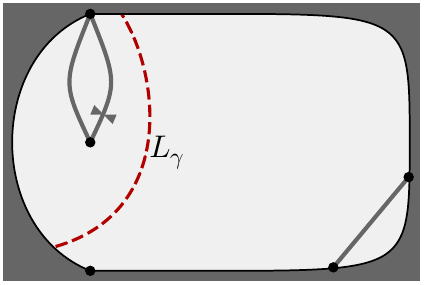}}&\raisebox{30pt}{$+$}&\scalebox{0.75}[0.69]{\includegraphics{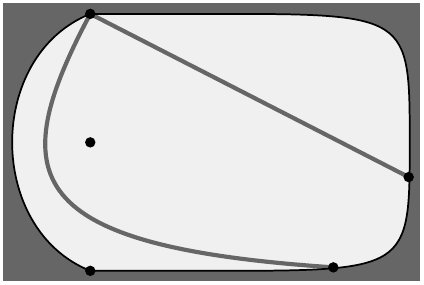}}\\[3pt]\hline\hline\\[-5.3pt]
\scalebox{0.75}[0.69]{\includegraphics{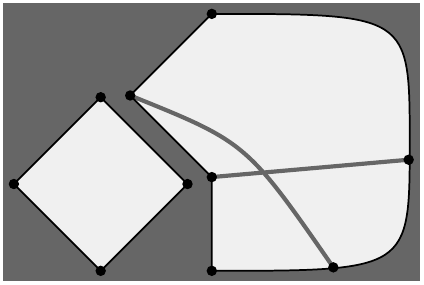}}&\raisebox{30pt}{$=$}&\scalebox{0.75}[0.69]{\includegraphics{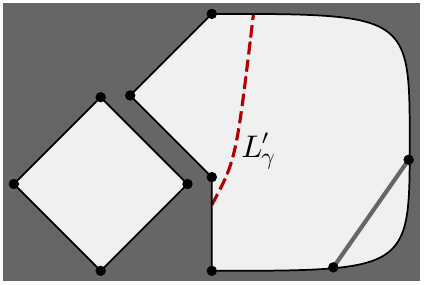}}&\raisebox{30pt}{$+$}&\scalebox{0.75}[0.69]{\includegraphics{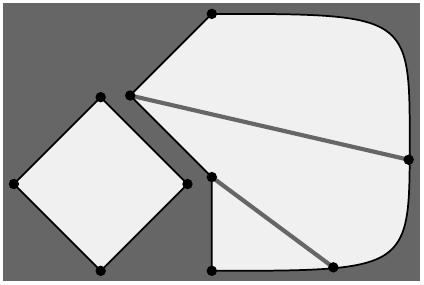}}\\[-3pt]
$\chi\downarrow$&&$\chi\downarrow$&&$\chi\downarrow$\\[1pt]
\scalebox{0.75}[0.69]{\includegraphics{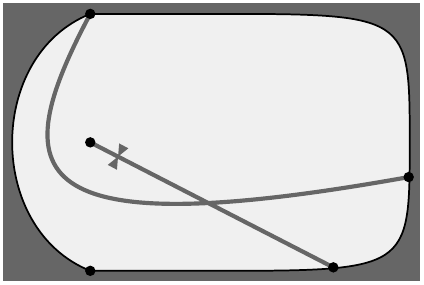}}&\raisebox{30pt}{$=$}&\scalebox{0.75}[0.69]{\includegraphics{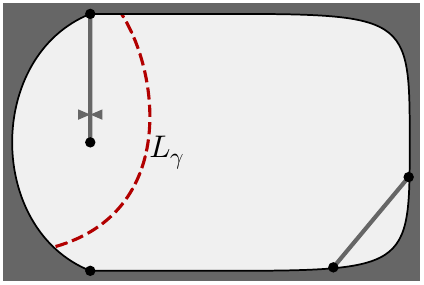}}&\raisebox{30pt}{$+$}&\scalebox{0.75}[0.69]{\includegraphics{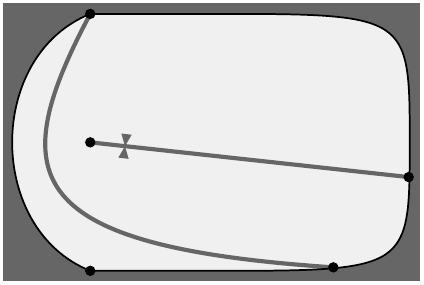}}
\end{tabular}
\caption{More exchange relations, Case 3}
\label{exch3 3}
\end{figure}
Similarly, in Case 3b, $\chi$ maps the unique exchange relation in the quadrilateral to an exchange relation, as illustrated in the last third and fourth lines of Figure~\ref{exch3 2}.
An exchange relation in $\A_\bullet(B(T'))$ involves $L'_\alpha$ if and only if it exchanges an arc with endpoint $p_1$ and an arc with endpoint $p_\alpha$.
An exchange relation involves $L'_\gamma$ if and only if exactly two of $p_1$, $p_\alpha$, and $p_2$ occur as endpoints of the two arcs being exchanged.
Again, $\chi$ takes each exchange relation to an exchange relation, as illustrated in the last four lines of Figure~\ref{exch3 2} and in Figure~\ref{exch3 3}.

Again, the proof concludes as outlined above.
\end{proof}

This completes the proof of Theorem~\ref{hom phen surf thm}.
Similar arguments should work for the surfaces of affine type, but will be even more complicated.
In these cases, $\chi$ maps some tagged arcs to closed allowable curves, so the proof will need more general skein relations \cite{MSW2,MW}, rather than only exchange relations.



\subsection*{Acknowledgments}
The author thanks Man Wai ``Mandy'' Cheung, Salvatore Stella, and Shira Viel for helpful conversations. 

%

\end{document}